\documentclass[12pt]{amsart}
\usepackage{comment}
\usepackage{mathrsfs}
\usepackage[margin=1in]{geometry}
\usepackage{subcaption}
\usepackage{hyperref}
\hypersetup{
    colorlinks=true,
    linkcolor=blue,
    filecolor=blue,      
    urlcolor=blue,
    citecolor=blue,
    pdftitle={Overleaf Example},
    pdfpagemode=FullScreen,
  }
  
\usepackage{amsmath}
\usepackage{amsfonts}
\usepackage{amssymb}
\usepackage{graphicx}
\usepackage{tikz-cd}
\usepackage{amsthm}
\usepackage[all]{xy}
\usepackage{adjustbox}
\usepackage{mathtools}

\newtheorem{counter}{counter}[section]
\newtheorem{theorem}[counter]{Theorem}
\newtheorem{algorithm}[counter]{Algorithm}
\newtheorem{corollary}[counter]{Corollary}
\newtheorem{lemma}[counter]{Lemma}
\newtheorem{proposition}[counter]{Proposition}

\theoremstyle{definition}
\newtheorem{definition}[counter]{Definition}
\newtheorem{notation}[counter]{Notation}
\newtheorem{remark}[counter]{Remark}

\newtheorem{example}[counter]{Example}

\title{The Cohomology of Spherical vector bundles on K3 surfaces}
\author{Yeqin Liu}

\begin{document}
\maketitle
\begin{abstract}
We find an algorithm to compute the cohomology groups of spherical vector bundles on complex projective K3 surfaces, in terms of their Mukai vectors. In many good cases, we give significant simplifications of the algorithm. As an application, when the Picard rank is one, we show a numerical condition that is equivalent to weak Brill-Noether for a spherical vector bundle.
\end{abstract}

\tableofcontents

\section{Introduction}\label{section1}

The cohomology groups of a vector bundle are important invariants. Classical Brill-Noether theory, which studies the cohomology of line bundles on curves, has been an important topic in algebraic geometry for a long time \cite{ACGH85}. Much less is known for the cohomology of vector bundles on curves. On higher dimensional varieties, computing the cohomology of a vector bundle is in general a hard problem. There has been some success concerning weak Brill-Noether for some surfaces such as minimal rational surfaces and certain del Pezzo surfaces \cite{CH18, CH20, GH98, LZ19}. On K3 surfaces, \cite{Ley12, Ley06} studied the Brill-Noether theory, and recently there has also been progress on the cohomology for general sheaves \cite{CNY21}.
 
For a moduli space of sheaves, when weak Brill-Noether fails, less is known about how to compute the exact dimensions of the cohomology groups of a general vector bundle. On a K3 surface, we consider this question for spherical vector bundles. We find an algorithm to compute the cohomology of any stable spherical vector bundle in terms of its Mukai vector. In particular, the dimension of any linear system can be computed in principle. As an interesting corollary, we see that on Picard rank one K3 surfaces with a fixed degree, the cohomology groups of stable spherical vector bundles depend only on their Mukai vectors, but not on the underlying K3 surfaces.

When the Picard rank is one, we give a numerical condition that is equivalent to weak Brill-Noether for a spherical vector bundle. This is a generalization of the results in \cite{CNY21}. We also show a bound for the cohomology of a spherical vector bundle when the Picard rank is one. When there is no restriction on the Picard group, we prove an asymptotic estimate on the cohomology, in the sense that the Mukai vector is obtained by many spherical reflections from two other spherical Mukai vectors.

Our algorithm in fact computes the cohomology of any rigid vector bundle with a known Harder-Narasimhan filtration. It is also very likely that our algorithm can be generalized to compute the generic cohomology of moduli of sheaves on K3 surfaces for any Mukai vector.

\subsection{Main Theorem}

Let $(X,H)$ be any polarized K3 surface over $\mathbb{C}$. Let $v=(r,D,a)$ be a Mukai vector with $v^{2}=-2$, $r>0$ and $D\cdot H>0$. Let $E\in M_{H}(v)$ be the unique $H$-Gieseker stable vector bundle. Our main theorem is an algorithm that computes the cohomology groups of $E$.

\begin{theorem}\label{main}
Algorithm \ref{global reduction} computes the cohomology groups of $E$.
\end{theorem}

In particular, note that since a line bundle is stable and spherical, Theorem \ref{main} can be applied.

\begin{corollary}
Algorithm \ref{global reduction} computes the dimension of any linear system on $X$. 
\end{corollary}

In Algorithm \ref{global reduction}, every step is totally determined by the Mukai vector $v$. Since the lattices for all Picard rank one K3 surfaces of a fixed degree are isomorphic, as a corollary we see in the Picard rank one case the cohomology groups are independent of the underlying K3 surfaces.

\begin{corollary}
  Let $v=(r,dH,a)\in \mathbb{Z}\oplus \mathbb{Z}H \oplus \mathbb{Z}$ be a Mukai vector with $v^{2}=-2$, $r>0$ and $d>0$. Then for any two K3 surfaces $X, X'$ with $\mbox{Pic}(X)\cong \mbox{Pic}(X') \cong \mathbb{Z}H$ and $E\in M_{X,H}(v)$, $E'\in M_{X', H'}(v)$, we have
  $$ \mbox{h}^{i}(X, E)=\mbox{h}^{i}(X', E'), \forall i\in \mathbb{Z}.$$
\end{corollary}

In practical terms, we also make simplifications in some good cases. We will define the \emph{height} (Definition \ref{ht}) of an object that measures the complexity for carrying out this algorithm. An important simplification occurs for an object $E$ with height $\leq 2$. Such an object is an extension of $S\otimes \mbox{Hom}(S,E)$ and $T\otimes \mbox{Hom}(E,T)^{*}$ for some factors $S,T$, where the connecting homomorphisms always have maximal rank.

\begin{theorem}[Theorem \ref{global simplification}, Global simplification]\label{global simplification intro} 
Let $E\in M_{H}(v)$ be a height 2 spherical vector bundle and $S, T$ be its factors. Then in the long exact sequence induced by $ \mbox{H}^{0}(-)$ on
  $$0 \longrightarrow S\otimes \mbox{Hom}(S,E) \longrightarrow E \longrightarrow T\otimes \mbox{Hom}(E,T)^{*} \longrightarrow 0, $$
  the connecting homomorphism $ \mbox{H}^{0}(T)\otimes \mbox{Hom}(E,T)^{*} \longrightarrow  \mbox{H}^{1}(S)\otimes \mbox{Hom}(S,E)$ has maximal rank.
\end{theorem}

We will give an explicit formula for the cohomology of $S$ and $T$ in Section \ref{section4}. In particular, the cohomology of $E$ is computed by Proposition \ref{global simplification intro}. The global simplification covers many stable spherical vector bundles $E$ with relatively small Mukai vectors. In fact, Proposition \ref{global simplification intro} is a consequence of a much more general local simplification (Theorem \ref{local simplification}). Among all examples that the author did, most of them can be computed by Theorem \ref{local simplification}. 

In the cases where we do not need the precise dimension of the cohomology, we give some straightforward results. First, we show a numerical condition that is equivalent to weak Brill-Noether for spherical vector bundles on Picard rank one K3 surfaces.

\begin{theorem} [Theorem \ref{weak BN}, weak Brill-Noether]
 Let $(X,H)$ be a polarized K3 surface of Picard rank one. Let $v=(r,dH,a)$ be a Mukai vector with $v^{2}=-2$ and $r, d>0$, $E\in M_{H}(v)$.
Let $y$ be the largest possible value of $\frac{a_{1}d-ad_{1}}{r_{1}d-rd_{1}} $ where $v_{1}=(r_{1}, d_{1}H, a_{1})\neq v$ satisfies 
$$v_{1}^{2}=-2, vv_{1}<0, \frac{a_{1}d-ad_{1}}{r_{1}d-rd_{1}}>0, 0<d_{1}\leq d.$$
Then $\mbox{H}^{1}(E)=0$ if and only if $y<1$.
\end{theorem}

When the Picard rank is one and the degree is at least 4, we also give a bound on the number of global sections.

\begin{theorem}  [Proposition \ref{h0}, Bound of $\mbox{h}^{0}$]
 Let $X$ be a K3 surface with $\mbox{Pic}(X)=\mathbb{Z}H$ and $H^{2}\neq 2$. Then for any $v=(r,dH,a)\in  \mbox{H}^{*}_{alg}(X)$ with $v^{2}=-2$ and $r,d>0$, we have
  $$\mbox{h}^{0}(E)<2(r+a)$$
  for $E\in M_{H}(v)$.
\end{theorem}

When there is no restriction on the Picard group, we can give an aysmptotic result, in the sense that the Mukai vector comes from many spherical transforms. A stable spherical vector bundle can be uniquely labeled as $S_{j}$ for some $j\leq 0$. It can be constructed inductively from two spherical objects $S_{0}, T_{1}$ by
\begin{align*}
  0 \longrightarrow S_{0}\otimes \mbox{Ext}^{1}(T_{1}, S_{0}) \longrightarrow S_{-1}  \longrightarrow T_{1}  \longrightarrow 0,\\ 
  0 \longrightarrow S_{j+1} \longrightarrow S_{j}\otimes \mbox{Hom}(S_{j+1},S_{j})^{*} \longrightarrow S_{j-1} \longrightarrow 0, ~j\leq -1.
  \end{align*}
We let $coev_{j}:  \mbox{H}^{0}(S_{j+1}) \longrightarrow  \mbox{H}^{0}(S_{j})\otimes \mbox{Hom}(S_{j+1},S_{j})^{*}$ be the induced maps on the cohomology.

\begin{theorem}[Theorem \ref{asymptotic}, Asymptotic result] \label{asymptotic}
 Let $X$ be any K3 surface. Then $coev_{j}$ is injective for $j\leq -3$. If $\mbox{Pic}(X)=\mathbb{Z}H$ and $H^{2}\neq 2$, then $coev_{j}$ is injective for $j\leq -1$.
\end{theorem} 

When all $coev_{j}$ are injective, the connecting homomorphism $ \mbox{H}^{0}(T_{1})\otimes \mbox{Hom}(S_{j}, T_{1})^{*} \longrightarrow  \mbox{H}^{1}(S_{0})\otimes \mbox{Hom}(S_{0},S_{j})$ is injective. Therefore the cohomology of $S_{j}$ is computed, provided the cohomology of $S_{0}$ and $T_{1}$ are known. The asymptotic result tells us that this could only fail for the first few values of $j$. As $j$ tends to $-\infty$, this estimate has relatively small error. 

\subsection{Idea of the Algorithm}

We illustrate the idea of the algorithm. For simplicity and without affecting the gist of the central idea, let us assume the K3 surface $(X,H)$ has Picard rank one. Then as we will explain in Section \ref{Bridgeland}, there is a complex manifold $\mbox{Stab}(X)$ parametrizing stability conditions on $\mathcal{D}^{b}(X)$, and there is a slice $\mathcal{L}(X)\subset \mbox{Stab}(X)$ which is an open subset of the upper half plane
$$\mathbb{H}=\{(sH, tH): s\in \mathbb{R}, t\in \mathbb{R}_{>0}\},$$
and the stability conditions in $\mathcal{L}(X)$ can be explicitly written. We consider the vertical line $\mathbf{b}_{\epsilon}=\{(\epsilon, t): t>0\}\subset \mathcal{L}(X)$ for $0< \epsilon \ll 1$. By \cite{Bri08} (Theorem \ref{large volume limit} below), when $t\gg 0$, $M_{\sigma_{(\epsilon,t)}}(v)=M_{H}(v)$. The set of such $\sigma_{(\epsilon, t)}$ is called the \emph{Gieseker chamber} of $v$.

The point $\sigma_{0}=(0,\sqrt{2/H^{2}})$ is not in $\mathcal{L}(X)$ because $Z_{\sigma_{0}}(\mathcal{O}_{X})=0$, hence for any Mukai vector $v$ which is independent from $v(\mathcal{O}_{X})=(1,0,1)$, the numerical wall $W(v,\mathcal{O}_{X})$ defined by $v$ and $\mathcal{O}_{X}$ must pass through $\sigma_{0}$. This wall is called the \emph{Brill-Noether wall} for $v$ \cite{Fey20}. Let $\mathcal{C}$ be the chamber below the Brill-Noether wall.

A wall for $v$ is called \emph{totally semistable} \cite{BM14a}, if for a stability condition $\sigma$, all members of $M_{\sigma}(v)$ are destablized when $\sigma$ passes through the wall. If there is no totally semistable wall of $v\in  \mbox{H}^{*}_{alg}(X)$ between its Gieseker chamber and $\mathcal{C}$, then $v$ has weak Brill-Noether. If there are totally semistable walls between the Gieseker chamber and $\mathcal{C}$, there is a chamber $\mathcal{C}_{0}$ such that for any $\sigma\in \mathcal{C}_{0}$, $\sigma$ is below the Brill-Noether walls of all $\sigma$-Harder-Narasimhan factors of $E$. For $\sigma\in \mathcal{C}_{0}$, the last $\sigma$-Harder-Narasimhan factor is $\mathcal{O}_{X}[1]^{\oplus h}$ for some $h\in \mathbb{Z}_{\geq 0}$, and $\mbox{h}^{1}(E)=h$ (Corollary \ref{output}). The process to get $\sigma$-Harder-Narasimhan filtration for $\sigma\in \mathcal{C}_{0}$ is called the global reduction (Algorithm \ref{global reduction}). Until this point we do not require that $v$ is spherical. We deal with global reduction in Section \ref{section3}.

To compute the $\sigma$-Harder-Narasimhan filtration for $\sigma\in \mathcal{C}_{0}$, we need to keep track of the Harder-Narasimhan filtration whenever it changes. The points on $\mathbf{b}_{\epsilon}$ such that the Harder-Narasimhan filtration changes are called the \emph{Harder-Narasimhan walls}. We want every Harder-Narasimhan wall on $\mathbf{b}_{\epsilon}$ to only involve a rank 2 sublattice of $ \mbox{H}^{*}_{alg}(X)$. This is indeed true for $\epsilon$ sufficiently small. Now if $v$ is spherical, then every time we cross such a wall, there are two stable spherical objects $S, T$ that generate all semistable spherical objects in the associated rank 2 lattice. There is a local reduction process (Algorithm \ref{local reduction}) that computes the Harder-Narasimhan filtration of a rigid semistable object on one side of a wall, provided that we know the Harder-Narasimhan filtration on the other side of that wall, in terms of the order of $S$ and $T$ appearing in the filtrations. Then we may compute the new Harder-Narasimhan filtration of $E$ whenever we encounter a Harder-Narasimhan wall, and the $\sigma$-Harder-Narasimhan filtration for $\sigma\in \mathcal{C}_{0}$ can be eventually computed.

The main idea of local reduction is to find an intermediate filtration that connects the information of the Harder-Narasimhan filtrations on two sides of the wall. This notion is Definition \ref{exhaustive}, it is the core of the whole algorithm. We deal with local reduction in Section \ref{section4} and Section \ref{section5}. In many good cases the local reduction can be significantly simplified, we discuss this in Section \ref{section6}.

\subsection{Acknowledgements}
I want to thank Izzet Coskun for many valuable advice and conversations. I also want to thank Benjamin Gould for many useful discussions.

\section{Preliminaries}\label{section2}

In this section we collect necessary preliminaries about theory of moduli spaces of sheaves and Bridgeland stability conditions on K3 surfaces. Some good references are \cite{HL10, Bri07, Bri08, BM14a, BM14b}.

\subsection{Lattices and the Mukai Pairing}

In this paper, a \emph{K3 surface} is a projective K3 surface $X$ over $\mathbb{C}$ together with a polarization $H$. We let
$$ \mbox{H}^{*}_{alg}(X)= \mbox{H}^{0}(X, \mathbb{Z})\oplus \mbox{Pic}(X) \oplus  \mbox{H}^{4}(X,\mathbb{Z})$$ be the algebraic part of its cohomology ring.

Let $\mathcal{D}^{b}(X)$ be the bounded derived category of $X$. For any object $E\in \mathcal{D}^{b}(X)$, its Chern character is defined by
$$\mbox{ch}(E)=(r(E), c_{1}(E), \frac{c_{1}(E)^{2}}{2}-c_{2}(E) )\in  \mbox{H}^{*}_{alg}(X), $$
and we define its Mukai vector
$$v(E)=\mbox{ch}(E)\cdot \sqrt{\mbox{td}(X)}=(r(E), c_{1}(E), r(E)+\mbox{ch}_{2}(E)). $$
By the Riemann-Roch Theorem, the Euler characteristic of two objects can be computed in terms of their Mukai vectors. Let $v(E)=(r,D,a)$ and $v(F)=(r',D',a')$, then 
$$\chi(E,F)=\sum_{i\in \mathbb{Z}}(-1)^{i}\mbox{ext}^{i}(E,F)=-D\cdot D'+ra'+r'a.$$
The pairing on $ \mbox{H}^{*}_{alg}(X)$ defined by
$$ (r,D,a)\cdot (r',D',a') =D\cdot D' -ra'-r'a$$
is called the Mukai pairing. A vector $v=(r,D,a)\in  \mbox{H}^{*}_{alg}(X)$ is called \emph{positive} \cite{Yos01}, if $v^{2}\geq -2$, and
\begin{itemize}
\item either $r>0$,
\item or $r=0$, $D$ is effective, and $a\neq 0$,
\item or $r=0$, $D=0$, $a>0$.
\end{itemize}
A vector $v$ is called \emph{spherical} if $v^{2}=-2$.

\subsection{Stability of Sheaves}

Throughout this paper, all sheaves are assumed to be coherent and pure dimensional. Two of the classical stability notions for sheaves are slope stability and Gieseker stability. 

\begin{definition}
  We call $\mu_{H}(E)=\frac{c_{1}(E)\cdot H}{r(E)}$ the $H$-slope of $E$. A sheaf $E$ on $X$ is $H$-slope (semi)stable if for any subsheaf $0\neq F\subset E$,
  $$\mu_{H}(F) <(\leq) \mu_{H}(E). $$
\end{definition}

  In general if the underlying variety has dimension $>1$, slope stability does not form good moduli spaces. A better notion is Gieseker stability. Recall that for any sheaf $E$, the $H$-Hilbert polynomial of $E$ is
  $$P_{E}(n)=\chi(E(nH)).$$
Let $a_{E}$ be the leading coefficient of $P_{E}$. The reduced Hilbert polynomial of $E$ is $p_{E}=\frac{P_{E}}{a_{E}} $.

\begin{definition}
  A sheaf $E$ is $H$-(Gieseker) (semi)stable if for any subsheaf $0\neq F \subset E$,
  $$p_{F}(n)<(\leq) p_{E}(n) \mbox{ for } n\gg 0.  $$
\end{definition}

On a K3 surface, Gieseker stability reads explicity in the following form

\begin{definition}
  On a polarized K3 surface $(X,H)$, a sheaf $E$ with Mukai vector $(r,D,a)$ is $H$-semistable if for any subsheaf $0\neq F \subset E$ with Mukai vector $(r',D'a')$,
  $$\frac{D'H}{r'}\leq \frac{DH}{r}, \mbox{ and if }\frac{D'H}{r'}=\frac{DH}{r}   \mbox{ then } \frac{a'}{r'}\leq \frac{a}{r}.  $$
The sheaf $E$ is stable if the equalities above cannot all hold.
\end{definition}

\begin{remark}
  The different notions of stability have the following implications:
  $$\mbox{slope stable} \implies \mbox{Gieseker stable} \implies \mbox{Gieseker semistable} \implies \mbox{slope semistable}.$$
\end{remark}

For a primitive Mukai vector $v$ (not a multiple of another integral vector), there exists a space $M_{H}(v)$ parametrizing $H$-semistable (hence stable by primitivity) sheaves with Mukai vector $v$ \cite{Gie77, Mar77, Mar78}. We collect some facts about these moduli spaces. 

\begin{theorem} [\cite{BM14b, KLS06, O'G99, PR14, Yos01, Yos99}]\label{moduli}
 Let $v\in  \mbox{H}^{*}_{alg}(X)$ be a primitive positive Mukai vector. If $H$ is generic in the sense of \cite{O'G97}, then
  \begin{enumerate}
  \item $M_{H}(v)$ is non-empty.
  \item The dimension of $M_{H}(v)$ is $v^{2}+2$.
  \item $M_{H}(v)$ is a smooth and irreducible, deformation equivalent to Hilbert scheme of $(\frac{v^{2}+2}{2})$ points on a K3 surface.
  \end{enumerate}
\end{theorem}

In this paper we consider the following question. Let $v=(r,D,a)$ be a spherical Mukai vector with $r>0$ with respect to a generic polarization. Since it is spherical, it is primitive. And it is positive since $r>0$. By Theorem \ref{moduli}, $M_{H}(v)=\{*\}$ is a single point, let $E$ be the unique sheaf in this moduli space. It is a vector bundle \cite{HL10}. We would like to compute $ \mbox{H}^{0}(E)$. We will eventually answer this question by an algorithm.

\subsection{Bridgeland Stability Condition on K3 Surfaces}\label{Bridgeland}
The main tool of this paper is Bridgeland stability. We collect facts about stability conditions on K3 surfaces in this subsection, the main references are \cite{Bri07, Bri08}.

\begin{definition}
  A \emph{stability condition} on a triangulated category $\mathcal{D}$ is a pair $(Z,\mathcal{A})$, where $\mathcal{A}$ is the heart of a bounded t-structure on $\mathcal{D}$, and $Z: K(\mathcal{A}) \longrightarrow \mathbb{C}$ is a group homomorphism called the \emph{central charge}, such that
  \begin{itemize}
  \item For any $0\neq E\in \mathcal{A}$, $Z(E)=r(E)\cdot e^{i\pi \phi(E)}$ for some $r(E)> 0$ and $\phi\in (0,1]$. $\phi(E)$ is called the \emph{phase} of $E$.
  \item An object $G\in \mathcal{A}$ is called (semi)stable, if for any subobject $0\neq F\subset G$, we have $\phi(F)<(\leq)\phi(G)$. For any object $E\in \mathcal{A}$, there exists a filtration called the \emph{Harder-Narasimhan filtration}:
    $$0= E_{0}\subset E_{1}\subset \cdots \subset E_{n}=E,$$
    such that $G_{i}=E_{i}/E_{i-1}$ are semistable, and $\phi(G_{i})>\phi(G_{i+1})$ for all $i$.
  \end{itemize}
\end{definition}
A stability condition $\sigma=(Z,\mathcal{A})$ is called \emph{numerical} if the central charge $Z:K(X) \longrightarrow \mathbb{C}$ factors through the numerical group $\mathcal{N}(X)=K(X)/K(X)^{\perp}$. A stability condition $\sigma$ is called \emph{geometric} if all skyscraper sheaves $\mathcal{O}_{x}$ are $\sigma$-stable.

There is an effective way to construct concrete stability conditions, which we recall here. For any $\beta\in \mbox{Pic}(X)_{\mathbb{R}}$ and $\omega\in \mbox{Amp}(X)_{\mathbb{R}}$, we construct a torsion pair $(\mathcal{T}_{\beta, \omega}, \mathcal{F}_{\beta, \omega})$ on the abelian category $\mbox{Coh}(X)$ as follows:
\begin{itemize}
\item $\mathcal{F}_{\beta, \omega}=\{F\in \mbox{Coh}(X): \mbox{for all subsheaves }0\neq F'\subset F, \mu_{\omega}(F')\leq \beta\cdot \omega\},$
\item $\mathcal{T}_{\beta, \omega}=\{T\in \mbox{Coh}(X): T \mbox{ is torsion, or for all quotient }T\twoheadrightarrow T', \mu_{\omega}(T')> \beta\cdot \omega\}$.
\end{itemize}
Let
$$\mathcal{A}_{\beta, \omega}=\{E\in \mathcal{D}^{b}(X): \mathcal{H}^{-1}(E)\in \mathcal{F}_{\beta, \omega}, \mathcal{H}^{0}(E)\in \mathcal{T}_{\beta, \omega}, \mathcal{H}^{i}(E)=0 \mbox{ for all }i\neq -1,0\}$$
be the tilt of $\mbox{Coh}(X)$ with respect to this torsion pair. Let $Z_{\beta, \omega}: \mathcal{N}(X) \longrightarrow \mathbb{C}$ be
$$Z_{\beta, \omega}(r, D, a)=-a-r \frac{\beta^{2}-\omega^{2}}{2}+D\cdot \beta + i \omega \cdot (D-r\beta) .$$
We let
$$\mathcal{K}(X)=\{(\beta, \omega)\in \mbox{Pic}(X)_{\mathbb{C}} : \omega\in \mbox{Amp}(X)_{\mathbb{R}}\},$$
and $\mathcal{L}(X)$ be the locus of $(\beta, \omega)\in \mathcal{K}(X)$ such that $Z_{\beta, \omega}(v)\notin \mathbb{R}_{\leq 0}$ for all spherical Mukai vector $v$ with positive rank. 
Then $(Z_{\beta,\omega}, \mathcal{A}_{\beta, \omega})$ is a geometric stability condition for $(\beta, \omega)\in \mathcal{L}(X)$.

There is a natural complex manifold structure on the set of all stability conditions $\mbox{Stab}(X)$. Let $\widetilde{\mbox{GL}_{2}^{+}(\mathbb{R})}$ be the universal cover of the orientation preserving component of $\mbox{GL}_{2}(\mathbb{R})$. Then there is a natural action of $\widetilde{\mbox{GL}_{2}^{+}(\mathbb{R})}$ on $\mbox{Stab}(X)$. There exists a connected component $\mbox{Stab}^{\dagger}(X)\subset \mbox{Stab}(X)$ and an open subset $U(X)\subset \mbox{Stab}^{\dagger}(X)$, such that $\widetilde{\mbox{GL}_{2}^{+}(\mathbb{R})}$ acts freely on $U(X)$. A section of this action is naturally identified with $\mathcal{L}(X)$. We will carry all the work in this paper on $\mathcal{L}(X)$.

\subsection{Wall and Chamber Structure}

Our main usage of $\mathcal{L}(X)$ is the wall and chamber structure. Let $v$ be a Mukai vector. Then there is a locally finite collection of real codimension one submanifolds $\{W_{i}\}_{i\in I}$ of $\mathcal{L}(X)$, such that each connected component $C$ of $\mathcal{L}(X)\backslash \cup_{i\in I}W_{i}$ has the following property: If $E\in \mathcal{D}^{b}(X)$ is $\sigma$-semistable for some $\sigma \in C$, then $E$ is $\sigma$-semistable for all $\sigma\in C$. The walls $\{W_{i}\}_{i\in I}$ are called the \emph{(actual) walls for $v$}, and such connected components are called \emph{chambers}.

For two Mukai vectors $u,v$, the \emph{numerical wall} $W(u,v)$ is
$$W(u,v)=\{\sigma\in \mathcal{L}(X): \phi_{\sigma}(u)=\phi_{\sigma}(v)\}.$$
An actual wall must be numerical, but a numerical wall is not necessarily actual.

For a stability condition $\sigma=(Z,\mathcal{A})$ and a primitive Mukai vector $v$, the $\sigma$-semistable objects (hence stable by primitivity) in $\mathcal{A}$ with Mukai vector $v$ form a moduli space, which we shall denote by $M_{\sigma}(v)$. The following theorem is a generalization of Theorem \ref{moduli}.

\begin{theorem}[\cite{BM14a, BM14b}]\label{B-moduli}
  Let $\sigma\in \mathcal{L}(X)$ be a stability condition and $v\in  \mbox{H}^{*}_{alg}(X)$ be a primitive Mukai vector. Then $M_{\sigma}(v)$ is non-empty when $v^{2}\geq -2$. If in addition $\sigma$ is generic with respect to $v$, then
  \begin{enumerate}
  \item The dimension of $M_{\sigma}(v)$ is $v^{2}+2$,
  \item $M_{\sigma}(v)$ is an irreducible normal projective variety with $\mathbb{Q}$-factorial singularities.
  \end{enumerate}
\end{theorem}
  In particular, if $v$ is spherical and $\sigma$ is generic, then $M_{\sigma}(v)$ consists of a single point. The following theorem connects Gieseker stability and Bridgeland stability.

  \begin{theorem}[\cite{Bri08}, Large volume limit]\label{large volume limit}
 Let $H$ be a polarization of $X$, $E\in \mathcal{D}^{b}(X)$ such that $r(E)>0$ and $c_{1}(E)\cdot H>0$. Then $E$ is $\sigma$-semistable for $\sigma=(Z_{0,tH}, \mathcal{A}_{0,tH}), t\gg 0$ if and only if some shift of $E$ is $H$-Gieseker semistable.
  \end{theorem}
We shall refer the chamber that contains this stability condition the \emph{H-Gieseker chamber} of $v$.

\section{Global Reduction}\label{section3}

Let $v=(r,D,a)$ be a spherical Mukai vector with $r>0$. Let $H$ be a generic polarization. We first assume $D\cdot H\neq 0$, otherwise we may take some $H'$ in the ample cone that is sufficiently close to $H$, such that $M_{H'}=M_{H}$. Then we may assume $D\cdot H>0$: since $E\in M_{H}(v)$ is a vector bundle \cite{HL10}, $E^{*}\in M_{H}(r,-D,a)$. If $D\cdot H <0$, then $-D\cdot H>0$. By Serre duality, $ \mbox{h}^{i}(E)=\mbox{h}^{2-i}(E^{*})$, hence it suffices to compute the cohomology groups of $E^{*}$. Hence from now on we assume $v$ satisfies $r>0$ and $D\cdot H>0$. 

Since $E$ is an $H$-Gieseker stable vector bundle, $E^{*}$ is also stable. We have $ \mbox{H}^{2}(E)= \mbox{H}^{0}(E^{*})^{*}=0$ since
$$\mu(E^{*})=-\mu(E)=-D\cdot H<0=\mu(\mathcal{O}_{X}).$$
Hence the only possible non-zero cohomology groups of $E$ are $ \mbox{H}^{0}(E)$ and $ \mbox{H}^{1}(E)$. By Riemann-Roch we have
$$\mbox{h}^{0}(E)-\mbox{h}^{1}(E)=\chi(E)=r+a.$$
Hence if we know $\mbox{h}^{0}(E)$ or $\mbox{h}^{1}(E)$, the other is also known.

\subsection{Finiteness of Harder-Narasimhan Walls}

The strategy of the algorithm is to compute the Harder-Narasimhan filtration of a spherical object at a certain stability condition. To make this precise, we have to first prove that there are only finitely many walls at which the Harder-Narasimhan filtration changes.

\begin{definition}
Let $E\in \mathcal{D}^{b}(X)$ be an object. A Harder-Narasimhan chamber $U$ for $E$ is an open subset of $\mathcal{L}(X)$ such that the $\sigma$-Harder-Narasimhan filtration of $E$ is constant for $\sigma\in U$. A Harder-Narasimhan wall is the intersection of the closures of two adjacent Harder-Narasimhan chambers.
\end{definition}

For an ample divisor $\omega$ and any divisor $\beta$, consider the upper half plane slice $\mathbb{H}=\{(s\beta,t\omega)| t>0\}$. Let $\beta_{1}, \cdots , \beta_{\rho}$ be an effective basis for $\mbox{Pic}(X)$, and $\omega_{1}, \cdots , \omega_{\rho}$ be an ample basis for $\mbox{Pic}(X)$ with $\omega_{1}=H$. The $\mathcal{Q}(X)$ described in \cite{Bri08} can be embedded into
$$\mbox{Pic}(X)_{\mathbb{C}}\cong \mathbb{R}^{2\rho}=\{(s_{1} \beta_{1}, s_{2}\beta_{2}, \cdots , s_{\rho}\beta_{\rho}, t_{1}\omega_{1}, t_{2}\omega_{2} \cdots, t_{\rho}\omega_{\rho}| s_{i}, t_{j}\in \mathbb{R} \}.$$
For $\epsilon >0$ and $0<y_{0}<y_{1}$ let
$$\widetilde{\mathbb{B}}(\epsilon; y_{0}, y_{1})=\{(s_{1}, \cdots, s_{\rho},t_{1}, \cdots, t_{\rho})| 0\leq s_{1}, \cdots, s_{\rho} \leq \epsilon, y_{0}\leq t_{1} \leq y_{1}, 0\leq t_{2}, \cdots, t_{\rho}<\epsilon \}\subset \mathbb{H},$$
and let $\mathbb{B}=\mathbb{B}(\epsilon; y_{0}, y_{1})=\widetilde{\mathbb{B}}(\epsilon; y_{0}, y_{1})\cap \mbox{Stab}(X)$.

Let $E$ be some object which lies in $\mathcal{A}_{0,\omega_{1}}$ with $c_{1}(E)\cdot \omega_{1}>0$, and suppose $E$ is stable with respect to $(Z_{0, t_{1}\omega_{1}}, \mathcal{A}_{0, t_{1}\omega_{1}})$ for $t_{1}\gg 0$ (By \cite{Bri08}, this is clearly true for an $H$-stable vector bundle $E$). Then we may choose $\epsilon$ sufficiently small, such that  for any $p\in \mathbb{B}(\epsilon; y_{0}, y_{1})$, $E\in \mathcal{A}_{p}$. Let
$$\mathcal{S}_{\mathbb{B}}(E)=\{v(F)| F \mbox{ is a } \sigma-\mbox{Jordan-H\"{o}lder factor of } E \mbox{ for some }\sigma\in \mathbb{B}\}.$$
The following proof is essentially covered in \cite{Bri08} Section 9.

\begin{proposition}\label{finite}
Let the Mukai vector of $E$ be $(r,D,a)$, where $D\cdot H>0$. Then $\mathcal{S}_{\mathbb{B}(\epsilon; y_{0}, y_{1})}(E)$ is finite for fixed $y_{0}, y_{1}$ and suffciently small $\epsilon$. 
\end{proposition}

\begin{proof}
  Let the Mukai vector of $E$ be $(r,D,a)$.
  Take any $z_{1}>\sqrt{2/H^{2}}$. Take sufficiently small $\epsilon'$ so that in the box
  $$B=\{(s,t)| 0 \leq s_{1}, \cdots, s_{\rho}, t_{2}, \cdots, t_{\rho} \leq \epsilon', z_{1}-\epsilon' \leq t_{1} \leq z_{1}\},$$
  the Harder-Narasimhan filtration of $E$ is fixed. Let $E'$ be the first factor. Let $B'$ be a slightly smaller box inside $B$, and let $\mu'=\mbox{min}\{-\frac{1}{\mu_{s,t}(A')} | (s,t)\in B \}$. Consider the set
  $$\mathcal{C}=\bigcup_{(s,t)\in B'}\{\mbox{Mukai vectors of subobjects of }E \mbox{ in }\mathcal{A}_{s,t}\}. $$
  I claim for all $(r',D',a')\in \mathcal{C}$, $r'$ is bounded below and $a'$ is bounded above. Suppose not, let $v_{i}=(r_{i}, D_{i}, a_{i})\in \mathcal{C}$ be a sequence such that $\mbox{lim}_{i \rightarrow \infty}r_{i}=-\infty$. Suppose $v_{i}$ is the Mukai vector of a suboject $E_{i}\subset E$ in $\mathcal{A}_{i}=\mathcal{A}_{\beta^{(i)}, \omega^{(i)}}$, let $(\beta, \omega)$ be an accumulation point of $\{(\beta^{(i)}, \omega^{(i)})\}$ and we may assume $\mbox{lim}(\beta^{(i)}, \omega^{(i)})=(\beta,\omega)$. Then since $E_{i}\subset E$ in $\mathcal{A}_{i}$, their central charges at $Z_{\beta^{(i)}, \omega^{(i)}}$ have imaginary parts no larger than that of $E$: 
  $$\omega^{(i)}(D_{i}-r_{i}\beta^{(i)})\leq \omega^{(i)}(D-r\beta^{(i)}).$$
Note that the right hand side has a finite limit, so $\omega^{(i)}(D_{i}-r_{i}\beta^{(i)})$ is bounded. On the other hand, note that for a $\lambda$ sufficiently close to 1, $(\beta^{(i)}, \lambda\omega^{(i)})\in B$
  $$\frac{-a_{i}-r_{i} \frac{(\beta^{(i)})^{2}-(\lambda \omega^{(i)})^{2}}{2}+D_{i}\beta^{(i)} }{\omega^{(i)}(D_{i}-r_{i}\beta^{(i)})} $$
  is universally bounded below by $\mu'$, otherwise $E_{i}$ has larger phase than $E'$ and would have been the first Harder-Narasimhan factor. Moreover, since $\omega^{(i)}H(D_{i}-r_{i}\beta^{(i)})$ is universally bounded,
  $$-a_{i}-r_{i}\frac{(\beta^{(i)})^{2}-(\lambda \omega^{(i)})^{2}}{2}+ D_{i}\beta^{(i)} $$
is universally bounded below. Note that since $\epsilon$ is small, $(\beta^{(i)})^{2}-(\lambda \omega^{(i)})^{2}$ are negative. If $a_{i}$ is not bounded above or $r_{i}$ not bounded below, we may perturb $\lambda$ so that this quantity is unbounded below, a contradiction. The claim is proved.

Now let
$$\mathcal{S}'_{\mathbb{B}(\epsilon; y_{0}, y_{1})}(E)=\{v(F)| F \mbox{ is the first Jordan-H\"{o}lder factor of } E \mbox{ for some }\sigma \in \mathbb{B} \}.$$
We may choose $\epsilon$ small enough so that for every $(\beta, \omega) \in \mathbb{B}$, there exists a scaling $(\beta, \nu \omega)\in B$ for some $\nu\in \mathbb{R}_{+}$. Note that $\mathcal{A}_{\beta, \nu \omega}$ is independent of $\nu$, we have $\mathcal{S}'\subset \mathcal{C}$. For any $v'=(r',D',a')\in \mathcal{S}'$, since $a'$ is bounded above and $r'$ is bounded below, the imaginary parts of the central charges are universally bounded, and the phase is universally bounded below by the phase of $E$, we see $\frac{|Z_{\beta, \omega}(v')|}{|Z_{\beta,\omega}(v)|} $ is universally bounded above and below for all $(\beta, \omega)\in \mathbb{B}$. Hence by the geometry of the Harder-Narasimhan polygon, $\frac{|Z_{\beta, \omega}(v'')|}{|Z_{\beta, \omega}(v)|} $ is universally bounded above and below for all $v''\in \mathcal{S}(\mathbb{B})$. By Lemma 9.2 of \cite{Bri08}, it suffices to show that $|Z_{\beta, \omega}(E)|$ is bounded above, but this is clear since we may extend $Z_{s,t}$ continuously to $\widetilde{\mathbb{B}}$.
\end{proof}

\begin{corollary}\label{finite}
The number of Harder-Narasimhan walls in $\mathbb{B}$ is finite.
\end{corollary}

We donote the set of Harder-Narasimhan walls by $HN(\mathbb{B})$.

\subsection{Separating Walls}

The idea of the algorithm is to compute the Harder-Narasimhan filtration at a certain stability condition $\sigma_{0}$. In this subsection, we will choose a path $\mathbf{b}$ from a stability condition $\sigma_{1}$ in the Gieseker chamber to $\sigma_{0}$, so that the intersections of all the walls in $HN(\mathbb{B})$ with $\mathbf{b}$ are pairwise distinct. The order on $\mathbf{b}$ induces an order on $HN(\mathbb{B})$ that records the order of touching walls when we move from $\sigma_{1}$ to $\sigma_{0}$ along $\mathbf{b}$. We also show how to compute this order.

The next lemma shows that all Harder-Narasimhan walls in $\mathbb{B}$ are locally a function in $s$.

\begin{lemma}\label{impossible}
  There is no divisor $D$ on a K3 surface such that
  $$D^{2}=-2, D\cdot H=0.$$
\end{lemma}

\begin{proof}
  If $D^{2}=-2$, by Riemann-Roch, $\chi(\mathcal{O}_{X}(D))=2+\frac{1}{2}D^{2}=1 $. Either $ \mbox{H}^{0}(\mathcal{O}_{X}(D))$ or $ \mbox{H}^{2}(\mathcal{O}_{X}(D))= \mbox{H}^{0}(\mathcal{O}_{X}(-D))^{*}$ is non-zero. Hence $D$ or $-D$ is effective, so $D\cdot H\neq 0$.
\end{proof}

The Harder-Narasimhan wall $W$, on each slice $\mathbb{H}=\mathbb{H}_{s\beta, t\omega_{1}}$, is not a vertical wall, hence a function near $s=0$. Consequently, by the inverse function theorem locally near $\{(0, t_{1}\omega_{1})| t_{1}\in \mathbb{R}_{+}\}$ it is still a function in $(s_{1}, \cdots, s_{\rho}, t_{2}, \cdots, t_{\rho})$ valued in $t_{1}$. We may write it as $t_{1}=f_{W}(s_{1}, \cdots, s_{\rho}, t_{2}, \cdots, t_{\rho})$, where recall that we have chosen the embedding
$$\mathcal{Q}(X)\subset \mbox{Pic}(X)\otimes \mathbb{C}=\{(s_{1}\beta_{1}, \cdots, s_{\rho}\beta_{\rho}, t_{1}\omega_{1}, \cdots, t_{\rho}\omega_{\rho})| s_{i}, t_{j}\in \mathbb{R}\},$$
$\omega_{i}$ are ample, and $\beta_{i}$ are chosen to be effective. Then we may write
$$t_{1}=f_{0}(W)+\partial s_{1}(W)s_{1}+\cdots +\partial s_{\rho}(W)s_{\rho}+ \partial t_{2}(W) t_{2}+ \cdots +\partial t_{\rho} (W) t_{\rho} +o(s,t),$$
where $o(s,t)$ consists of higher order terms of $s_{i}, t_{j}$.
Let
$$\mathcal{W}(\mathbb{B})=\{\partial W = (f_{0}(W), \partial s_{1}(W), \cdots, \partial s_{\rho} (W), \partial t_{2}(W), \cdots, \partial t_{\rho}(W))| W\in HN(\mathbb{B})\}\subset \mathbb{R}^{2\rho}.$$
By Proposition \ref{finite}, this is a finite set. The next lemma shows that $\mathcal{W}(\mathbb{B})$ separates all walls.

\begin{lemma}
$HN(\mathbb{B}) \longrightarrow \mathcal{W}(\mathbb{B})$ is bijective. 
\end{lemma}

\begin{proof}
  We need to show that if $(f_{0}(W), \partial s_{1}(W), \cdots, \partial s_{\rho}(W), \partial t_{2}(W), \cdots, \partial t_{\rho}(W))$ is known, then the wall $W$ is determined. Viewing $W\subset \mathbb{R}^{2n}$, it has an equation
  $$F=F_{0}+\sum_{i=1}^{\rho}\frac{\partial F}{\partial s_{i}}s_{i}+\sum_{i=1}^{\rho}\frac{\partial F}{\partial t_{i}}+o(s,t). $$
  By assumption, $[F]=[F_{0}, \frac{\partial F}{\partial s_{1}}, \cdots, \frac{\partial F}{\partial s_{\rho}}, \frac{\partial F}{\partial t_{1}}, \cdots, \frac{\partial F}{\partial t_{\rho}}]\in \mathbb{P}^{2\rho}_{\mathbb{R}}$ is known. Hence if we take any slice $\mathbb{R}^{2}_{\beta, \omega}=\{(s\beta, t\omega)| s,t\in \mathbb{R}\}$ and consider $W\cap \mathbb{R}^{2}_{\beta, \omega}$, it is a circle with center lying on the $\beta$-axis by \cite{Mac14}. We may write its equation as $t=g(s)$, then $g(0)$, $g'(0)$ are already determined by $[F]$. Since $g$ is a circle centered on the $\beta$-axis, there is a unique such $g$ with given $g(0)$, $g'(0)$. Now since $\mathcal{L}(X)\subset \mathbb{R}^{2n}$ is covered by such $\mathbb{R}^{2}_{\beta, \omega}$ for $\beta\in \mbox{Pic}(X)_{\mathbb{R}}$ and $\omega \in \mbox{Amp}(X)_{\mathbb{R}}$, $W\cap \mathcal{L}(X)$ is uniquely determined. Hence in particular $HN(\mathbb{B})\longrightarrow \mathcal{W}(\mathbb{B})$ is bijective.  
\end{proof}

Before we proceed to separate the walls, we need a technical lemma that describes $\mathcal{L}(X)$ near the $\omega_{1}$-axis. Let $ \mathbb{H}:=\{(s\beta, t\omega)|s\in \mathbb{R}, t\in \mathbb{R}_{>0}\}$ be an upper half plane. Not all points in $\mathbb{H}$ are stability conditions, however we have the following estimate:

\begin{lemma}\label{slice}
Let $ V=\mathbb{H}\setminus \mathbb{H}\cap \mathcal{L}(X)$. Then 
$$ V \cap \{s=0\}=\{0\}\times (0,\sqrt{2/\omega^{2}}],$$
$$ \mbox{lim}_{\epsilon \rightarrow 0} \mbox{sup}\{t : (s,t)\in V, 0<s<\epsilon\}=0.$$
\end{lemma}

\begin{proof}
Let $ v=(r,D,a)$ be a Mukai vector such that $ v^{2}=-2$ and $ r>0$.

For the first claim, by definition, 
$$Z_{0,t}(v)=-a+r\frac{t^{2}}{2}\omega^{2}+it \omega \cdot D.$$
Suppose $ Z_{0,t}(v)\in \mathbb{R}_{\leq 0}$. Then $ \omega\cdot D=0$. There are two cases. If $ D=0$, then $ v^{2}=-2ra=-2$, $ r=a=\pm 1$. By assumption $ r>0$, hence $ r=a=1$. In this case $ -a+r\frac{t^{2}}{2}\omega^{2}\leq 0$ is equivalent to $t\leq \sqrt{2/\omega^{2}} $. If $ D\neq 0$, then since $ D\omega=0$, by the Hodge Index Theorem, $ D^{2}<0$. Since on a K3 surface the pairing is even, $ D^{2}\leq -2$. Since $-2=v^{2}=D^{2}-2ra\leq -2-2ra $, $ ra\leq 0$. Since $ r>0$, $ a\leq 0$. In this case $ -a+r\frac{t^{2}}{2}\omega^{2}$ is always positive. The first claim is proved.

For the second claim, by definition, 
$$Z_{\epsilon, t}(v)=-a-r\frac{\epsilon^{2}\beta^{2}-t^{2}\omega^{2}}{2}+\epsilon D\cdot \beta  +i t\omega(D-r\epsilon \beta).$$
Suppose $ Z_{\epsilon, t}(v)\in \mathbb{R}_{\leq 0}$. Write $\beta=c\omega+\gamma$, where $\gamma\cdot \omega=0$. Then $|D\gamma|\leq M \sqrt{|D^{2}|}\sqrt{|\gamma^{2}|}$ for some constant $M>0$. Since $Im(Z(v))=0$, $ D\omega=r\epsilon \omega\beta$, 
$$-a-r\frac{\epsilon^{2}\beta^{2}-t^{2}\omega^{2}}{2}+\epsilon D\cdot\beta =-a-r\frac{\epsilon^{2}\beta^{2}-t^{2}\omega^{2}}{2}+r\epsilon^{2}c\omega\beta+\epsilon D\gamma\leq 0$$
is equivalent to
$$\frac{t^{2}\omega^{2}}{2} \leq \frac{a}{r}+ \frac{\epsilon^{2}\beta^{2}}{2}-\epsilon^{2}c\omega\beta+\epsilon \frac{D\gamma}{r}. $$
Since the left hand side is positive and
$$\frac{D\gamma}{r} \leq M\frac{\sqrt{|D^{2}|}}{r}\sqrt{|\gamma^{2}|}=\frac{\sqrt{|2ra-2|}}{r}\sqrt{|\gamma^{2}|}\leq \sqrt{2|\frac{a}{r}|}\sqrt{|\gamma^{2}|}+\sqrt{2|\gamma^{2}|}, $$
it suffices to prove $ a/r \rightarrow 0$ as $ \epsilon \rightarrow 0$. Now since $ \epsilon>0$, $ D\omega=r\epsilon \omega\beta$ is a positive integer, hence $ r=\frac{D\omega}{\epsilon \omega\beta} \geq \frac{1}{\epsilon \omega\beta} $. Also note that by the Hodge Index Theorem, $ \omega^{2}D^{2}\leq (D\omega)^{2}=r^{2}\epsilon^{2}(\omega\beta)^{2}$, namely $ D^{2}\leq r^{2}\epsilon^{2}\frac{(\omega\beta)^{2}}{(\omega^{2})}$. Since $ v^{2}=D^{2}-2ra=-2$, we have  
$$\frac{a}{r}=\frac{D^{2}+2}{2r^{2}}\leq \frac{r^{2}\epsilon^{2}\frac{(\omega\beta)^{2}}{(\omega^{2})}+2}{2r^{2}}\leq \frac{r^{2}\epsilon^{2}(\omega\beta)^{2}+2\omega^{2}}{2\omega^{2}(1/\epsilon \omega\beta)^{2}}=\epsilon^{2}(\omega\beta)^{2}(r^{2}\epsilon^{2}(\omega\beta)^{2}+2\omega^{2})/(2\omega^{2}).$$
The second claim is proved.
\end{proof}

Now we put a lexicographic order on $\mathcal{W}(\mathbb{B})$. Then there exists
$$0<\epsilon'_{\rho} \ll \cdots \ll \epsilon'_{2} \ll \epsilon_{\rho} \ll \cdots \ll \epsilon_{1} \ll 1, $$
such that if we let $t$ range from infinity to, say $\sqrt{2/\omega_{1}^{2}}/2$, the order of the stability condition
$$\{(Z_{(\epsilon_{1}\beta_{1}+\cdots + \epsilon_{\rho}\beta_{\rho}, t(\omega_{1}+\epsilon'_{2}\omega_{2}+\cdots + \epsilon'_{\rho}\omega_{\rho}))}, \mathcal{A}_{(\epsilon_{1}\beta_{1}+\cdots + \epsilon_{\rho}\beta_{\rho}, t(\omega_{1}+\epsilon'_{2}\omega_{2}+\cdots + \epsilon'_{\rho}\omega_{\rho}))}), t>\sqrt{2/\omega_{1}^{2}}/2\}$$
touching walls in $HN(\mathbb{B})$ coincide with the order of walls in $\mathcal{W}(\mathbb{B})$. We fix this set once and for all and call it $\mathbf{b}$. By Lemma \ref{slice}, $\mathbf{b}\subset \mathcal{L}(X)$. Note that the abelian category $\mathcal{A}=\mathcal{A}_{(\epsilon_{1}\beta_{1}+\cdots + \epsilon_{\rho}\beta_{\rho}, t(\omega_{1}+\epsilon'_{2}\omega_{2}+\cdots + \epsilon'_{\rho}\omega_{\rho}))}$ is independent of $t$.

\subsection{Global Reduction}

In this subsection we state the global reduction, which is the algorithm to compute the cohomology of a spherical vector bundle.

We say a \emph{shape} of $E$ is a pair
$$(E^{\bullet}, \partial W)\in \mbox{Fil}(\mathcal{A})\times \mathcal{W}(\mathbb{B}),$$
where $E^{\bullet}$ is a filtration of $E$.
Now suppose $\sigma\in \mathbf{b}$ is a stability condition that is not on a wall in $HN(\mathbb{B})$. It determines a shape of $E$
$$sh(\sigma)=(E^{\bullet}_{\sigma},\partial W)\in \mbox{Fil}(\mathcal{A})\times \mathcal{W}(\mathbb{B}),$$
where $E^{\bullet}_{\sigma}$ is the $\sigma$-Harder-Narasimhan filtration, and $W\in HN(\mathbb{B})$ is the wall right below $\sigma$. Note that by the construction of $\mathbf{b}$, all walls are distinct points on $\mathbf{b}\cong \mathbb{R}_{+}$, this is well-defined. The next simple but important observation shows how shapes are related to the cohomology of $E$.

\begin{proposition}\label{preoutput}
  Let $\sigma\in \mathbf{b}$ be a stability condition with $sh(\sigma)=(E^{\bullet}, \partial W)$, and write
  $$\partial W = (f_{0}(W), \partial s_{1}(W), \cdots, \partial s_{\rho} (W), \partial t_{2}(W), \cdots, \partial t_{\rho}(W)).$$
  If $f_{0}(W)< \sqrt{2/\omega_{1}^{2}}$, then the last $\sigma$-Harder-Narasimhan factor of $E^{\bullet}$ is $\mathcal{O}_{X}[1]^{\oplus h}$ for some $h\in \mathbb{Z}_{\geq 0}$ ($h$ can be zero, in this case we still have $\mbox{Hom}(E_{i},\mathcal{O}_{X}[1])=0$ for any factor $E_{i}$).
\end{proposition}
 
\begin{proof}
  Let $E_{i}$ be any Harder-Narasimhan factor of $E$ that is not a direct sum of $\mathcal{O}_{X}[1]$. By Lemma \ref{di>0} that we shall show below, $d_{i}>0$. Then we may shrink $\epsilon_{i}, \epsilon_{j}'$ further such that $\phi_{\infty}(E_{i})<\phi_{\infty}(\mathcal{O}_{X}[1])$.

  Then there are two cases. First if $\phi_{\sigma}(E_{i})< \phi_{\sigma}(\mathcal{O}_{X}[1])$ for some factor $E_{i}$, consider the wall
  $$W'=W(E_{i}, \mathcal{O}_{X}[1]).$$
  By assumption, $W'$ is under $\sigma$. By construction, $\partial W'$ has $f_{0}(W')=\sqrt{2/\omega_{1}^{2}}$, which contradicts our assumption. Hence $W'\notin HN(\mathbb{B})$, namely when crossing $W'$ the Harder-Narasimhan filtration does not change. Hence $\mathcal{O}_{X}[1]$ cannot appear in $E^{\bullet}$, and $\mbox{Hom}(E_{i},\mathcal{O}_{X}[1])=0$. For those $E_{j}$ with $\phi_{\sigma}(E_{j})\geq \phi_{\sigma}(\mathcal{O}_{X}[1])$, they are not direct sums of $\mathcal{O}_{X}[1]$, hence by the assumption that $\sigma$ is not on a wall they have $\phi_{\sigma}(E_{j})> \phi_{\sigma}(\mathcal{O}_{X}[1])$. Hence $E_{j}$ also have $\mbox{Hom}(E_{j},\mathcal{O}_{X}[1])=0$. Hence the proposition is proved in this case.

  The second case is that $\phi_{\sigma}(E_{i})\geq \phi_{\sigma}(\mathcal{O}_{X}[1])$. By assumption $\sigma$ is not on a wall, hence for those $E_{i}\neq \mathcal{O}_{X}[1]^{\oplus h}$, they have $\phi_{\sigma}(E_{i})> \phi_{\sigma}(\mathcal{O}_{X}[1])$. Hence if $\mathcal{O}_{X}[1]$ appears in $E^{\bullet}$, it must be the last Harder-Narasimhan factor.
\end{proof}

\begin{corollary}\label{output}
  Let $\sigma$ satisfy the condition in Proposition \ref{preoutput}. Then
  $$\mbox{h}^{1}(E)=h,$$
  where the last Harder-Narasimhan factor of $E^{\bullet}$ is $\mathcal{O}_{X}[1]^{\oplus h}$. 
\end{corollary}

\begin{proof}
  We have the exact sequence
  $$0 \longrightarrow E_{m-1} \longrightarrow E \longrightarrow \mathcal{O}_{X}[1]^{\oplus h} \longrightarrow 0, $$
  where $\mathcal{O}_{X}[1]^{\oplus h}$ is the last Harder-Narasimhan factor. By Proposition \ref{preoutput}, $E_{m-1}$ has a filtration whose factors have no map to $\mathcal{O}_{X}[1]$. Hence $\mbox{Hom}(E,\mathcal{O}_{X}[1])=\mbox{Hom}(Q_{m},\mathcal{O}_{X}[1])$.
\end{proof}

Hence to compute the cohomology groups of $E$, it suffices to compute the shape at a certain stability condition. Now suppose we know the shape $\mathcal{P}=(E^{\bullet}_{\mathcal{P}}, \partial W_{\mathcal{P}})$ of $E$ at a stability that is not on a wall. We call the shape $\mathcal{P}'=(E^{\bullet}_{\mathcal{P}'}, \partial W_{\mathcal{P}'})$ the \emph{next} shape after $\mathcal{P}$ if $\mathcal{P}'=sh(\sigma)$ for $\sigma$ in the chamber right below $W_{\mathcal{P}}$. By definition, $W_{\mathcal{P}'}<W_{\mathcal{P}}$. Note that $\mathcal{W}(\mathbb{B})$ is finite, hence if we keep computing the next shapes, in finite steps it must happen that $f_{0}(W_{\mathcal{P}})<\sqrt{2/\omega_{1}^{2}}$ for some shape $\mathcal{P}$. By Corollary \ref{output}, the cohomology is computed and the algorithm terminates here.

We have to examine closly when shapes change. Let $\mathcal{P}=(E^{\bullet}, \partial W)$ be the shape at $\sigma_{+}$ with filtration $0=E_{0}^{+}\subset E_{1}^{+}\subset \cdots \subset E_{m}^{+}=E$ and $\mathcal{P}'=sh(\sigma_{-})$ be the next shape. We know that by definition of the Harder-Narasimhan filtration, $G_{i}$ are semi-stable, and
$$\phi_{+}(G_{1}^{+})> \phi_{+}(G_{1}^{+}) > \cdots > \phi_{+}(G_{m}^{+}).$$
It will stop being a Harder-Narasimhan filtration at $\sigma_{-}$, if either one or both of the following happpen:
\begin{itemize}
\item Some $G_{i}$ become unstable;
\item For some pair $j<k$, the phases become $\phi_{-}(G_{j})< \phi_{-}(G_{k})$.
\end{itemize}
In the first case, an actual wall for some factor is in between $\sigma_{\pm}$, and in the second case, the numerical wall of a pair $G_{j}, G_{k}$ is in between $\sigma_{\pm}$. By construction of $\mathbf{b}$, all such walls are separated, hence $W$ corresponds to a unique rank 2 lattice $\mathcal{H}\subset  \mbox{H}^{*}_{alg}(X)$, the $G_{i}, G_{j}, G_{k}$ described above all have Mukai vectors lying in $\mathcal{H}$. Furthermore, let $i$/$j$ be the minimal/maximal index such that $G_{i/j}\in \mathcal{H}$, then for any $i\leq k \leq j$ and $\sigma_{0}\in W\cap \mathbf{b}$, $\phi_{0}(G_{i})=\phi_{0}(G_{k})=\phi_{0}(G_{j})$. Since all walls are separated, $G_{k}$ also has Mukai vector in $\mathcal{H}$. This provides us a way to determine the next wall:
\\
Compute all numerical walls of all adjacent factors of the current shape, and all actual walls for all factors, denote them by $\{W_{l}\}$. Then we pick the $W_{l}$ with maximal $\partial W_{l}\in \mathcal{W}(\mathbb{B})$, that is the next wall. We may then track back the $E_{j}/E_{i}$ that produced it.

Let
  $$0=F_{0}^{-}\subset F_{1}^{-} \subset \cdots \subset F_{n}^{-}=E_{j}/E_{i-1}$$
  be the $\sigma_{-}$-Harder-Narasimhan filtration of $E_{j}/E_{i}$. Then it has a lift that replaces the segment $E_{j}/E_{i-1}$ of the original filtration of $E$:
  $$0=E_{0}^{+}\subset E_{1}^{+}\subset \cdots \subset E_{i-1}^{+} \subset \widetilde{F_{1}^{-}}\subset \cdots \subset \widetilde{F_{n-1}^{-}}\subset E_{j}^{+}\subset \cdots \subset E_{m}^{+}=E.$$
 By continuity, it is clearly the $\sigma_{-}$-Harder-Narasimhan filtration for $\sigma_{-}$ being sufficiently close to $W$, hence the Harder-Narasimhan filtration of the next shape. Hence the rest of the algorithm is to compute the Harder-Narasimhan filtration of all possible $E_{j}/E_{i-1}$ that can appear above at an adjacent chamber from its known Harder-Narasimhan filtration of the given chamber. Then from a shape, we can compute its next shape, and by iterating this, eventually the algorithm terminates. Before we do the computation, we need to capture some common properties of all such possible $E_{j}/E_{i-1}$. 

\begin{lemma}\label{di>0}
Preserve the notation in Proposition \ref{preoutput}. Let $G$ be a $\sigma$-Harder-Narasimhan factor. If $G $ is not a direct sum of $\mathcal{O}_{X}[1]$, then its Mukai vector $(r,D,a)$ satisfies $D\cdot \omega_{1}>0$. 
\end{lemma}

\begin{proof}
  We use induction on the finite ordered set $\mathcal{W}(\mathbb{B})$. Let $\partial W_{0}$ be the initial element in $\mathcal{W}(\mathbb{B})$. Since $E$ is stable, the corresponding shape of $E$ is $(E, \partial W_{0})$, here the filtration is trivial. By assumption $E$ has $c_{1}(E)\cdot \omega_{1}>0$, the base case is true.

  Now suppose we have proved the lemma for $sh(\sigma_{+})=(E^{\bullet}_{+}, \partial W)$ and $\mathcal{P}'=sh(\sigma_{-})$ is the next shape. By the discussion above, it suffices to show $G_{i}^{-}=F_{i}^{-}/F_{i-1}^{-}$ has positive degree with respect to $\omega_{1}$. Let the Mukai vector of $E_{j}/E_{i-1}$ be $v=(r,D,a)$ and the Mukai vectors of any Jordan-H\"{o}lder factor of $G_{i}^{-}$ be $v_{i}=(r_{i},D_{i}, a_{i})$, $Z'=Z_{(0,f_{0}(W)\cdot \omega_{1})}$.

  If $f_{0}(W)>\sqrt{2/\omega_{1}^{2}}$, by induction $Im(Z'(v))>0$, and by semi-stability of $G_{i}^{-}$, $Z'(v_{i})\neq 0$. Hence
  $$Im(Z'(v_{i}))=f_{0}(W)\omega_{1}\cdot D_{1}$$
  is either positive or negative, since $Z'$ is a stability condition on the wall. It cannot be negative. Otherwise let $Z=W\cap \mathbf{b}$, then there is a continuous path $\gamma \subset W$ between $Z'$ and $Z$. Since $Im(Z(v_{i}))>0$ there must be some $Z''$ in between such that $Im(Z''(v_{i}))=0$. Since we have chosen $\mathbb{B}$ sufficiently small such that $Im(Z''(v))>0$ for any $Z''\in \mathbb{B}$, $Z''(v_{i})=0$. This is impossible by the positivity of stability conditions.

  If $f_{0}(W)=\sqrt{2/\omega_{1}^{2}}$, $Z'$ is no longer a stability condition, but it is still a function with $Im(Z'(v))>0$. The only difference is that $Im(Z'(v_{i}))$ could be zero. In this case, since $v$ and $v_{i}$ have the same phase at $Z'$, $Z'(v_{i})=0$. Explicitly,
  $$-a_{i}+r_{i}\frac{\omega_{1}^{2}}{2}\cdot \frac{2}{\omega_{1}^{2}}=r_{i}-a_{i}=0, D_{i}\cdot \omega_{1}=0.  $$
  From the second equation, either $D_{i}=0$ or $D_{i}^{2}<0$ by the Hodge Index Theorem. If $D_{i}=0$, since $v_{i}$ is spherical, $-2=D_{i}^{2}-2r_{i}a_{i}$, we know $r_{i}=a_{i}=\pm 1$. Since $G_{i}^{-}\in \mathcal{A}$, $r_{i}=a_{i}=-1$ and $G_{i}^{-}$ is a direct sum of $\mathcal{O}_{X}[1]$, which is excluded. Hence $D_{i}^{2}<0$. Since the K3 lattice is even, $D_{i}^{2}\leq -2$. Since $r_{i}-a_{i}=0$, we have
  $$-2=D_{i}^{2}-2r_{i}a_{i}\leq -2-2r_{i}^{2}\leq -2.$$
  Hence all inequalities are equalities. In particular, $D_{i}^{2}=-2$. By Lemma \ref{impossible}, this is impossible. Hence $Im(Z'(v_{i}))\neq 0$, and all arguments for $f_{0}(W)\neq \sqrt{2/\omega_{1}^{2}}$ case apply without change.  
\end{proof}

\begin{definition}
We say an object $E\in \mathcal{D}^{b}(X)$ is \emph{rigid} if $\mbox{Ext}^{1}(E,E)=0$.
\end{definition}

A very important fact for rigid objects is Mukai's Lemma.

\begin{lemma}[\cite{Bri08}, Lemma 5.2]\label{Mukai's Lemma}
  Let $\mathcal{A}$ be the heart of a bounded $t$-structure of $\mathcal{D}^{b}(X)$. Let
  $$0 \longrightarrow F \longrightarrow E \longrightarrow G \longrightarrow 0 $$
  be a short exact sequence in $\mathcal{A}$ such that $\mbox{Hom}(F,G)=0$. Then
  $$\mbox{ext}^{1}(E,E)\geq \mbox{ext}^{1}(F,F)+\mbox{ext}^{1}(G,G).$$
\end{lemma}

\begin{corollary}
Using the notation of Lemma \ref{Mukai's Lemma}, if $E$ is rigid, then $F$ and $G$ are rigid.
\end{corollary} 

The next lemma shows that all possible $E_{j}/E_{i-1}$ that appear in the Harder-Narasimhan filtration at some stability condition are rigid.
 
\begin{lemma}\label{rigid}
  Let $E$ be a rigid object, and
  $$0=E_{0}\subset E_{1}\subset \cdots \subset E_{m}=E$$
  be the $\sigma$-Harder-Narasimhan filtration for some $\sigma\in \mbox{Stab}(X)$. Then $E_{j}/E_{i-1}$ is rigid for any $i\leq j$. 
\end{lemma}

\begin{proof}
  It suffices to prove $E_{m-1}/E_{0}$ and $E_{m}/E_{1}$ are rigid, then we may repeat this and the lemma is proved. Since this is a Harder-Narasimhan filtration,
  $\mbox{Hom}(E_{1},  G_{i})=0$ for $i\geq 2$, where $G_{i}=E_{i}/E_{i-1}$. Hence $\mbox{Hom}(E_{1},E_{m}/E_{1})=0$. By Mukai's Lemma (Lemma \ref{Mukai's Lemma}),
  $$0=\mbox{ext}^{1}(E,E)\geq \mbox{ext}^{1}(E_{1},E_{1})+\mbox{ext}^{1}(E_{m}/E_{1},E_{m}/E_{1}),$$
  $E_{m}/E_{1}$ is rigid. For $E_{m-1}/E_{0}$, the argument is similar.
\end{proof}

Hence what we still need is a local reduction process, which we describe here. Let $W$ be a wall and $\mathcal{H}$ be its rank 2 lattice. Let $\mathcal{A}_{\mathcal{H}}$ be the full subcategory of $\mathcal{A}$ that consists of objects whose Mukai vectors lie in $\mathcal{H}$. Let $\sigma_{\pm}$ be two stability conditions on two sides of $W$. Let $E\in \mathcal{A}_{\mathcal{H}}$ be a rigid object with $\omega_{1}$-degree positive. The objective of local reduction is to compute the $\sigma_{-}$-Harder-Narasimhan filtration of $E$, provided that we know the $\sigma_{+}$-Harder-Narasimhan filtration. We will develop the local reduction in Section \ref{section4} and Section \ref{section5} formally. Assuming local reduction, we may summarize the global reduction formally as follows.

\begin{algorithm}[Global reduction]\label{global reduction}
  Let $E\in M_{H}(v)$ be a spherical vector bundle with $c_{1}(E)\cdot H>0$. By Theorem \ref{large volume limit}, the initial shape is $(E,\partial W_{0})$, where $E$ is the trivial filtration, and $W_{0}$ is the largest wall of $v$ on $\mathbf{b}$.

  For the current shape $(E^{\bullet}, \partial W)$, let the rank 2 lattice of $W$ be $\mathcal{H}$.
  
  \begin{enumerate}
  \item Find the minimal index $i$ and maximal index $j$ such that $G_{i}=E_{i}/E_{i-1}$ and $G_{j}=E_{j}/E_{j-1}$ are in $\mathcal{H}$.
    \item Let $\sigma_{-}$ be a stability condition right below $W$. Do local reduction (Algorithm \ref{local reduction}) on $E_{j}/E_{i-1}$ to find the $\sigma_{-}$-Harder-Narasimhan filtration of $E_{j}/E_{i-1}$. It lifts to the $\sigma_{-}$-Harder-Narasimhan filtration of $E$, which we denote by $E_{-}^{\bullet}$. 
  \item Compute all numerical walls for all adjacent factors of $E_{-}^{\bullet}$, and compute all actual walls for all factors of $E_{-}^{\bullet}$ by Proposition \ref{criterion}, denote the union of them by $\{W_{l}\}$.
  \item Pick the maximal $W_{-}\in \{W_{l}\}$ with respect to the order on $\mathcal{W}(\mathbb{B})$. Let $(E_{-}^{\bullet}, W_{-})$ be the current shape.
    \item If $f_{0}(W_{-})<\sqrt{2/\omega_{1}^{2}}$, then by Corollary \ref{output}, the last Harder-Narasimhan factor of $E_{-}^{\bullet}$ is $\mathcal{O}_{X}[1]^{\oplus h}$ and $\mbox{h}^{1}(E)=h$, the algorithm terminates. Otherwise, do (1)-(5) again.
    \end{enumerate}
    By doing (1)-(5) each time, $\partial W_{-}<\partial W$ in $\mathcal{W}(\mathbb{B})$. Since $\mathcal{W}(\mathbb{B})$ is finite by Corollary \ref{finite}, the algorithm terminates in finitely many steps.
\end{algorithm}

For an example of global reduction (Algorithm \ref{global reduction}), see Example \ref{height three}.

\section{Wall Crossing for A Single Stable Spherical Object}\label{section4}

In this section we compute the Harder-Narasimhan filtration of a stable spherical object in an adjacent chamber. A good reference for background is \cite{BM14a}.

\subsection{Rank 2 Lattice}

As mentioned in Section \ref{section3}, to carry out the global reduction (Algorithm \ref{global reduction}), we need to know the Harder-Narasimhan walls. There are two types of walls: an actual wall for a factor, and a numerical wall defined by two adjacent factors. The latter type is easy to compute. In this section we deal with the former type. We show how to compute all actual walls for a spherical Mukai vector. 

Let $v\in  \mbox{H}^{*}_{alg}$ be a spherical Mukai vector. Since $M_{\sigma}(v)$ is a single point, every actual wall for $v$ is totally semistable \cite{BM14a}. Proposition 5.7 in \cite{BM14a} gives a numerical criterion for such walls when the associated rank 2 lattice is hyperbolic, however when $v$ is spherical, the proposition needs the following modification.

\begin{proposition} \label{criterion}
Let $\mathcal{H}\subset  \mbox{H}^{*}_{alg}$ be a rank 2 lattice (not neccesarily hyperbolic). Then $\mathcal{H}$ defines an actual wall for $v$ (hence totally semistable) if and only if there exists a spherical Mukai vector $v_{1}\in \mathcal{H}$ such that $vv_{1}<0$ and $v_{1}, v-v_{1}$ are both effective.
\end{proposition}

\begin{proof}
  If $\mathcal{H}$ defines an actual wall $W$, let $\sigma_{+}$ be a generic stability condition in an adjacent chamber of $W$. Then there is a unique object $E\in M_{+}(v)$. By definition, $E$ is destabilized at $\sigma_{0}\in W$. By Mukai's Lemma (Lemma \ref{Mukai's Lemma}), at $\sigma_{0}$ all Jordan-H\"{o}lder factors are effective $\sigma_{0}$-stable spherical objects, whose Mukai vectors are denoted by $v_{1}, \cdots , v_{n}$. Then $v=\sum_{i=1}^{n}a_{i}v_{i}$, $a_{i}>0$. By assumption $v$ is spherical, hence
  $$-2=v^{2}=v(\sum_{i=1}^{n}a_{i}v_{i})=\sum_{i=1}^{n}a_{i}vv_{i}.$$
  There must exist some $i$ such that $a_{i}vv_{i}<0$. Since $a_{i}>0$, we have $vv_{i}<0$. Furthermore, $v-v_{1}$ is clearly effective since $E$ is not $\sigma_{0}$-stable, and the Jordan-H\"{o}lder filtration is non-trivial.

  Conversely if there is some effective spherical $v_{1}$ with $vv_{1}<0$ and $v_{1}, v-v_{1}$ both being effective, then since $\sigma_{+}$ is generic, there is a unique $F\in M_{+}(v_{1})$. Since
  $$0<-vv_{1}=\chi(v,v_{1})=\mbox{hom}(E, F) -\mbox{ext}^{1}(E,F) +\mbox{ext}^{2}(E,F),$$
either $\mbox{Hom}(E,F)$ or $\mbox{Hom}(F,E)=\mbox{Ext}^{2}(E,F)^{*}$ is non-zero, say $\mbox{Hom}(F,E)$. Consider any non-zero map $F \rightarrow E$. First note that it cannot be surjective, otherwise $v-v_{1}$ is not effective, since $E, F \in \mathcal{A}$. Now let $F_{1}, \cdots , F_{n}$ be Jordan-H\"{o}lder factors of $F$ at $\sigma_{0}$, whose Mukai vectors are denoted by $v_{1}, \cdots , v_{n}$. By Mukai's Lemma (Lemma \ref{Mukai's Lemma}), they are all spherical. Then since $\mbox{Hom}(F,E)\neq 0$, there exists some $i$ such that $\mbox{Hom}(F_{i},E)\neq 0$. Since any map $F \rightarrow E$ is not surjective, the map $F_{i} \rightarrow E$ cannot be surjective. On the other hand, since $F_{i}$ is $\sigma_{0}$-stable and $\phi_{0}(E)=\phi_{0}(F_{i})$, the map $F_{i} \rightarrow E$ must be injective. Hence $E$ is destablized by $F_{i}$, and the wall is actual.  
\end{proof}

The next goal is to classify all such possible rank 2 lattices $\mathcal{H}$. Now if $\mathcal{H}$ is hyperbolic, then everything in \cite{BM14a} applies. Let $v_{1}$ be any Mukai vector as in Proposition \ref{criterion}, and let $-\chi=vv_{1}$. If $\chi \geq 3$, then $\mathcal{H}$ is hyperbolic. By Proposition \ref{criterion}, $\chi>0$. Hence there are two cases left: $\chi=1$ or $\chi=2$.

We first look at the case $\chi=2$. In this case $\eta'= v-v_{1}$ is an isotropic vector in $ \mathcal{H}$. Take any spherical $w \in \mathcal{H}$, then $w=av+bv_{1}$, $a,b \in \mathbb{Q}$. We have
$$-2=(av+bv_{1})^{2}=-2a^{2}-2b^{2}+2ab(-2)=-2(a+b)^{2},$$
hence $a+b=\pm 1$, $w=\pm v+t \eta'$ for some $t \in \mathbb{Q}$. Now let $F_{1}, F_{2}, F_{3}$ be three $\sigma_{+}$-stable spherical objects with Mukai vectors $w_{1}, w_{2}, w_{3}$. Then there exist two of them whose difference is isotropic, say $w_{1}, w_{2}$. Then
$$0=(w_{1}-w_{2})^{2}=-2-2w_{1}w_{2}-2,$$
namely $w_{1}w_{2}=-2$. Hence $F_{i}, i=1,2,3$, cannot all be $\sigma_{0}$-stable, there are at most two $\sigma_{0}$-stable spherical objects whose Mukai vectors lie in $\mathcal{H}$. On the other hand, there are at least two such, since $\mathcal{H}$ defines an actual wall. Hence there are exactly two, denoted by $s_{0}, t_{1}$, and let $\eta=s_{0}+t_{1}$. We have $s_{0}t_{1}=2$ and $\eta^{2}=0$.

In this case, there are infinitely many spherical classes in $\mathcal{H}$, and all of them still come from spherical reflections, as in \cite{BM14a}. For any $w=at_{1}+b\eta \in \mathcal{H}$, $a, b \in \mathbb{Q}$, since the Mukai pairing is even, we have
$$w^{2}=a^{2}t_{1}^{2}=-2a^{2}\in 2 \mathbb{Z},$$
hence $a\in \mathbb{Z}$, $b\eta \in \mathcal{H}$. If $b\notin \mathbb{Z}$, since $\eta \in  \mathcal{H}$, there exists $0< b' <1$ such that $b'\eta \in \mathcal{H}$. Then $t_{1}'=t_{1}-b'\eta$ is an effective class. Let $T_{1}' \in M_{+}(t_{1}')$ and $T_{1}\in M_{0}(t_{1})$, then $\chi(T_{1}, T_{1}')=-t_{1}t_{1}'=2$, either $\mbox{Hom}(T_{1},T_{1}')$ or $\mbox{Hom}(T_{1}',T_{1})$ is non-zero, say $\mbox{Hom}(T_{1}',T_{1})$. Since $t_{1}-t_{1}'=b'\eta$ is effective, the map $T_{1}' \rightarrow T_{1}$ cannot be surjective. On the other hand since $T_{1}$ is $\sigma_{0}$-stable and $\phi_{0}(T_{1}')=\phi_{0}(T_{1})$, the image of this non-zero map must be $T_{1}$ itself, a contradiction. Hence $b\in \mathbb{Z}$, we see $s_{0}, t_{1}$ is an integral basis of $\mathcal{H}$. Then all effective spherical classes are given recursively: $t_{2}=s_{0}+(s_{0}t_{1})t_{1}$, $t_{i}=-(t_{i-1}t_{i-2})t_{i-1}-t_{i-2}$ for $i\geq 3$. Similarly $s_{-1}=t_{1}+(t_{1}s_{0})s_{0}$, $s_{j}=-s_{j+2}-(s_{j+2}s_{j+1})s_{j+1}$ for $j\leq -2$.

We are left with the case where $\chi=1$. In this case $\mathcal{H}$ is negative definite. We claim there are exactly 3 effective spherical objects whose Mukai vectors lie in $\mathcal{H}$, among which exactly two are $\sigma_{0}$-stable. Note that the discriminant
$$\mbox{disc}(v,v_{1})=\begin{pmatrix}
                         v^{2} & v_{1}v \\
                         vv_{1} & v_{1}^{2}
                       \end{pmatrix}
                       =4-1=3$$
is square free, hence $v,v_{1}$ is a basis for $\mathcal{H}$.                   

Now suppose $w=xv+yv_{1}$, $x,y \in \mathbb{Z}$ is spherical, then
$$-2=w^{2}=(xv+yv_{1})^{2}=-2x^{2}-2y^{2}-2xy=-x^{2}-y^{2}-(x+y)^{2}.$$
We have either $x=0$, $y=0$, or $x+y=0$. Solving this we get
\begin{eqnarray*}
  x=0, y=\pm 1 \\
  x=\pm 1, y=0 \\
  x=1, y=-1\\
  x=-1, y=1
\end{eqnarray*}
Hence there are exactly 6 spherical Mukai vectors, with 3 of them being effective. By assumption $\mathcal{H}$ defines a wall for $v$, hence $E$ is not $\sigma_{0}$-stable. However there are at least two effective $\sigma_{0}$-stable objects, hence there are exactly two. Denote them by $S_{0}, T_{1}$ whose Mukai vectors are $s_{0}, t_{1}$. In this case we see $v=s_{0}+t_{1}$. Explicitly, $E\in M_{+}(v)$ is a general extension
\[
0 \longrightarrow S_{0} \longrightarrow E \longrightarrow T_{1} \longrightarrow 0.
\]

We summarize the section by separate the rank 2 lattices into two cases in the following proposition.

\begin{proposition}\label{lattice}
  Let $\mathcal{H}$ be the rank 2 lattice associated to a wall of $v$, whose Mukai vectors of the two stable spherical objects are $s_{0}, t_{1}$. Let $v_{1}$ be any Mukai vector as in Proposition \ref{criterion}, and let $\chi=-vv_{1}$. Then
\begin{enumerate}
\item  If $\chi \geq 2$, then there are infinitely many effective spherical Mukai vectors in $\mathcal{H}$, constructed recursively by
  \begin{align*}
    t_{2}=s_{0}+(s_{0}t_{1})t_{1}, ~ t_{i}=-(t_{i-1}t_{i-2})t_{i-1}-t_{i-2} \mbox{ for } i\geq 3,\\
    s_{-1}=t_{1}+(t_{1}s_{0})s_{0}, ~ s_{j}=-s_{j+2}-(s_{j+2}s_{j+1})s_{j+1} \mbox{ for } j\leq -2.
  \end{align*}
\item If $\chi =1$, then there are exactly 6 spherical Mukai vectors in $\mathcal{H}$, with 3 of them being effective, which are $s_{0}, t_{1}, v=s_{0}+t_{1}$. Let $E\in M_{\sigma_{+}}(v)$, then $E$ fits into the unique non-trivial extension
  $$0 \longrightarrow S_{0} \longrightarrow E \longrightarrow T_{1} \longrightarrow 0 .$$ 
\end{enumerate}
\end{proposition}

Note that item (2) in Proposition \ref{lattice} can happen. See Example \ref{ndlattice}.

\subsection{Construction of Spherical Objects}

Let $v$ be a spherical Mukai vector. Let $W$ be a wall for $v$, and $\mathcal{H}$ its rank 2 lattice. Let $\sigma_{0}\in W$, and $\sigma_{+}, \sigma_{-}$ be two sufficiently close stability conditions near $\sigma_{0}$ on two sides of $W$. Then there are two stable spherical objects $S_{0}, T_{1}$ on $W$. We assume that $\phi_{+}(T_{1})>\phi_{+}(S_{0})$.

In this section we explicitly construct all $\sigma_{\pm}$-stable spherical objects whose Mukai vectors are in $\mathcal{H}$ from $S_{0}, T_{1}$. From this construction, we get the $\sigma_{-}$(resp. $\sigma_{+}$)-Harder-Narasimhan filtration of a $\sigma_{+}$(resp. $\sigma_{-}$)-stable spherical object (Corollary \ref{JH}). 

If $\mathcal{H}$ is negative definite, then the description is already given in Proposition \ref{lattice}. Hence in the following we assume $\mathcal{H}$ is degenerate or hyperbolic. By Proposition \ref{lattice} there are infinitely many $\sigma_{+}$-stable objects whose Mukai vectors are in $\mathcal{H}$. Let $T_{i}\in M_{\sigma_{+}}(t_{i}), i\geq 1$ and $S_{j}\in M_{\sigma_{+}}(s_{j}), j\leq 0$.

\begin{proposition} \label{construction}
The $T_{i}$ are constructed as follows:
\begin{gather*}
  0 \longrightarrow S_{0} \longrightarrow T_{2} \longrightarrow T_{1}\otimes \mbox{Ext}^{1}(T_{1}, S_{0}) \longrightarrow 0, \\
    0 \longrightarrow T_{m+1} \longrightarrow T_{m}\otimes \mbox{Hom}(T_{m},T_{m-1}) \longrightarrow T_{m-1}  \longrightarrow 0, m\geq 2.
    \end{gather*}
    where $T_{2}$ is a general extension, and $T_{i}\otimes \mbox{Hom}(T_{i}, T_{i-1}) \longrightarrow T_{i-1}$ are evaluation maps. \\
    Similarly, $S_{j}$ are constructed as follows: 
\begin{gather*}
  0 \longrightarrow S_{0}\otimes \mbox{Ext}^{1}(T_{1}, S_{0}) \longrightarrow S_{-1} \longrightarrow T_{1} \longrightarrow 0,\\
  0 \longrightarrow S_{m+1} \longrightarrow S_{m}\otimes \mbox{Hom}(S_{m+1}, S_{m})^{*} \longrightarrow S_{m-1} \longrightarrow 0, m\leq -1.
\end{gather*}
where $S_{-1}$ is a general extension, and $S_{j+1} \longrightarrow S_{j} \otimes \mbox{Hom}(S_{j+1}, S_{j})^{*} $ are coevaluation maps.
\end{proposition}

\begin{proof}
  We only prove the statement for $T_{i}$, $S_{j}$ are similar. First note that the Mukai vectors of $T_{i}$'s are spherical, hence it suffices to prove $T_{i}$ are $\sigma_{+}$-stable. We start with $T_{2}$. It is in $\mathcal{A}$, since it is an extension of two objects in $\mathcal{A}$.
  Applying $\mbox{Hom}(T_{2},-)$ to the sequence
  $$ 0 \longrightarrow S_{0} \longrightarrow T_{2} \longrightarrow T_{1}\otimes \mbox{Ext}^{1}(T_{1}, S_{0}) \longrightarrow 0$$
  we get
  $$\mbox{Ext}^{1}(T_{2},S_{0}) \rightarrow \mbox{Ext}^{1}(T_{2}, T_{2}) \rightarrow \mbox{Ext}^{1}(T_{2}, T_{1})\otimes \mbox{Ext}^{1}(T_{1},S_{0}).$$
  To show $\mbox{Ext}^{1}(T_{2},T_{2})=0$, it suffices to show $\mbox{Ext}^{1}(T_{2},T_{1})=\mbox{Ext}^{1}(T_{2},S_{0})=0$. Applying $\mbox{Hom}(T_{1},-)$ to the same sequence, we have
  $$\mbox{Hom}(T_{1}, T_{1})\otimes \mbox{Ext}^{1}(T_{1}, S_{0}) \overset{\delta}{\rightarrow} \mbox{Ext}^{1}(T_{1},S_{0}) \rightarrow \mbox{Ext}^{1}(T_{1},T_{2}) \rightarrow \mbox{Ext}^{1}(T_{1}, T_{1}) \otimes \mbox{Ext}^{1}(T_{1}, S_{0})=0.$$
  Note that $\delta$ is an isomorphism, $\mbox{Ext}^{1}(T_{1}, T_{2})=0$. By a similar argument we also have $\mbox{Ext}^{1}(S_{0}, T_{2})=0$. Hence $\mbox{Ext}^{1}(T_{2}, T_{2})=0$.
  
  Let $T$ be the first Jordan-H\"{o}lder factor of the first Harder-Narasimhan factor, by Mukai's Lemma (Lemma \ref{Mukai's Lemma}), $T$ is $\sigma_{+}$-stable spherical object. Hence $\phi_{+}(T)\geq \phi_{+}(T_{2})$. The only possibilities are $v(T)=v(T_{1})$ or $v(T)=v(T_{2})$. By stability of $T$, any non-zero map $T \rightarrow T_{2}$ is an injection. If $v(T)=v(T_{2})$, then it is an isomorphism. If $v(T)=v(T_{1})$, then since $T_{1}$ is $\sigma_{+}$-stable, $T=T_{1}$. Now apply $\mbox{Hom}(T_{1},-)$ to
  $$0 \longrightarrow S_{0} \longrightarrow T_{2} \longrightarrow T_{1}\otimes \mbox{Ext}^{1}(T_{1}, S_{0}) \longrightarrow 0, $$
  we get
  $$ 0=\mbox{Hom}(T_{1}, S_{0}) \rightarrow \mbox{Hom}(T_{1},T_{2}) \rightarrow \mbox{Hom}(T_{1}, T_{1}) \otimes \mbox{Ext}^{1}(T_{1}, S_{0}) \overset{\delta'}{\rightarrow} \mbox{Hom}(T_{1}, S_{0}[1]),$$
  where $\delta'$ is an isomorphism. Hence $\mbox{Hom}(T_{1},T_{2})=0$, $T=T_{2}$, and therefore $T_{2}$ is $\sigma_{+}$-stable.
  
  Suppose we have proven that for $n\leq i$, the $T_{n}$ that are constructed as in the claim are $\sigma_{+}$-stable. Now consider the evaluation map
  $$T_{i}\otimes \mbox{Hom}(T_{i}, T_{i-1}) \longrightarrow T_{i-1}$$
  I claim that this map is surjective if $\mathcal{H}$ is hyperbolic or degenerate.
  Denote the image by $I$ and assume for contradiction that $I\neq T_{i-1}$. Let $I_{1}$ be the first Jordan-H\"{o}lder factor at $\sigma_{+}$. By Mukai's Lemma (Lemma \ref{Mukai's Lemma}), $I_{1}$ is $\sigma_{+}$-stable spherical object. If $I_{1}\neq T_{i-1}$, by semistability of $T_{i}\otimes \mbox{Hom}(T_{i}, T_{i-1})$ and stability of $T_{i-1}$, we have
  $$\phi_{+}(T_{i-1})>\phi_{+}(I_{1})\geq \phi_{+}(I) \geq \phi_{+}(T_{i})$$
  Since there is no spherical Mukai vectors in $\mathcal{H}$ whose $\sigma_{+}$-phase is strictly in between $\phi_{+}(T_{i})$ and $\phi_{+}(T_{i-1})$, the equalities have to hold, namely $I$ is semistable whose Jordan-H\"{o}lder factors all have Mukai vectors equal to $v(T_{i})$. Since $M_{+}(v(T_{i}))=\{T_{i}\}$, all Jordan-H\"{o}lder factors of $I$ are $T_{i}$. Since $\mbox{Ext}^{1}(T_{i}, T_{i})=0$, $I=T_{i}^{\oplus a}$. Since $\mathcal{H}$ is hyperbolic or degenerate, $v(T_{i})\mbox{hom}(T_{i}, T_{i-1})-v(T_{i-1})=v(T_{i+1})$ is effective, $a<\mbox{hom}(T_{i}, T_{i-1})$. This is impossible, since $T_{i}\otimes \mbox{Hom}(T_{i}, T_{i-1}) \rightarrow T_{i-1}$ is the evaluation map. If $a< \mbox{hom}(T_{i}, T_{i-1})$ then there is a non-zero $f\in \mbox{Hom}(T_{i}, T_{i-1})$ such that $T_{i}\otimes f$ is mapped to zero in $T_{i}^{\oplus a}$, a contradiction. Hence $I=I_{1} = T_{i-1}$, the evaluation map $T_{i} \otimes \mbox{Hom}(T_{i}, T_{i-1}) \rightarrow T_{i-1}$ is surjective. Hence $T_{i+1}$ is the kernel of this map, which is in $\mathcal{A}$, since $\mathcal{A}$ is abelian.

  Now we need to prove $T_{i+1}$ is $\sigma_{+}$-stable. First I claim $\mbox{Ext}^{1}(T_{i+1}, T_{i+1})=0$. Applying $\mbox{Hom}(T_{i}, -)$ to the exact sequence
  $$0 \longrightarrow T_{i+1} \longrightarrow T_{i} \otimes \mbox{Hom}(T_{i},T_{i-1}) \longrightarrow T_{i-1} \longrightarrow 0 $$
  we get
  $$\mbox{Hom}(T_{i}, T_{i}) \otimes \mbox{Hom}(T_{i},T_{i-1}) \rightarrow \mbox{Hom}(T_{i},T_{i-1}) \rightarrow \mbox{Ext}^{1}(T_{i},T_{i+1}) \rightarrow \mbox{Ext}^{1}(T_{i},T_{i}) \otimes \mbox{Hom}(T_{i},T_{i-1}).$$
  the first map is an isomorphism, and $\mbox{Ext}^{1}(T_{i},T_{i})=0$ by the induction hypothesis, hence $\mbox{Ext}^{1}(T_{i+1}, T_{i})=\mbox{Ext}^{1}(T_{i},T_{i+1})^{*}=0$.

  Applying $\mbox{Hom}(-,T_{i})$ we get
  $$0=\mbox{Hom}(T_{i-1},T_{i}) \rightarrow \mbox{Hom}(T_{i},T_{i-1})^{*} \rightarrow \mbox{Hom}(T_{i+1}, T_{i}) \rightarrow \mbox{Ext}^{1}(T_{i}, T_{i-1})=0.$$
  Hence there is a natural identification
  $$\mbox{Hom}(T_{i}, T_{i-1})^{*}=\mbox{Hom}(T_{i+1},T_{i}).$$
  The map of functors $\mbox{Hom}(-, T_{i}\otimes \mbox{Hom}(T_{i},  T_{i-1})) \rightarrow \mbox{Hom}(-,T_{i-1})$ applied to the evaluation map $T_{i+1} \longrightarrow T_{i}\otimes \mbox{Hom}(T_{i},T_{i-1})^{*}$ gives rise to a commutative diagram
  \\
  
  \begin{adjustbox}{scale=0.9, center}
\begin{tikzcd}
  \mbox{Hom}(T_{i},T_{i})\otimes \mbox{Hom}(T_{i},T_{i-1})\otimes \mbox{Hom}(T_{i},T_{i-1})^{*} \arrow{r}\arrow{d} & \mbox{Hom}(T_{i+1}, T_{i})\otimes \mbox{Hom}(T_{i},T_{i-1}) \arrow{r} \arrow{d} & 0\\
   \mbox{Hom}(T_{i},T_{i-1}) \otimes \mbox{Hom}(T_{i},T_{i-1})^{*} \arrow{r}\arrow{d} & \mbox{Hom}(T_{i+1},T_{i-1}) \arrow{r}\arrow{d} & 0\\
   0 & \mbox{Ext}^{1}(T_{i+1},T_{i+1}) \arrow[d] & \\
   & \mbox{Ext}^{1}(T_{i+1},T_{i})\otimes \mbox{Hom}(T_{i},T_{i-1})=0 & 
\end{tikzcd}
\end{adjustbox}
\\

Hence $\mbox{Ext}^{1}(T_{i+1},T_{i+1})=0$. Let $T$ be the first Jordan-H\"{o}lder factor of the first Harder-Narasimhan factor of $T_{i+1}$, by Mukai's Lemma (Lemma \ref{Mukai's Lemma}), $T$ is $\sigma_{+}$-stable spherical object. If $v(T)=v(T_{i+1})$, then by stability of $T$ we have $T=T_{i+1}$. Otherwise, $T=T_{j}$ for some $j\leq i-1$, by stability $\mbox{Hom}(T_{j}, T_{i})=0$, and 
$$0=\mbox{Ext}^{-1}(T_{j}, T_{i-1}) \rightarrow \mbox{Hom}(T_{j},T_{i+1}) \rightarrow \mbox{Hom}(T_{j},T_{i}) \otimes \mbox{Hom}(T_{i},T_{i-1})=0.$$
Hence $\mbox{Hom}(T_{j},T_{i+1})=0$, a contradiction. Hence $T_{i+1}$ is $\sigma_{+}$-stable. By induction, the $T_{n}$ that are constructed in the claim are $\sigma_{+}$-stable for all $n\geq 1$.

\end{proof}

\begin{remark}\label{coevaluation}
The map $T_{i+1} \rightarrow T_{i}\otimes \mbox{Hom}(T_{i}, T_{i-1})$ is the coevaluation map.
\end{remark}

\begin{proof}
  Let $V$ be a vector space. In general, if we have a map $A \rightarrow B\otimes V$, then to check it is the coevaluation map, we need to check the map $$\mbox{id}\otimes V^{*}\subset \mbox{Hom}(B,B) \otimes V^{*}=\mbox{Hom}(B\otimes V, B) \rightarrow \mbox{Hom}(A, B)$$ is an isomorphism. We apply $\mbox{Hom}(-, T_{i})$ to the sequence in Proposition \ref{construction} and get
  $$0=\mbox{Hom}(T_{i-1}, T_{i}) \rightarrow \mbox{Hom}(T_{i}, T_{i}) \otimes \mbox{Hom}(T_{i}, T_{i-1})^{*} \rightarrow \mbox{Hom}(T_{i+1}, T_{i}) \rightarrow \mbox{Ext}^{1}(T_{i-1}, T_{i})$$
  By the proof of Proposition \ref{construction}, $\mbox{Ext}^{1}(T_{i-1}, T_{i})=0$. Hence the remark is proved.
\end{proof}

Then there is a natural identification $\epsilon_{i}: \mbox{Hom}(T_{i+1}, T_{i})^{*}\overset{\cong}{\longrightarrow}\mbox{Hom}(T_{i}, T_{i-1})$. We shall denote
$$m_{i}: \mbox{Hom}(T_{i}, T_{i-1}) \otimes \mbox{Hom}(T_{i+1}, T_{i}) \overset{\mbox{id}\otimes (\epsilon_{i}^{*})^{-1}}{\longrightarrow } \mbox{Hom}(T_{i}, T_{i-1})\otimes \mbox{Hom}(T_{i},T_{i-1})^{*}\rightarrow \mathbb{C}$$
where the second map is the natural pairing. There is also a comultiplication induced by the sequence itself. Since $T_{i+1} \rightarrow T_{i}\otimes \mbox{Hom}(T_{i}, T_{i-1})$ is the coevaluation map, applying $\mbox{Hom}(T_{i+1},-)$ to this sequence gives us a map
$$\delta_{i}: \mathbb{C}=\mbox{Hom}(T_{i+1},T_{i+1}) \rightarrow \mbox{Hom}(T_{i+1},T_{i})\otimes \mbox{Hom}(T_{i},T_{i-1}).$$
The identifications here satisfy the following commutativity:
\\
\begin{lemma}\label{identification}
  Let $\delta: \mathbb{C} \rightarrow \mbox{Hom}(T_{i},T_{i-1})^{*}\otimes \mbox{Hom}(T_{i},T_{i-1})$ be the diagonal map. Then $(\epsilon_{i}\otimes \mbox{id})\circ \delta = \delta_{i}$, namely there is a commutative diagram
  \[
\begin{tikzcd}
  \mathbb{C} \arrow{rr}{\delta}\arrow{dr}{\delta_{i}} & & \mbox{Hom}(T_{i},T_{i-1})^{*}\otimes \mbox{Hom}(T_{i},T_{i-1})\arrow{dl}{\epsilon_{i}^{*}\otimes \mbox{id}} \\
  & \mbox{Hom}(T_{i+1},T_{i}) \otimes \mbox{Hom}(T_{i},T_{i-1}) & 
\end{tikzcd}
\]
Dually, let $m: \mbox{Hom}(T_{i},T_{i-1})^{*}\otimes \mbox{Hom}(T_{i},T_{i-1}) \rightarrow \mathbb{C}$ be the natural pairing. Then $m\circ (\mbox{id}\otimes \epsilon_{i}) = \delta_{i}^{*}$, namely there is a commutative diagram
\[
\begin{tikzcd}
  \mbox{Hom}(T_{i},T_{i-1})^{*}\otimes \mbox{Hom}(T_{i+1},T_{i})^{*} \arrow{rr}{\mbox{id}\otimes \epsilon_{i}} \arrow{dr}{\delta_{i}^{*}} & & \mbox{Hom}(T_{i},T_{i-1})^{*}\otimes \mbox{Hom}(T_{i},T_{i-1}) \arrow{dl}{m} \\
  & \mathbb{C} & 
\end{tikzcd}
\]
\end{lemma}

\begin{proof}
  We prove the first claim, the second follows by dualizing. Perhaps the most direct way to prove this is to calculate with an explicit basis. Let $e_{1}, \cdots , e_{n}$ be a basis for $\mbox{Hom}(T_{i},T_{i-1})$. Applying $\mbox{Hom}(-, T_{i}\otimes \mbox{Hom}(T_{i},T_{i-1}))$ to $\Phi: T_{i+1} \rightarrow T_{i}\otimes \mbox{Hom}(T_{i},T_{i-1})$ gives the isomorphism
  $$\phi: \mbox{Hom}(\mbox{Hom}(T_{i},T_{i-1}), \mbox{Hom}(T_{i},T_{i-1})) \rightarrow \mbox{Hom}(T_{i+1},T_{i})\otimes \mbox{Hom}(T_{i},T_{i-1}). $$
  Since $\Phi: T_{i+1} \rightarrow T_{i}\otimes \mbox{Hom}(T_{i},T_{i-1})$ is the coevaluation map, we may write $\Phi \in \mbox{Hom}(T_{i+1},T_{i})\otimes \mbox{Hom}(T_{i},T_{i-1})$ as $\sum f_{k}\otimes e_{k}$, where $f_{1}, \cdots , f_{n}$ is a basis for $\mbox{Hom}(T_{i+1}, T_{i})$. Suppose $A= (a_{kl})_{1\leq k,l \leq n}\in \mbox{Hom}(\mbox{Hom}(T_{i},T_{i-1}),\mbox{Hom}(T_{i},T_{i-1}))$ is a matrix, then it is naturally identified with $\sum a_{kl}e_{l}^{*}\otimes e_{k}$ in $\mbox{Hom}(T_{i},T_{i-1})^{*}\otimes \mbox{Hom}(T_{i},T_{i-1})$. By definition, $\phi(A)=A \circ \Phi \in \mbox{Hom}(T_{i+1}, T_{i})\otimes \mbox{Hom}(T_{i},T_{i-1})$, explicitly
  $$A\circ \Phi=\sum_{l} f_{l}\otimes Ae_{l}= \sum_{k,l}f_{l}\otimes a_{kl}e_{k}=\sum_{k,l}a_{kl}f_{l}\otimes e_{k},$$
  hence the identification $\epsilon_{i}^{*}: \mbox{Hom}(T_{i},T_{i-1})^{*} \rightarrow \mbox{Hom}(T_{i+1},T_{i})$ sends $e_{l}^{*}$ to $f_{l}$ for all $1\leq l \leq n$.

  On the other hand, by definition applying $\mbox{Hom}(T_{i+1},-)$ to $\Phi: T_{i+1} \rightarrow T_{i}\otimes \mbox{Hom}(T_{i},T_{i-1})$ sends $\mbox{id}\in \mbox{Hom}(T_{i+1},T_{i+1})$ to $\Phi\in \mbox{Hom}(T_{i+1}, T_{i})\otimes \mbox{Hom}(T_{i},T_{i-1})$, namely $\delta_{i}(\mbox{id})=\Phi$. We have
  $$(\epsilon_{i}^{*}\otimes \mbox{id})\circ \delta (1)=(\epsilon_{i}^{*}\otimes \mbox{id})(\sum e_{k}^{*}\otimes e_{k})=\sum f_{k}\otimes e_{k}= \Phi = \delta_{i}(\mbox{id}), $$
  the lemma is proved.

\end{proof}

The next goal is to compute homomorphisms between $T_{i}$ and $T_{1}$ or $S_{0}$.

\begin{lemma}\label{commutativity}
  Consider the maps
  \begin{gather*}coev_{i+1}: T_{i+1} \rightarrow T_{i}\otimes \mbox{Hom}(T_{i+1}, T_{i})^{*}, \\
    ev_{i+1}: T_{i+1}\otimes \mbox{Hom}(T_{i+1}, T_{i}) \rightarrow T_{i}.
  \end{gather*}
  Denote $m: \mbox{Hom}(T_{i+1}, T_{i})^{*}\otimes \mbox{Hom}(T_{i+1}, T_{i}) \rightarrow \mathbb{C}$ the natural pairing. Then we have
  $$(\mbox{id}\otimes m) \circ (coev_{i+1}\otimes \mbox{id}) = ev_{i+1}.$$ In other words, we have the following commutative diagram
  \[
\begin{tikzcd}
  T_{i+1}\otimes \mbox{Hom}(T_{i+1}, T_{i}) \arrow{rr}{coev_{i+1}\otimes \mbox{id}}\arrow{rd}{ev_{i+1}} & & T_{i}\otimes \mbox{Hom}(T_{i+1}, T_{i})^{*}\otimes \mbox{Hom}(T_{i+1}, T_{i}) \arrow{ld}{\mbox{id}\otimes m}\\
   & T_{i} &
\end{tikzcd}
  \]
\end{lemma}

\begin{proof}
  Let $f_{1}, \cdots , f_{n}$ be a basis of $\mbox{Hom}(T_{i+1}, T_{i})$, and let $f_{1}^{*}, \cdots , f_{n}^{*}$ be the dual basis. Take any non-zero $f\in \mbox{Hom}(T_{i+1}, T_{i})$. Restricting to $T_{i+1} \otimes f$, we have
  \begin{gather*}
    coev_{i+1}\otimes \mbox{id}=f_{1}(-)\otimes f_{1}^{*}\otimes f + \cdots + f_{n}(-)\otimes f_{n}^{*}\otimes f, \\
    (\mbox{id}\otimes m)(f_{1}(-)\otimes f_{1}^{*}\otimes f + \cdots + f_{n}(-)\otimes f_{n}^{*}\otimes f)=f(-).
  \end{gather*}
  On the other hand, the restriction of $ev_{i+1}$ on $T_{i+1}\otimes f$ is equal to $f$ by definition.
\end{proof}

Note that $\mbox{Hom}(T_{1}, T_{i})=0$ for $i\geq 2$, $\mbox{Hom}(T_{1}, S_{j})=0$ for all $j$. Hence the only relevant functors are $\mbox{Ext}^{1}(T_{1},-)$ and $\mbox{Hom}(-, T_{1})=\mbox{Ext}^{2}(T_{1},-)^{*}$.

We compute $\mbox{Ext}^{1}(T_{1},T_{i})$ for $1\leq i \leq 4$ by directly chasing the long exact sequences.
For $i=1$, $\mbox{Ext}^{1}(T_{1},T_{1})=0$ by definition. For $i=2$, we also have $\mbox{Ext}^{1}(T_{1},T_{2})=0$.\\
For $i=3$, we have $\mbox{Ext}^{1}(T_{1},T_{3})\cong \mbox{Hom}(T_{1}, T_{1})=\mathbb{C}$. For $i=4$, we have $\mbox{Ext}^{1}(T_{1},T_{4})=\mbox{Hom}(T_{3}, T_{2})$. For larger $i$, we have the following

\begin{proposition}\label{description}
  For $i\geq 3$,
  $$\mbox{Ext}^{1}(T_{1},T_{i+1})=\bigcap_{j=3}^{i-1}\mbox{ker}(m_{j})\subset \bigotimes_{j=3}^{i}\mbox{Hom}(T_{j}, T_{j-1}),$$
  where by abuse of notation $m_{j}$ denotes $\mbox{id}\otimes \cdots \otimes \mbox{id}\otimes m_{j} \otimes \mbox{id}\otimes \cdots \otimes \mbox{id}$.
\end{proposition}

\begin{proof}
  We use induction on $i$. The proposition is true for $i=3$. Suppose it is proven for $3\leq n\leq i$, and denote the embedding $\mbox{Ext}^{1}(T_{1}, T_{i+1})\subset \bigotimes_{j=3}^{i}\mbox{Hom}(T_{j}, T_{j-1})$ by $\iota$. We have an exact sequence
  $$0 \longrightarrow \mbox{Ext}^{1}(T_{1},T_{i+1}) \longrightarrow \mbox{Ext}^{1}(T_{1},T_{i})\otimes \mbox{Hom}(T_{i}, T_{i-1}) \overset{ev_{i}}{\longrightarrow}\mbox{Ext}^{1}(T_{1},T_{i-1}). $$
\\
Hence $\mbox{Ext}^{1}(T_{1},T_{i+1})$ is identified with
$$\mbox{ker}(ev_{i}) \subset \mbox{Ext}^{1}(T_{1},T_{i})\otimes \mbox{Hom}(T_{i}, T_{i-1}) \subset \bigotimes_{j=3}^{i}\mbox{Hom}(T_{j}, T_{j-1}).$$
By induction, an element in $\bigotimes_{j=3}^{i}\mbox{Hom}(T_{j}, T_{j-1})$ is inside $\mbox{Ext}^{1}(T_{1},T_{i})\otimes \mbox{Hom}(T_{i}, T_{i-1})$ precisely if it is inside  $\mbox{ker}(m_{j})$ for $3\leq j\leq i-2$. On the other hand, we have a commutative diagram
\\

\begin{adjustbox}{scale=0.9, center}
\begin{tikzcd}
  \mbox{Ext}^{1}(T_{1},T_{i})\otimes \mbox{Hom}(T_{i}, T_{i-1}) \arrow{r}{ev_{i}}\arrow{d}{((\mbox{id}\otimes \epsilon_{i-1})\circ coev_{i})\otimes \mbox{id}} & \mbox{Ext}^{1}(T_{1},T_{i-1})\arrow{d}{\mbox{id}} \\
  \mbox{Ext}^{1}(T_{1},T_{i-1})\otimes \mbox{Hom}(T_{i-1}, T_{i-2})\otimes \mbox{Hom}(T_{i}, T_{i-1})\arrow{r}{\mbox{id}\otimes m_{i-1}}\arrow{d}{\iota\otimes \mbox{id}\otimes \mbox{id}} & \mbox{Ext}^{1}(T_{1},T_{i-1}) \arrow{d}{\iota}\\
  \bigotimes_{j=3}^{i}\mbox{Hom}(T_{j}, T_{j-1}) \arrow{r}{m_{i-1}} & (\bigotimes_{j=3}^{i-2}\mbox{Hom}(T_{j}, T_{j-1})) 
\end{tikzcd}
\end{adjustbox}
\\
\\
where commutativity of the top square is by Lemma \ref{commutativity} and the definition of $m_{i-1}$, and commutativity of the bottom square is clear. Since all vertical maps are inclusions, being mapped to zero by $ev_{i}$ is equivalent to being inside $\mbox{ker}(m_{i-1})$. The theorem is proved.
  
\end{proof}

We want to compute the dimension of $\mbox{Ext}^{1}(T_{1},T_{i+1})$. By the long exact sequence
$$\mbox{Ext}^{1}(T_{1}, T_{i}\otimes \mbox{Hom}(T_{i}, T_{i-1})) \rightarrow \mbox{Ext}^{1}(T_{1}, T_{i-1}) \rightarrow \mbox{Ext}^{2}(T_{1}, T_{i+1}) \rightarrow \mbox{Ext}^{2}(T_{1}, T_{i}\otimes \mbox{Hom}(T_{i}, T_{i-1})),$$
$\mbox{ext}^{1}(T_{1},T_{i+1})$ can be computed by an inductive formula if we know the surjectivity of
$\mbox{Ext}^{1}(T_{1}, T_{i}\otimes \mbox{Hom}(T_{i}, T_{i-1})) \longrightarrow \mbox{Ext}^{1}(T_{1}, T_{i-1})$, which is equivalent to the surjectivity of 
$$\mbox{Hom}(T_{i} \otimes \mbox{Hom}(T_{i+1}, T_{i})^{*}, T_{1})=\mbox{Hom}(T_{i+1}, T_{i})\otimes \mbox{Hom}(T_{i}, T_{1})\rightarrow \mbox{Hom}(T_{i+1}, T_{1}).$$
By Remark \ref{coevaluation}, this is the composition map, since it is the dual map of the coevaluation map. By induction, we have a surjection
$$\mbox{Hom}(T_{i+1}, T_{i})\otimes \mbox{Hom}(T_{i},T_{i-1}) \otimes \cdots \otimes \mbox{Hom}(T_{2}, T_{1}) \longrightarrow \mbox{Hom}(T_{i+1}, T_{1}),$$

By an argument similar to the one in Theorem \ref{description}, the kernel of this map is generated by all diagonals $\delta_{j}$, $2\leq j \leq i$.

\begin{proposition} \label{description2}
  For $i\geq 2$,
  $$\bigotimes_{j=i}^{1} \mbox{Hom}(T_{j+1},T_{j})\twoheadrightarrow \mbox{Hom}(T_{i+1}, T_{1})=\mbox{coker}(\sum_{j=i}^{2}\delta_{j}),$$
  where by abuse of notation $\delta_{j}$ denotes $\mbox{id}\otimes \cdots \otimes \mbox{id}\otimes \delta_{j} \otimes \mbox{id}\otimes \cdots \otimes \mbox{id}$.
\end{proposition}

Hence now we know the dimension of $\mbox{Ext}^{1}(T_{1},T_{i+1})$ and $\mbox{Hom}(T_{i+1},T_{1})$ by an inductive formula. We make the following definition.

\begin{definition}\label{fundamental sequence}
  Let $W$ be a wall whose corresponding rank 2 lattice is $\mathcal{H}$. Let $S_{0}, T_{1}$ be the two $\sigma_{0}$-stable objects in $\mathcal{H}$ for general $\sigma_{0}\in \mathcal{H}$.
  Write $g(\mathcal{H})=\mbox{ext}^{1}(T_{1},S_{0})$. The following sequence $\{a_{n}(\mathcal{H})\}_{n\in \mathbb{Z}_{\geq 0}}$ completely determined by $\mathcal{H}$ is called the \emph{fundamental sequence} of $\mathcal{H}$:
  $$a_{0}(\mathcal{H})=1,~ a_{1}(\mathcal{H})=g(\mathcal{H}),~ a_{n}(\mathcal{H})=g(\mathcal{H})\cdot a_{n-1}(\mathcal{H})-a_{n-2}(\mathcal{H}), n\geq 2.$$
  We will write $g$ for $g(\mathcal{H})$ and $a_{n}$ for $a_{n}(\mathcal{H})$ if the context is clear.
\end{definition}

Then the argument above and a symmetric argument yield the following theorem.

\begin{theorem}\label{numbers}
  For all $i\geq 1$, we have
  $$\mbox{hom}(T_{i}, T_{1})=a_{i-1},~ \mbox{hom}(S_{0}, T_{i})=a_{i-2},~ \mbox{ext}^{1}(T_{i},T_{1})=a_{i-3},~ \mbox{ext}^{1}(T_{i},S_{0})=a_{i}.$$
  Similarly, for all $j\geq 0$, we have
  $$\mbox{hom}(S_{0}, S_{-j})= a_{j},~ \mbox{hom}(S_{-j}, T_{1})=a_{j-1},~ \mbox{ext}^{1}(S_{0},S_{-j})=a_{j-2},~ \mbox{ext}^{1}(T_{i},S_{-j})=a_{j+1}.$$
\end{theorem}

So the dimensions of these groups are all known in terms of $\mbox{ext}^{1}(T_{1}, S_{0})$.

\subsection{The Harder-Narasimhan Filtration}

In this subsection we compute the $\sigma_{-}$-Harder-Narasimhan filtration of a $\sigma_{+}$-stable spherical object.

We first prove a commutativity lemma:

\begin{lemma}\label{conn}
  For any $i\geq j$, the following diagram commutes:
  \[
\begin{tikzcd}
  T_{j}\otimes \mbox{Hom}(T_{i},T_{j})^{*} \arrow{r}{coev_{j, j-1}\otimes \mbox{id}} \arrow{d}{\mbox{id}\otimes ev_{j+1, j}^{*}} & T_{j-1}\otimes \mbox{Hom}(T_{j}, T_{j-1})^{*}\otimes \mbox{Hom}(T_{i},T_{j})^{*} \arrow{d}{\mbox{id}\otimes coev_{j+1, j}^{*}} \\
  T_{j}\otimes \mbox{Hom}(T_{j+1}, T_{j})^{*}\otimes \mbox{Hom}(T_{i}, T_{j+1})^{*} \arrow{r}{(ev_{j,j-1}\circ \epsilon_{j})\otimes \mbox{id}} & T_{j-1}\otimes \mbox{Hom}(T_{i},T_{j+1})^{*}
\end{tikzcd}
\]
Similarly, for any $i\geq j$, the following diagram commutes:
\[\begin{tikzcd}
    S_{i+1}\otimes \mbox{Hom}(S_{i-1}, S_{j}) \arrow{r}{(\epsilon_{i}\circ coev_{i+1, i})\otimes \mbox{id}}\arrow{d}{\mbox{id}\otimes ev_{i, i-1}^{*}} & S_{i}\otimes \mbox{Hom}(S_{i},S_{i-1})\otimes \mbox{Hom}(S_{i-1}, S_{j})\arrow{d}{\mbox{id}\otimes coev_{i, i-1}^{*}} \\
    S_{i+1}\otimes \mbox{Hom}(S_{i+1}, S_{i})\otimes \mbox{Hom}(S_{i},S_{j}) \arrow{r}{ev_{i+1, i}\otimes \mbox{id}} & S_{i}\otimes \mbox{Hom}(S_{i}, S_{j})
  \end{tikzcd}\]
\end{lemma}

\begin{proof}
We prove the commutativity of the first diagram; the second is similar. By Theorem \ref{description2}, for any $k\geq l$,
  $$\mbox{Hom}(T_{k}, T_{l})\subset \bigotimes_{j=l}^{k-1}\mbox{Hom}(T_{j+1}, T_{j})^{*},$$
  where the embedding is given by the composition of the duals of the evaluation maps. Hence it suffices to check the commutativity of the following diagram
  \[
\begin{tikzcd}
  T_{j}\otimes \bigotimes_{k=j}^{i-1}\mbox{Hom}(T_{k+1}, T_{k})^{*} \arrow{rr}{coev_{j,j-1}}\arrow{d}{\mbox{id}} & & T_{j-1}\otimes \bigotimes_{k=j-1}^{i-1}\mbox{Hom}(T_{k+1},T_{k})^{*}\arrow{d}{\mbox{id}\otimes \delta_{j}^{*}}\\
  T_{j}\otimes \bigotimes_{k=j}^{i-1}\mbox{Hom}(T_{k+1}, T_{k})^{*} \arrow{rr}{ev_{j, j-1}\circ (\mbox{id}\otimes \epsilon_{j})} & & T_{j-1}\otimes \bigotimes_{k=j+1}^{i-1}\mbox{Hom}(T_{k+1},T_{k})^{*}
\end{tikzcd}
\]
By Lemma \ref{identification} and Lemma \ref{commutativity},
\begin{align*}(\mbox{id}\otimes \delta_{j}^{*})\circ (coev_{j, j-1}\otimes \mbox{id})&=(\mbox{id}\otimes m) \circ (\mbox{id}\otimes \epsilon_{j})\circ (coev_{j, j-1}\otimes \mbox{id})\\
  &=(\mbox{id}\otimes m)\circ (coev_{j, j-1}\otimes \mbox{id})\circ( \mbox{id}\otimes \epsilon_{j})=ev_{j, j-1}\circ (\mbox{id}\otimes \epsilon_{j}),\end{align*}
namely a commutative diagram
\\
  \begin{adjustbox}{scale=0.8, center}
\begin{tikzcd}
  T_{j}\otimes \mbox{Hom}(T_{j+1},T_{j})^{*} \arrow{rr}{coev_{j,j-1}\otimes \mbox{id}} \arrow{dr}{coev_{j, j-1}\otimes \epsilon_{j}} \arrow{dd}{\mbox{id}\otimes \epsilon_{j}} & & T_{j-1}\otimes \mbox{Hom}(T_{j},T_{j-1})^{*}\otimes \mbox{Hom}(T_{j+1},T_{j})^{*}\arrow{dl}{\mbox{id}\otimes \epsilon_{j}}\arrow{dd}{\mbox{id}\otimes \delta_{j}^{*}}\\
  & T_{j-1}\otimes \mbox{Hom}(T_{j},T_{j-1})^{*}\otimes \mbox{Hom}(T_{j},T_{j-1}) \arrow{dr}{\mbox{id}\otimes m} & \\
  T_{j}\otimes \mbox{Hom}(T_{j},T_{j-1}) \arrow{ur}{coev_{j, j-1}\otimes \mbox{id}}\arrow{rr}{ev_{j}} & & T_{j-1},
\end{tikzcd}
\end{adjustbox}
commutativity of the left and upper triangles are clear, commutativity of the bottom triangle is by Lemma \ref{commutativity}, and commutativity of the right triangle is by Lemma \ref{identification}.

\end{proof}

  We let
  $$\mbox{conn}_{(i,j),(i,j-1)}: T_{j}\otimes \mbox{Hom}(T_{i},T_{j})^{*} \longrightarrow T_{j-1}\otimes \mbox{Hom}(T_{i},T_{j+1})^{*}$$
  $$\mbox{conn}_{(i+1, j), (i,j)}: S_{i+1}\otimes \mbox{Hom}(S_{i-1}, S_{j}) \longrightarrow S_{i}\otimes \mbox{Hom}(S_{i}, S_{j})$$
  be the composition maps described in Lemma \ref{conn}. We sometimes simply write $\mbox{conn}$ if the context is clear.

  \begin{theorem}\label{decomposition}
    There are short exact sequences
  \begin{gather*}0 \longrightarrow T_{i} \longrightarrow T_{j}\otimes \mbox{Hom}(T_{i},T_{j})^{*} \overset{\mbox{conn}}{\longrightarrow} T_{j-1}\otimes \mbox{Hom}(T_{i}, T_{j+1})^{*} \longrightarrow 0,\\
    0 \longrightarrow S_{i+1}\otimes \mbox{Hom}(S_{i-1},S_{j}) \overset{\mbox{conn}}{\longrightarrow} S_{i}\otimes \mbox{Hom}(S_{i},S_{j}) \longrightarrow S_{j} \longrightarrow 0,
    \end{gather*}
for all $i\geq j$, where $T_{0}:=S_{0}[1]$, $S_{1}:=T_{1}[-1]$.  
\end{theorem}

\begin{proof}
We prove the first sequence, the second is similar. Use induction on $i-j$. When $i-j=0$, $\mbox{Hom}(T_{i}, T_{j})^{*}=\mbox{Hom}(T_{i}, T_{i})^{*}=\mathbb{C}$, and $\mbox{Hom}(T_{i}, T_{j+1})^{*}=\mbox{Hom}(T_{i}, T_{i+1})^{*}=0$, hence the claim reads
$$0 \longrightarrow T_{i} \longrightarrow T_{i} \longrightarrow 0 $$
which is clearly true. When $i-j=1$, $\mbox{Hom}(T_{i},T_{j})^{*}=\mbox{Hom}(T_{i},T_{i-1})^{*}=\mbox{Hom}(T_{i-1},T_{i-2})$, and $\mbox{Hom}(T_{i}, T_{j+1})^{*}=\mbox{Hom}(T_{i}, T_{i})^{*}=\mathbb{C}$, hence the sequence reads
$$0 \longrightarrow T_{i} \longrightarrow T_{i-1}\otimes \mbox{Hom}(T_{i-1},T_{i-2}) \longrightarrow T_{i-2} \longrightarrow 0 $$
which is true by construction.

Now suppose the claim is true for up to $i-j$. We have the following commutative diagram
\\
\begin{adjustbox}{scale=0.8, center}
\begin{tikzcd}
  & 0\arrow{d} & 0\arrow{d}  & \\
  \mbox{ker}\arrow{r} & T_{i+1} \arrow{r}{coev_{i+1, j}}\arrow{d}{coev_{i+1, j+1}} & T_{j}\otimes \mbox{Hom}(T_{i+1},T_{j})^{*}\arrow{r}\arrow{d}{\mbox{id}\otimes ev_{j+1, j}^{*}} & \mbox{coker} \\
0 \arrow{r} & T_{j+1}\otimes \mbox{Hom}(T_{i+1},T_{j+1})^{*} \arrow{r}{coev_{j+1,j}\otimes \mbox{id}}\arrow{d}{\mbox{conn}} & T_{j}\otimes \mbox{Hom}(T_{j+1},T_{j})^{*}\otimes \mbox{Hom}(T_{i+1},T_{j+1})^{*} \arrow{r}{ev_{j, j-1}\circ (\mbox{id}\otimes \epsilon_{j})}\arrow{d}{\mbox{id}\otimes coev_{j+2, j+1}^{*}} & T_{j-1}\otimes \mbox{Hom}(T_{i+1},T_{j+1})^{*} \\
0 \arrow{r} & T_{j}\otimes \mbox{Hom}(T_{i+1}, T_{j+2})^{*} \arrow{r}{=}\arrow{d} & T_{j}\otimes \mbox{Hom}(T_{i+1},T_{j+2})^{*} \arrow{r}\arrow{d} & 0\\
 & 0 & 0 & .
\end{tikzcd}
\end{adjustbox}
The commutativity of the top square is clear, and the commutativity of the bottom square is by Lemma \ref{conn}. The first vertical sequences is exact by induction hypothesis. The second vertical sequence is exact by an argument similar to the one in Proposition \ref{description}. Hence by the Snake Lemma, $\mbox{ker}=0$ and $\mbox{coker}=T_{j-1}\otimes \mbox{Hom}(T_{i+1}, T_{j+1})^{*}$. By definition, the composition of
$T_{j}\otimes \mbox{Hom}(T_{i+1}, T_{j})^{*} \overset{\mbox{id}\otimes ev_{j+1, j}^{*}}{\longrightarrow} T_{j}\otimes \mbox{Hom}(T_{j+1},T_{j})^{*}\otimes \mbox{Hom}(T_{i+1},T_{j+1})^{*}$
and
$T_{j}\otimes \mbox{Hom}(T_{j+1},T_{j})^{*}\otimes \mbox{Hom}(T_{i+1},T_{j+1})^{*} \overset{ev_{j, j-1} \circ (\mbox{id}\otimes \epsilon_{j})}{\longrightarrow} T_{j-1}\otimes \mbox{Hom}(T_{i+1}, T_{j+1})^{*}$
is exactly $\mbox{conn}_{(i+1, j),(i+1, j-1)}$. By induction, the theorem is proved.

\end{proof}

\begin{corollary} \label{JH}
 For all $i\geq 1$, the Jordan-H\"{o}lder filtration of $T_{i}$ is given by
 $$ 0 \longrightarrow S_{0}\otimes \mbox{Hom}(T_{i},T_{2})^{*} \longrightarrow T_{i} \longrightarrow T_{1}\otimes \mbox{Hom}(T_{i}, T_{1})^{*} \longrightarrow 0, $$
 similarly, for all $j\leq 0$, the Jordan-H\"{o}lder filtration of $S_{j}$ is given by
 $$0 \longrightarrow S_{0}\otimes \mbox{Hom}(S_{0},S_{j}) \longrightarrow S_{j} \longrightarrow T_{1}\otimes \mbox{Hom}(S_{-1}, S_{j}) \longrightarrow 0.$$
\end{corollary}

\begin{proof}
For the first claim, let $j=1$ in Theorem \ref{decomposition}, for the second claim, let $i=0$ in Theorem \ref{decomposition}. 
\end{proof}

\begin{remark}
Since $\mbox{Hom}(S_{0}, T_{1})=\mbox{Hom}(T_{1},S_{0})=0$, we see $S_{0}\otimes \mbox{Hom}(T_{i},T_{2})^{*} \longrightarrow T_{i}$ is the evaluation map, and $S_{j}\longrightarrow T_{1}\otimes \mbox{Hom}(S_{-1}, S_{j})$ is the coevaluation map. 
\end{remark}

\section{Local Reduction}\label{section5}

In this section, we develop the local reduction algorithm that appeared in the global reduction (Algorithm \ref{global reduction}). The objective is to find the $\sigma_{-}$(resp. $\sigma_{+}$)-Harder-Narasimhan filtration of a rigid object (not necessarily stable), from its $\sigma_{+}$(resp. $\sigma_{-}$)-Harder-Narasimhan filtration. This can be viewed as a generalization of Corollary \ref{JH}. For a rough idea of local reduction, see Section \ref{section3}.

\subsection{Exhaustive Filtration}

In this subsection we introduce the notion of an exhaustive filtration. This notion is the central idea of the whole algorithm. It connects the information of the Harder-Narasimhan filtrations at $\sigma_{+}$ and $\sigma_{-}$. 

Let $W$ be a wall and $\mathcal{H}$ be the rank 2 lattice of $W$ with the two $\sigma_{+}$-stable spherical objects $S_{0}, T_{1}$. Assume $\phi_{+}(T_{1})>\phi_{+}(S_{0})$. Let $\mathcal{A}_{W}$ be the full subcategory of $\mathcal{A}$ consisting of objects whose Mukai vectors are in $\mathcal{H}$. 

\begin{definition}[Exhaustive filtration]\label{exhaustive}
For any rigid object $E\in \mathcal{A}_{W}$, a finite filtration
  $$0 \subset E_{1} \subset F_{1} \cdots \subset E_{n} \subset F_{n} \subset \cdots \subset E$$
  of $E$ is called $S$-exaustive, if $\mbox{Hom}(S_{0}, Q_{i})=0$ for all $Q_{i}=E/E_{i}, i\geq 1$, and $\mbox{Hom}(T_{1}, R_{j})=0$ for all $R_{j}= E/F_{j}, j\geq 1$. Similarly, it is called $T$-exhaustive if $\mbox{Hom}(T_{1},Q_{i})=0$ for all $i\geq 1$ and $\mbox{Hom}(S_{0}, R_{j})=0$ for all $j\geq 1$.
\end{definition}

\begin{lemma}\label{exist}
Exhaustive filtrations exist.
\end{lemma}

\begin{proof}
  We construct an $S$-exhaustive filtration, the other case is similar. Let $E_{1}=S_{0}\otimes \mbox{Hom}(S_{0}, E)$. This is the first $\sigma_{+}$-Harder-Narasimhan factor of $E$, and let $Q_{1}=E/E_{1}$. Then clearly $\mbox{Hom}(S_{0},Q_{1})=0$. Suppose we have constructed $Q_{i}$, to construct $R_{i}$, just note that $T_{1}\otimes \mbox{Hom}(T_{1},Q_{i})$ is the first $\sigma_{-}$-Harder-Narasimhan factor of $Q_{i}$ and we let $F_{i}/E_{i}=T_{1}\otimes \mbox{Hom}(T_{1},Q_{i})$. Similarly to construct $R_{i+1}$, we let $E_{i+1}/F_{i}=S_{0}\otimes \mbox{Hom}(S_{0},R_{i})$. Since the length of $E$ is finite, this process terminates and we get an $S$-exhaustive filtration.
  
\end{proof}

The next lemma shows that the exhaustive filtration of a rigid object is determined by its Jordan-H\"{o}lder factors.

\begin{lemma}\label{shape0}
  Let $E\in \mathcal{A}_{W}$ be rigid and
  $$0 \subset E_{1} \subset F_{1} \cdots \subset E_{n} \subset F_{n} \subset \cdots \subset E$$
  a rigid filtration of $E$. Then $Q_{i}, R_{i}$ are all rigid, and
  \begin{gather*}
    0 \longrightarrow F_{i}/E_{i} \longrightarrow Q_{i} \longrightarrow R_{i} \longrightarrow 0,\\
    0 \longrightarrow E_{i+1}/F_{i} \longrightarrow R_{i} \longrightarrow Q_{i+1} \longrightarrow 0
    \end{gather*}
  are general extensions.
\end{lemma}

\begin{proof}
  It suffices to prove the statement for $i=1$. Note that $E_{1}$ is the first $\sigma_{+}$ or $\sigma_{-}$ Harder-Narasimhan factor of $E$, hence by Lemma \ref{rigid}, $E_{1}$ and $Q_{1}$ are rigid. Since $E$ is rigid, the extension
  $$0 \longrightarrow E_{1} \longrightarrow E \longrightarrow Q_{1} \longrightarrow 0 $$
  is general. A similar argument proves the exactness of
  $$0 \longrightarrow F_{1}/E_{1} \longrightarrow E/E_{1} \longrightarrow R_{1} \longrightarrow 0.$$
\end{proof}

\subsection{Admissible Extensions}

\begin{notation}\label{notationexhaustive}
From now on we write
$$0\subset \cdots \subset E_{2}\subset F_{2} \subset E_{1}\subset F_{1} \subset E_{0}=E$$ for an exhaustive filtration, $Q_{i}=E/E_{i}, R_{i}=E/F_{i+1}$, $G_{i}=F_{i}/E_{i}$, $H_{i}= E_{i}/F_{i+1}$.
\\
Then $H_{0}=T\otimes B_{0}$ or $H_{0}=S\otimes A_{0}$, where $B_{0}, A_{0}$ are some finite dimensional vector spaces. For simplicity, we treat the case $H_{0}=T\otimes B_{0}$. The other case is similar. Then $G_{i}=S\otimes A_{i}$ and $H_{i}=T\otimes B_{i}$. Let $V=\mbox{Ext}^{1}(T_{1},S_{0})$.
\end{notation}

\begin{notation}\label{notationtensor}
  For the rest of the paper, for two linear spaces $V,W$, we sometimes write $VW$ for $V\otimes W$.
\end{notation}

  Consider $E_{i}/E_{i+1}\in B_{i}^{*}VA_{i+1}$ or $F_{i}/F_{i+1}\in A_{i}^{*}V^{*}B_{i}$, they are $V$ or $V^{*}$-Kronecker modules of certain fixed dimension vector. In this subsection we want to describe all possible Kronecker modules that arise from a fixed exhaustive filtration. Since we do not address stability of these Kronecker modules nor the geometry of these loci in the moduli spaces, we will work in the representation spaces $B_{i}^{*}VA_{i+1}$ and $A_{i}^{*}V^{*}B_{i}$. For more details about Kronecker modules, see \cite{Kin94}.

 Consider $\mbox{Ext}^{1}(R_{i},G_{i+1})$. We have the following commutative diagram. 
\[\begin{tikzcd}
    & & \mbox{Ext}^{2}(R_{i-1},G_{i+1})=0 \arrow[d]\\
    \mbox{Ext}^{1}(R_{i},G_{i+1}) \arrow[r] & \mbox{Ext}^{1}(H_{i},G_{i+1}) \arrow[r]\arrow[dr, "\phi"] & \mbox{Ext}^{2}(Q_{i}, G_{i+1})\arrow[d]\\
    & & \mbox{Ext}^{2}(G_{i},G_{i+1})\arrow[d]\\
    & & 0
  \end{tikzcd}\]
where the map $\phi$ is induced by the previous extension class $\delta_{i,i}\in \mbox{Ext}^{1}(G_{i},H_{i})=A_{i}^{*}V^{*}B_{i}$. Since $G_{i+1}=S\otimes A_{i+1}$, $\phi$ is the adjoint map
$$\delta_{i,i}\otimes \mbox{id}: B_{i}^{*}VA_{i+1} \longrightarrow A_{i}^{*}A_{i+1}.$$

\begin{definition}
  We say an extension $\delta_{i,i+1}\in \mbox{Ext}^{1}(H_{i},G_{i+1})=B_{i}^{*}VA_{i+1}$ is admissible, if it is the image of some $\widetilde{\delta_{i,i+1}}\in \mbox{Ext}^{1}(R_{i},G_{i+1})$ under the map above, where the middle term of $\widetilde{\delta_{i,i+1}}$ is isomorphic to $Q_{i+1}$.
  
  Similarly, we say $\delta_{i,i}\in \mbox{Ext}^{1}(G_{i},H_{i})=A_{i}^{*}V^{*}B_{i}$ is admissible, if it is the image of some $\widetilde{\delta_{i,i}}\in \mbox{Ext}^{1}(Q_{i}, H_{i})$, where the middle term of $\widetilde{\delta_{i,i}}$ is isomorphic to $R_{i}$.
  
  We denote the set of admissible extensions by $\mbox{Adm}_{i,i+1}$ or $\mbox{Adm}_{i,i}$.
\end{definition}

There are quasi-projective varieties in $B_{i}^{*}VA_{i+1}$ or $A_{i}^{*}V^{*}B_{i}$ that bound the set of admissible extensions.

\begin{lemma}\label{variety}
  There exists a closed subvariety $\mbox{ADM}_{i,i+1}\subset B_{i}^{*}VA_{i+1}$ and an open subvariety $\mbox{adm}_{i,i+1}\subset \mbox{ADM}_{i,i+1}$, such that
  $$\mbox{adm}_{i,i+1}\subset \mbox{Adm}_{i,i+1}\subset \mbox{ADM}_{i,i+1}.$$
  A similar conclusion holds for $\mbox{Adm}_{i,i}$.
\end{lemma}

\begin{proof}
  We do induction, clearly this is true for $\mbox{Adm}_{0,1}$. Consider $\mbox{adm}_{i,i}\times X$ with projections $p, q$, respectively. Then we have the relative version of the long exact sequence
  $$R^{1}p_{*}\underline{Hom}(q^{*}R_{i}, q^{*}G_{i+1})\overset{f}{\longrightarrow} R^{1}p_{*} \underline{Hom}(q^{*}H_{i}, q^{*}G_{i+1}) \overset{\delta_{i,i}\otimes \mbox{id}}{\longrightarrow} R^{2}\underline{Hom}(q^{*}G_{i}, q^{*}G_{i+1}) $$
  of sheaves on $\mbox{adm}_{i,i}$. Now $R^{1}p_{*}\underline{Hom}(q^{*}R_{i}, q^{*}G_{i+1})$ is a trivial vector bundle whose fibers are $\mbox{Ext}^{1}(R_{i},G_{i+1})$. By the rigidity of $Q_{i+1}$, there is an open subvariety $U_{i,i+1}\subset \mbox{Ext}^{1}(R_{i},G_{i+1})$, whose elements as extensions have middle term isomorphic to $Q_{i+1}$. Then $f(U_{i,i+1}\times \mbox{adm}_{i,i})\subset R^{1}p_{*} \underline{Hom}(q^{*}H_{i}, q^{*}G_{i+1})$ contains an open subvariety $U'_{i,i+1}$ of the image of $f$, denote by $\mathcal{F}$. Here, $f$ is viewed as a morphism on the total spaces of vector bundles. Since $R^{1}p_{*} \underline{Hom}(q^{*}H_{i}, q^{*}G_{i+1})\cong \mbox{adm}_{i,i}\times \mbox{Ext}^{1}(H_{i},G_{i+1})$ is a trivial vector bundle, denote the projection to $\mbox{Ext}^{1}(H_{i},G_{i+1})$ by $\pi$. Then we may take
  $$\mbox{ADM}_{i,i+1}=\overline{\pi(\mathcal{F})}\subset \mbox{Ext}^{1}(H_{i},G_{i+1}).$$
  This subvariety is equal to the closure of
  $$\bigcup_{\delta_{i,i}\in \mbox{Adm}_{i,i}}\mbox{ker}(\delta_{i,i}\otimes \mbox{id}),$$ which contains $\mbox{Adm}_{i,i+1}$. Now by construction, $\mbox{Adm}_{i,i+1}$ contains $\pi(U'_{i,i+1})$, which contains an open subvariety of $\mbox{ADM}_{i,i+1}$. 
  
\end{proof}

We will show how to compute the ideals of $ADM_{j,j}$ and $\mbox{ADM}_{i,i+1}$ inductively. For this, we need the following lemma.

\begin{lemma}\label{component}
  Let $f: X \longrightarrow Y$ be a surjective morphism between varieties, and $Y_{1}\subset Y$ an irreducible component. Suppose a general fiber over $Y_{1}$ is a linear space of fixed dimension. Then there is a unique irreducible component $X_{1}\subset X$ that dominates $Y_{1}$.
\end{lemma}

\begin{proof}
By assumption, over an open subset $U\subset Y_{1}$, $f$ is a topological fibration with isomorphic irreducible fibers. Let $X_{1}$ be the closure of $f^{-1}(U)$, clearly it is irreducible. Let $x\in X_{1}$ be a general point, $y=f(x)$ be its image. Since $x$ is general, $y\in Y_{1}$ is also general. Take an open neighborhood $V$ of $x$, by shrinking if necessary we have $f(V)\subset U$. Then $V\subset f^{-1}(U)\subset X_{1}$, we see that any specialization to $x$ must come from $X_{1}$. Hence $X_{1}$ is an irreducible component.
\end{proof}

We describe how to compute $\mbox{ADM}_{i,i+1}$ assuming that the ideal of $\mbox{ADM}_{i,i}$ is known. The method for $\mbox{ADM}_{j,j}$ is similar. Consider the incidence correspondence
$$\Gamma=\{(\delta_{i,i}, \delta_{i,i+1}): \delta_{i,i}\in \mbox{ADM}_{i,i}, m(\delta_{i,i}, \delta_{i,i+1})=0\}\subset A_{i}^{*}V^{*}B_{i}\times B_{i}^{*}VA_{i+1},$$
where $m: A_{i}^{*}V^{*}B_{i}\otimes B_{i}^{*}VA_{i+1} \longrightarrow A_{i}^{*}\otimes A_{i+1}$ is the natural pairing. Suppose that we know the ideal of $ADM_{i,i}$. The ideal of $\Gamma$ can be computed, since it is the intersection of the pre-image of $ADM_{i,i}$ with the vanishing locus of the paring $m$, which can be write down explicitly in coordinates. By Lemma \ref{component}, $\Gamma$ has a distinguished component $\Gamma_{1}$ that dominates $ADM_{i,i}$. By doing primary decomposition, we may find the ideal of this component $\Gamma_{1}$. By the proof of Lemma \ref{variety}, the image of $\Gamma_{1}$ under the second projection is $\mbox{ADM}_{i,i+1}$ set-theoretically. Since we only care about the admissible locus set-theoretically, we may take the scheme-theoretic image of $\Gamma_{1}$ under the second projection. If the ideal of the image is not prime, then we take its radical. The resulting ideal is the ideal of $ADM_{i,i+1}$. Hence in principle, the ideal of $\mbox{ADM}_{i,i+1}$ can be computed, given the ideal of $\mbox{ADM}_{i,i}$. We summarize this procedure in the following figure:
$$\mbox{ADM}_{i,i} \xrightarrow{\begin{subarray} ~ \mbox{incidence}\\
    \mbox{correspondence}\end{subarray}} \Gamma \xrightarrow{ \begin{subarray}~\mbox{primary}\\
    \mbox{decomposition}
  \end{subarray}}\Gamma_{1} \xrightarrow{\begin{subarray}~\mbox{projection}\\
  + \mbox{radical} \end{subarray}} \mbox{ADM}_{i,i+1}.$$
Now $\mbox{ADM}_{0,1}=B_{0}^{*}VA_{1}$ is known. Hence to compute $\mbox{ADM}_{i,i+1}$, the only needed input is the dimensions of $A_{k}, B_{l}$ for $k\leq i+1, l\leq i$. We summarize this in the following porposition.

\begin{proposition}\label{admknown}
Using Notation \ref{notationexhaustive}, the ideal of $\mbox{ADM}_{i,i+1}\subset B_{i}VA_{i+1}$ can be computed in terms of the dimensions of $A_{k}, B_{l}$ for $k\leq i+1, l\leq i$. A similar conclusion holds for $\mbox{ADM}_{j,j}$.
\end{proposition} 

By Lemma \ref{shape0}, an exhaustive filtration is determined by the dimensions of all $A_{k}, B_{l}$ under Notation \ref{notationexhaustive}. Hence we introduce the following definition.

\begin{definition}\label{shape}
  Using Notation \ref{notationexhaustive}, the \emph{shape} of an exhaustive filtration is the sequence of numbers
  $$(\cdots, \mbox{dim}A_{i+1}, \mbox{dim} B_{i}, \mbox{dim} A_{i}, \cdots, \mbox{dim} A_{1}, \mbox{dim}B_{0}).$$
\end{definition}

Hence from now on we assume all admissible loci are computable in terms of the shape of the exhaustive filtration.

The next observation describes the symmetry of $\mbox{ADM}$.

\begin{lemma}
The variety $\mbox{ADM}_{i,i+1}$ is stable under $\mbox{GL}_{B_{i}}\times \mbox{GL}_{A_{i+1}}$. Similarly, $\mbox{ADM}_{i,i}$ is stable under $\mbox{GL}_{A_{i}} \times \mbox{GL}_{B_{i}}$. 
\end{lemma}

\begin{proof}
  We only prove the lemma for $\mbox{ADM}_{i,i+1}$, the other case is similar.

  First note that $\mbox{Adm}_{i,i+1}$ is clearly stable under $\mbox{GL}_{A_{i+1}}$, because a change of basis on $A_{i+1}$ does not change the middle term of the extension. Since $\mbox{ADM}_{i,i+1}$ is irreducible and $\mbox{Adm}_{i,i+1}$ contains an open subset of $\mbox{ADM}_{i,i+1}$, $\mbox{ADM}_{i,i+1}$ is stable under $\mbox{GL}_{A_{i+1}}$.
  \\
  For $\mbox{GL}_{B_{i}}$ part, take any element $\beta\in \mbox{GL}_{B_{i}}$. Under the notations above, consider
  $$\pi(\mathcal{F})=\bigcup_{\delta_{i,i}\in \mbox{adm}_{i,i}}\mbox{ker}(\delta_{i,i}\otimes \mbox{id}).$$
  By the induction hypothesis on $\mbox{ADM}_{i,i}$, we may shrink $\mbox{adm}_{i,i}$ if necessary so that it is preserved under $\mbox{GL}_{B_{i}}$. Now then clearly
  $$\beta(\pi(\mathcal{F}))= \pi(\beta(\mathcal{F}))=\pi(\mathcal{F}),$$
  since $\beta(\delta_{i,i}\otimes \mbox{id})=\beta(\delta_{i,i})\otimes \mbox{id}$ and $\beta(\delta_{i,i})\in \mbox{adm}_{i,i}$ by construction.
\end{proof}

\begin{remark}
One may be tempted to expect the varieties $\mbox{ADM}_{i,i}$ or $\mbox{ADM}_{i,i+1}$ are determinantal varieties, but in fact they can be more complicated. 
\end{remark}

\subsection{Relation to the Harder-Narasimhan Filtration}

In this subsection we relate exhaustive filtration to the $\sigma_{-}$-Harder-Narasimhan filtration. Suppose 
$$0=E_{n}\subset F_{n} \subset E_{n-1} \subset F_{n-1} \subset \cdots \subset E_{1} \subset F_{1} \subset E_{0}=E$$
is an $S$-exhaustive filtration of a rigid object $E\in \mathcal{A}_{W}$. We want to compute the $\sigma_{-}$-Harder-Narasimhan filtration of $E$. The first thing to note is that

\begin{lemma}
  The first $\sigma_{-}$-Harder-Narasimhan factor is
  $$\mbox{gr}_{1}^{-}(E)=F_{n}/E_{n}=G_{n}=S\otimes A_{n}.$$
\end{lemma}

\begin{proof}
  Apply $\mbox{Hom}(S,-)$ to the short exact sequence
  $$0 \longrightarrow G_{n} \longrightarrow E \longrightarrow R_{n-1} \longrightarrow 0 ,$$
  we have
  $$0 \longrightarrow \mbox{Hom}(S,G_{n}) \longrightarrow \mbox{Hom}(S, E) \longrightarrow \mbox{Hom}(S, R_{n-1}).$$
  By the definition of $S$-exhaustive filtration, $\mbox{Hom}(S,R_{n-1})=0$. Hence $\mbox{Hom}(S,E)\cong \mbox{Hom}(S,G_{n})=A_{n}$. Then $G_{n}\subset E$ is the evaluation map
  $$S\otimes \mbox{Hom}(S,E) \longrightarrow E.$$
  Note that $S=S_{0}$ has maximal $\sigma_{-}$-phase among all $\sigma_{-}$-stable spherical objects in $\mathcal{H}$, hence $S\otimes \mbox{Hom}(S, E)$ is the first Harder-Narasimhan factor.
\end{proof}

Hence our task is to compute the first $\sigma_{-}$-Harder-Narasimhan factor of $R_{n-1}=E/\mbox{gr}_{1}^{-}(E)$. For simplicity we write $O_{i}=S_{i}^{-}$ for $i\leq 0$ and $O_{i}=T_{i}^{-}$ for $i\geq 1$. Our next observation makes crucial use of the notion of exhaustive filtration.

\begin{proposition}\label{Rn-1}
  Under Notation \ref{notationexhaustive}, let $R_{n-1}$ be an object equipped with a $T$-exhaustive filtration. Then for any $i\in \mathbb{Z}$,
  $$\mbox{Hom}(O_{i},R_{n-1})=\mbox{Hom}(O_{i},F_{n-1}/F_{n}).$$
\end{proposition}

\begin{proof}
  Apply $\mbox{Hom}(O_{i},-)$ to the short exact sequence
  $$0 \longrightarrow F_n/F_{n-1} \longrightarrow R_{n-1} \longrightarrow R_{n-2} \longrightarrow 0 ,$$
  we get
  $$0 \longrightarrow \mbox{Hom}(O_{i},F_{n}/F_{n-1}) \longrightarrow \mbox{Hom}(O_{i},R_{n-1}) \overset{f}{\longrightarrow} \mbox{Hom}(O_{i},R_{n-2}).$$
  It suffices to show the map $f$ is zero.
\\
  We know $O_{i}$ fits into
  $$0 \longrightarrow T\otimes \mbox{Hom}(T,O_{i}) \longrightarrow O_{i} \longrightarrow S\otimes \mbox{Hom}(O_{i},S)^{*} \longrightarrow 0 .$$
  Let
  \begin{equation}\label{MiNi} 
    N_{i}=\mbox{Hom}(T,O_{i}), ~M_{i}=\mbox{Hom}(O_{i},S)^{*}.
  \end{equation}
  We have the following commutative diagram
  \[\begin{tikzcd}
      \mbox{Hom}(S\otimes M_{i},R_{n-1}) \arrow[r]\arrow[d] & \mbox{Hom}(O_{i},R_{n-1}) \arrow[r]\arrow[d, "f"] & \mbox{Hom}(T\otimes N_{i},R_{n-1}) \arrow[d, "g"]\\
      \mbox{Hom}(S\otimes M_{i},R_{n-2}) \arrow[r] & \mbox{Hom}(O_{i},R_{n-2}) \arrow[r] & \mbox{Hom}(T\otimes N_{i},R_{n-2}).
    \end{tikzcd}\]
  Since $R_{n-2}$ is a $T$-exhaustive filtration, we have $\mbox{Hom}(S,R_{n-2})=0$. Also note that the factorization of the map $R_{n-1} \longrightarrow R_{n-2}$ into the composition of $R_{n-1} \longrightarrow Q_{n-1}$ and $Q_{n-1} \longrightarrow R_{n-2}$ induces the factorization of the map $g: \mbox{Hom}(T\otimes N_{i},R_{n-1}) \longrightarrow \mbox{Hom}(T\otimes N_{i},R_{n-2})$ into $\mbox{Hom}(T\otimes N_{i}, R_{n-1}) \longrightarrow \mbox{Hom}(T\otimes N_{i}, Q_{i-1})$ and $\mbox{Hom}(T\otimes N_{i},Q_{i-1}) \longrightarrow \mbox{Hom}(T\otimes N_{i}, R_{n-2})$. By the definition of a $T$-exhaustive filtration, $\mbox{Hom}(T,Q_{i-1})=0$. Hence $g=0$, and the commutative diagram now reads
  \[\begin{tikzcd}
       & \mbox{Hom}(O_{i},R_{n-1}) \arrow[r]\arrow[d, "f"] & \mbox{Hom}(T\otimes N_{i},R_{n-1})\arrow[d, "0"]\\
      0 \arrow[r] & \mbox{Hom}(O_{i},R_{n-2}) \arrow[r] & \mbox{Hom}(T\otimes N_{i},R_{n-2}).
    \end{tikzcd}\]
Hence the map $f$ must be zero.
  
\end{proof}

Hence to get the first Harder-Narasimhan factor for $R_{n-1}$, it remains to find the index $i$ for $O_{i}$ such that $\mbox{Hom}(O_{i},R_{n-1})\neq 0$ and $\phi_{-}(O_{i})$ is maximal, and to compute the dimension of $\mbox{Hom}(O_{i},R_{n-1})$. For simplicity, we will sometimes drop the indices of $O_{i}$ and $R_{n-1}$ when there is no ambiguity.

Consider the following commutative diagram
\\
\begin{adjustbox}{scale=0.8, center}
\begin{tikzcd}
    & & & \mbox{Hom}(O,T\otimes B)=0 \arrow[d]\\
    & & 0 \arrow[d] & \mbox{Hom}(T\otimes N,T\otimes B) \arrow[d, "\delta_{0}\otimes \mbox{id}_{B}"] \\
    & 0=\mbox{Hom}(S\otimes M,\Gamma)\arrow[r]\arrow[d] & \mbox{Hom}(S\otimes M, S\otimes A) \arrow[r, "\mbox{id}_{M^{*}}\otimes \delta"]\arrow[d,"="] & \mbox{Ext}^{1}(S\otimes M,T\otimes B)\arrow[d]\\
    0=\mbox{Hom}(O,T\otimes B)\arrow[r] & \mbox{Hom}(O,\Gamma) \arrow[r] & \mbox{Hom}(O,S\otimes A)\arrow[r]\arrow[d] & \mbox{Ext}^{1}(O,T\otimes B)\arrow[d]\\
    & & 0=\mbox{Hom}(T\otimes N,S\otimes A) & \mbox{Ext}^{1}(T\otimes N,T\otimes B)=0,
  \end{tikzcd}
\end{adjustbox}
\\
where $\delta\in \mbox{ADM}\subset A^{*}V^{*}B$ is a general admissible extension, and
$$[0 \longrightarrow T\otimes N \longrightarrow O \longrightarrow S\otimes M \longrightarrow 0 ]=\delta_{0}\in M^{*}V^{*}N$$
is general, hence in particular the image of $\mbox{Hom}(T\otimes N, T\otimes B)$ under $\delta_{0}\otimes \mbox{id}_{B}$ is of the form $W\otimes B$, where $W\subset M^{*}V^{*}$ is general.

$J=\mbox{Hom}(O,\Gamma)$ can be identified with
$$\mbox{Hom}(O,\Gamma)=\mbox{Hom}(S\otimes M,S\otimes A)\cap \mbox{Hom}(T\otimes N, T\otimes B)=M^{*}\delta(A)\cap \delta_{0}(N^{*})B\subset M^{*}V^{*}B.$$
The dimension of $\delta_{0}(N^{*})B$ is known in terms of the dimensions of $M, N$, since $\delta_{0}\in M^{*}V^{*}N$ is general. The dimension of $M^{*}\delta(A)$ is also known in principle, since we know the ideal of $\mbox{ADM}$ by Proposition \ref{admknown}. We would like to compute the dimension of their intersection. Unfortunately, in general they do not intersect transversely.

Hence we need to know the locus of $\delta(A)\subset V^{*}B$ for a general $\delta\in \mbox{ADM}$. Let $a=\mbox{dim}\delta(A)$ for a general $\delta\in \mbox{ADM}$ and $A'\subset A$ be any subspace of dimension $a$, and $i: A' \longrightarrow A$ be the inclusion. Let $\mbox{ADM}'\subset A'^{*}V^{*}B$ be the image of $\mbox{ADM}$ under the projection $i^{*}:A^{*}V^{*}B \longrightarrow A'^{*}V^{*}B$. Since $\mbox{ADM}$ is stable under $\mbox{GL}_{A}\times \mbox{GL}_{B}$, $\mbox{ADM}'$ is independent of the choice of $A'$. Then for a general $\delta'\in \mbox{ADM}'$, the map $\delta': A' \longrightarrow V^{*}B$ is injective. 

Let $x_{1}, \cdots , x_{a}$ be a basis of $A'$. Then there is a rational map $l: \mbox{ADM}' \dashrightarrow G(a, V^{*}B)$, defined on the locus $\{\delta' \in \mbox{ADM}: \delta': A' \longrightarrow V^{*}B \mbox{ is injective} \}$, by
\begin{equation}
  l(\delta')=\delta'(x_{1})\wedge \delta'(x_{2})\wedge \cdots \wedge\delta'(x_{a})
\end{equation}
under the Pl\"{u}cker embedding. Denote the image by $l(\mbox{ADM})$. Let
$$\sigma(N,k)=\{\Lambda\in G(ma, M^{*}V^{*}B): \mbox{dim}(\Lambda\cap \delta_{0}(N)\otimes B)\geq k\}.$$
The $\sigma(N,k)$ are given by Proposition \ref{description} and Proposition \ref{description2}, hence in principle we may compute the maximal integer $k$ such that $M^{*}\otimes l(\mbox{ADM}')\subset \sigma(N,k)$. This number is determined by the index $i$ of $O_{i}$, which we shall denote by $k(i)$.

In order to find the first Harder-Narasimhan factor, we need to compare the $\phi_{-}$-phases of $O_{i}$. Recall that we defined $a_{i}$ to be the fundamental sequence of $\mathcal{H}$ in Definition \ref{fundamental sequence}. Now we define a function $\phi_{-}$ on $\mathbb{Z}$ by
$$\phi_{-}(i)=
\begin{cases}
  \frac{a_{i-2}}{a_{i-2}+a_{i-1}}, i\geq 1,\\
  \frac{a_{-i}}{a_{-i}+a_{-i-1}}, i\leq 0. 
\end{cases}
$$

Then $\phi_{-}(O_{i})>\phi_{-}(O_{j})$ if and only if $\phi_{-}(i)>\phi_{-}(j)$. Hence we have the following proposition.

\begin{proposition}\label{first factor}
Under the assuptions in Proposition \ref{Rn-1}, the first Harder-Narasimhan factor of $R_{n-1}$ is $O_{i}\otimes J$ such that $\phi_{-}(i)$ is maximal and $k(i)>0$. The vector space $J_{i}=\mbox{Hom}(O_{i}, R_{n-1})$ has dimension $k(i)$. 
\end{proposition}

\begin{remark}\label{few}
  At first glance there are infinitely many $i$ such that $k(i)>0$, and it is not clear that if we can find the $i$ such that $\phi_{-}(i)$ is maximal.
  
  However, in practice there are only finitely many $i$ to check. Using the notation in Proposition \ref{first factor}, if $O_{i}$ is the first $\sigma_{-}$-Harder-Narasimhan factor, then $O_{i}\otimes J_{i} \longrightarrow R_{n-1}$ is injective. Hence we must have
  $$\mbox{dim}(M_{i})\leq \mbox{dim}(M_{i})\cdot k(i) \leq \mbox{dim}(A), ~ \mbox{dim}(N_{i})\leq \mbox{dim}(N_{i})\cdot k(i) \leq \mbox{dim}(B). $$
  Since the dimensions of $A$ and $B$ are data in the exhaustive filtration of $E$, the possible range of $i$ is computed by Theorem \ref{numbers}. There are only finitely many such $i$, since the fundamental sequence (Definition \ref{fundamental sequence}) $a_{i}$ tends to infinity.
\end{remark}

\begin{remark}
The first $\sigma_{-}$-Harder-Narasimhan factor is of the form $O_{i}^{\oplus n}$. Hence the information consists of two numbers $i$ and $n$. Proposition \ref{first factor} is where we find out the two numbers. Since the dimensions of $M_{k}, N_{l}$ are given by Theorem \ref{numbers}, the $i$ and $n$ are determined by the admissible locus $\mbox{ADM}\subset A^{*}V^{*}B$, which is eventually determined by the shape of the exhaustive filtration by Proposition \ref{admknown}.
\end{remark}

\subsection{Receptive Extensions}

Recall that our objective of the local reduction is to compute the $\sigma_{-}$-Harder-Narasimhan filtration of $E$. In the last subsection, we showed how to compute the first $\sigma_{-}$-Harder-Narasimhan factor $E_{1}$ of $E$ from the shape of an exhaustive filtration of $E$. In order to find the $\sigma_{-}$-Harder-Narasimhan filtration of $E$, we would like to know the shape of an exhaustive filtration of $E/E_{1}$, so that we may use Proposition \ref{first factor} again. In this subsection, we introduce a necessary notion called \emph{receptive extension}.

The original $S$-exhaustive filtration on $E$ is indeed a $T$-exhaustive filtration on $R_{n-1}$, hence we only need to consider $\mbox{Hom}(O_{i},F_{n-1}/F_{n})$. Let $\Gamma=F_{n-1}/F_{n}$. Then $\Gamma$ fits into an extension
$$[0 \longrightarrow T\otimes B_{n-1} \longrightarrow \Gamma \longrightarrow S\otimes A_{n-1} \longrightarrow 0 ]=\delta_{n-1,n-1}\in A_{n-1}^{*}V^{*}B_{n-1},$$
where $\delta_{n-1,n-1}\in \mbox{Adm}_{n-1,n-1}$. If $i$ is the integer so that $O_{i}\otimes \mbox{Hom}(O_{i},R_{n-1})$ is the first Harder-Narasimhan factor of $R_{n-1}$, then $O_{i}\otimes \mbox{Hom}(O_{i},R_{n-1}) \longrightarrow R_{n-1}$ is injective. Let $J=\mbox{Hom}(O_{i},R_{n-1})$. We have the following commutative diagram

\[\begin{tikzcd}
    &T\otimes K_{1}\arrow[d] & 0\arrow[d] & S\otimes K_{2} \arrow[d] & \\
    0 \arrow[r] & T\otimes N_{i}J \arrow[r]\arrow[d] & O_{i}J \arrow[r]\arrow[d] & S\otimes M_{i}J \arrow[r]\arrow[d] & 0 \\
    0 \arrow[r] & T\otimes B_{n-1} \arrow[r]\arrow[d] & \Gamma \arrow[r]\arrow[d] & S\otimes A_{n-1} \arrow[r]\arrow[d]& 0\\
    & T\otimes B''  & \Gamma'' & S\otimes A''&  ,
  \end{tikzcd}\]
where $B'', A''$ are some fixed finite dimensional vector spaces. By the Snake Lemma, $K_{1}=0$. Note that $\mbox{Hom}(S,T)=0$, hence $K_{2}=0$. Hence we get an extension
$$[0 \longrightarrow T\otimes B'' \longrightarrow \Gamma'' \longrightarrow S\otimes A'' \longrightarrow 0 ]=\delta''\in A''^{*}V^{*}B''.$$
Hence there is a rational map
\begin{gather}\label{rationalmap}
  \pi_{n-1,n-1}: \mbox{ADM}_{n-1,n-1} \dashrightarrow A''^{*}V^{*}B''.
\end{gather}
In this subsection we study the image of $\pi=\pi_{n-1,n-1}$. We call the closure of the image the receptive locus $\mbox{REC}_{n-1,n-1}$ and a general element in it receptive. Note that $\mbox{REC}$ is stable under $\mbox{GL}_{A''}\times \mbox{GL}_{B''}$ action.

We describe the rational map $\pi=\pi_{n-1,n-1}$ now. Recall that
$$J=\mbox{Hom}(O,\Gamma)=\mbox{Hom}(S\otimes M,S\otimes A)\cap \mbox{Hom}(T\otimes N, T\otimes B)=M^{*}\delta(A)\cap \delta_{0}(N^{*})B\subset M^{*}VB.$$
The difficulty here is that, with fixed vector spaces $A,B$, the embeddings $NJ\hookrightarrow B$ and $MJ\hookrightarrow A$ depends on $\delta$. However, this problem can be solved if we restrict to an open subset of $\mbox{ADM}=\mbox{ADM}_{n-1,n-1}$, because it does not affect the image of $\pi$.

The map factors as $\pi_{B}: \mbox{ADM} \dashrightarrow \mbox{ADM}'\subset A^{*}VB''$ and $\pi_{A}: \mbox{ADM}' \dashrightarrow A''^{*}V^{*}B''$. We first compute the image of $\mbox{ADM}$ under $\pi_{B}$. 
For a general $\delta\in \mbox{ADM}$, let $q_{\delta}: B \longrightarrow B''$ be the family of projections so that $B=NJ\oplus_{\delta} B''$. Then the map $\pi_{B}|_{U}$ is defined as
$$\pi_{B}(\delta)=(\mbox{id}\otimes q_{\delta}) (\delta)\in A^{*}V^{*}B''.$$

Fix a projection $q: B \longrightarrow B''$, then since $\mbox{ADM}'$ is stable under the $\mbox{GL}_{B''}$ action, for a general $\delta\in \mbox{ADM}$, $(\mbox{id}\otimes q)(\delta)\in \mbox{ADM}'$. Conversely, for a general $\delta$, there is some $g\in \mbox{GL}_{B}$ such that $(\mbox{id}\otimes q_{\delta})(\delta)=(\mbox{id}\otimes q)(g(\delta))$, and clearly $g(\delta)\in \mbox{ADM}$ since it is stable under $\mbox{GL}_{B}$. Hence $\mbox{ADM}'=q(\mbox{ADM})$. 

Similarly for every $\delta'\in \mbox{ADM}'\subset A^{*}V^{*}B''$, there exists a subspace $(A''_{\delta})^{*}\subset A''^{*}$ such that $\delta\in (A''_{\delta})^{*}V^{*}B''$. Fix a subspace $A''^{*} \subset A^{*}$ and a projection $p: A^{*} \longrightarrow A''^{*}$. Let $U\subset \mbox{ADM}'\subset A^{*}V^{*}B''$ be an open subset such that for $\delta\in U$, $\delta$ has maximal $(A,VB)$-rank and $(AV,B)$-rank, and $p|_{(A''_{\delta})^{*}}$ is an isomorphism. Then by a similar trick as above, we see the image of $\pi_{A}$ is $p(\mbox{ADM}')$. Hence in summary we have showed: 
 
\begin{proposition}\label{REC1}
  Let $\pi: \mbox{ADM} \dashrightarrow A''^{*}V^{*}B''$ be the rational map defined in (\ref{rationalmap}). Choose any projections $p: A^{*} \longrightarrow A''^{*}$ and $q: B \longrightarrow B''$. Then the image of $\pi$ is
  $$\mbox{REC}=(p\times \mbox{id} \times q)(\mbox{ADM}).$$
  Since $\mbox{ADM}$ is stable under $\mbox{GL}_{A}\times \mbox{GL}_{B}$, $\mbox{REC}$ is independent of the choice of $p,q$.
\end{proposition}

By Proposition \ref{REC1}, the receptional locus $\mbox{REC}$ can be computed, assuming that we know the admissible locus $\mbox{ADM}$. By Proposition \ref{admknown}, $\mbox{ADM}$ can be computed from the shape of the exhaustive filtration. Hence at this stage we may assume the receptional loci $\mbox{REC}$, produced by taking quotient from $\Gamma$ by the first Harder-Narasimhan factor, are all known in terms of the shape of the exhaustive filtration.

\subsection{Concatenation}

In this subsection we determine the shape of an exhaustive filtration of an object, that is a general extension of a known exhaustive filtration by a receptive extension.

We know $O\otimes J\subset R_{n-1}$ is the first Harder-Narasimhan factor of $R_{n-1}$, by Mukai's Lemma (Lemma \ref{Mukai's Lemma}), $R_{n-1}/ (O\otimes J)$ is rigid. On the other hand, $O\otimes J \longrightarrow R_{n-1}$ factors through $\Gamma$, we have the following commutative diagram
\[\begin{tikzcd}
    & O\otimes J \arrow[d] & O\otimes J \arrow[d] & & \\
    0\arrow[r] & \Gamma \arrow[r]\arrow[d] & R_{n-1} \arrow[r]\arrow[d] & R_{n-2}\arrow[r]\arrow[d, "="] & 0\\
    0\arrow[r] & \Lambda \arrow[r] & R_{n-1}/ (O\otimes J) \arrow[r] & R_{n-2} \arrow[r] & 0,
  \end{tikzcd}\]
where $\Lambda=\Gamma/ (O\otimes J)$ can fit into a general receptive extension
$$[0 \longrightarrow T\otimes B'' \longrightarrow \Lambda \longrightarrow S\otimes A'' \longrightarrow 0 ]=\delta''\in \mbox{REC}_{n-1.n-1} \subset A''^{*}V B''.$$
By rigidity of $Q_{1}^{-}=R_{n-1}/(O\otimes J)$, $Q_{1}^{-}$ fits into a general extension in $\mbox{Ext}^{1}(R_{n-2},\Lambda)$. 

Our next task is to compute the $T$-exhaustive filtration of $Q_{1}^{-}$. We slightly generalize the situation to the following questions:
\\
Question T \label{question T}:
\\
Let $P$ be a rigid object that fits into an exact sequence
$$0 \longrightarrow \Lambda \longrightarrow P \longrightarrow R_{m} \longrightarrow 0 ,$$
where $R_{m}$ is a rigid object equipped with a known $T$-exhaustive filtration labeled as in Definition \ref{exhaustive}, and $\Lambda$ fits into a general receptive extension
$$[0 \longrightarrow T\otimes B \longrightarrow \Lambda \longrightarrow S\otimes A \longrightarrow 0 ]=\delta \in \mbox{REC}\subset \mbox{Ext}^{1}(S\otimes A,T\otimes B)=A^{*}V^{*}B.$$
Compute the $T$-exhaustive filtration of $P$.

Similarly, we can formulate
\\
Question S \label{question S}:
\\
Let $P$ be a rigid object that fits into an exact sequence
$$0 \longrightarrow \Lambda \longrightarrow P \longrightarrow Q_{m} \longrightarrow 0 ,$$
where $Q_{m}$ is a rigid object equipped with a known $S$-exhaustive filtration labeled as in Definition \ref{exhaustive}, and $\Lambda$ fits into a general receptive extension
$$[0 \longrightarrow S\otimes A \longrightarrow \Lambda \longrightarrow T\otimes B \longrightarrow 0 ]=\delta \in \mbox{REC}\subset \mbox{Ext}^{1}(T\otimes B,S\otimes A)=B^{*}VA. $$
Compute the $S$-exhaustive filtration of $P$.
\\
We now solve Question T. First we consider the following exact sequence
$$0 \longrightarrow \mbox{Hom}(T,\Lambda) \longrightarrow \mbox{Hom}(T,P) \longrightarrow \mbox{Hom}(T,R_{m}) \overset{f}{\longrightarrow} \mbox{Ext}^{1}(R_{m},\Lambda),$$
let $B'=\mbox{ker}(f)$ and $B''=\mbox{Hom}(T,R_{m})/B'$. Let $R_{m}'=R_{m}/(T\otimes B')$. Then we have the following commutative diagram
\[\begin{tikzcd}
    & 0\arrow[d] & 0\arrow[d] & 0\arrow[d] & \\
0\arrow[r] & T\otimes B \arrow[r]\arrow[d] & T\otimes \mbox{Hom}(T,P)\arrow[r]\arrow[d] & T\otimes B'\arrow[r]\arrow[d] & 0 \\
0 \arrow[r] & \Lambda \arrow[r]\arrow[d] & P \arrow[r]\arrow[d] & R_{m}\arrow[r]\arrow[d] & 0\\
 & S\otimes A  & P'  & R_{m}' & 
\end{tikzcd}. \]
Now $\mbox{Hom}(T,P')=0$, hence if we can find the $S$-exhaustive filtration for $P'$, then the $T$-exhaustive filtration for $P$ is also known. Since $P$ is rigid and $\mbox{Hom}(T\otimes \mbox{Hom}(T,P), P')=0$, by Mukai's Lemma (Lemma \ref{Mukai's Lemma}), $P'$ is also rigid. Now consider the following pullback diagram
\[\begin{tikzcd}
    & & 0\arrow[d] & 0\arrow[d]& \\
    0 \arrow[r] & S\otimes A \arrow[r]\arrow[d, "="] & \Lambda' \arrow[r]\arrow[d] & T\otimes B'' \arrow[r]\arrow[d] & 0\\
    0 \arrow[r] & S\otimes A \arrow[r] & P' \arrow[r]\arrow[d] & R_{m}' \arrow[r]\arrow[d] & 0\\
    & & Q_{m} \arrow[r, "="]\arrow[d] & Q_{m}\arrow[d] & \\
    & & 0 & 0 & 
  \end{tikzcd}.\]
$P'$ fits into an exact sequence
\begin{align}\label{P'}
  0 \longrightarrow \Lambda' \longrightarrow P' \longrightarrow Q_{m} \longrightarrow 0,
\end{align}
where $Q_{m}$ is a rigid object equipped with a known $S$-exhaustive filtration. We want to formulate Question S for $P'$, now the only missing input is the receptional locus that $\Lambda'$ belongs to. We formalize this as follows.

There is a vector bundle $\mathcal{E}$ defined over an open subset of $\mbox{REC}\subset \mbox{Ext}^{1}(S\otimes A,T\otimes B)$, whose fiber over an extension
$$[0 \longrightarrow T\otimes B \longrightarrow \Lambda \longrightarrow S\otimes A \longrightarrow 0 ]=\delta \in \mbox{Ext}^{1}(S\otimes A, T\otimes B)$$ is naturally identified with $\mbox{Ext}^{1}(R_{m},\Lambda)$. On that fiber, a general extension
$$[0 \longrightarrow \Lambda \longrightarrow P \longrightarrow R_{m} \longrightarrow 0 ]=(\Delta, \delta) \in \mbox{Ext}^{1}(R_{m},\Lambda)$$
determines an extension class
$$[0 \longrightarrow S\otimes A \longrightarrow \Lambda' \longrightarrow T\otimes B'' \longrightarrow 0 ]=\delta'\in \mbox{Ext}^{1}(T\otimes B'',S\otimes A).$$
Hence there is a rational map $\pi: \mathcal{E} \dashrightarrow \mbox{Ext}^{1}(T\otimes B'', S\otimes A)$, where $\mathcal{E}$ is understood as the total space of the vector bundle. We define the receptive locus $\mbox{REC}'\subset \mbox{Ext}^{1}(T\otimes B'',S\otimes A)$ for $\Lambda$ to be the closure of $\pi(\mathcal{E})$.

Let $N=\mbox{Hom}(T,R_{m})$ and $M=\mbox{Hom}(S,Q_{m})$. The rational map $\pi$ factors through $\mbox{Ext}^{1}(T\otimes N, S\otimes A)$ by cutting down $T\otimes B$ and $Q_{m}$ from $P$, denote the image by $\overline{\mbox{REC}}\subset N^{*}VA$. Note that the map $\mbox{Ext}^{1}(R_{m},\Lambda) \longrightarrow \mbox{Ext}^{1}(T\otimes N,S\otimes A)$ is actually linear and independent of the entension class $\Delta\in \mbox{Ext}^{1}(R_{m},\Lambda)$. We can describe its image by the following commutative diagram
\[\begin{tikzcd} & \mbox{Ext}^{1}(Q_{m},S\otimes A) \arrow[d] & \mbox{Ext}^{2}(Q_{m},T\otimes B)=0 \arrow[d]\\
  \mbox{Ext}^{1}(R_{m},\Lambda) \arrow[r] & \mbox{Ext}^{1}(R_{m},S\otimes A) \arrow[r]\arrow[d] & \mbox{Ext}^{2}(R_{m},T\otimes B)\arrow[d]\\
  & \mbox{Ext}^{1}(T\otimes N, S\otimes A)\arrow[r]\arrow[d] & \mbox{Ext}^{2}(T\otimes N, T\otimes B) \arrow[d]\\
  & \mbox{Ext}^{2}(Q_{m},S\otimes A) & 0
\end{tikzcd}.\]
The image of $\mbox{Ext}^{1}(R_{m},\Lambda)$ in $N^{*}VA$ is the intersection
$$\mbox{ker}(\mbox{id}\otimes \delta_{1}: N^{*}VA \rightarrow N^{*}B)\cap \mbox{ker}(\delta_{2}\otimes \mbox{id}: N^{*}VA \rightarrow M^{*}A),$$
where $\delta_{2}\in \mbox{ADM}_{m,m}\subset M^{*}V^{*}N$ and $\delta_{1}\in \mbox{REC}\subset A^{*}V^{*}B$ are general. Consider the incidence correspondence in $M^{*}V^{*}N\times N^{*}VA \times A^{*}V^{*}B$:
$$\Gamma = \{(\delta_{1}, \delta, \delta_{2})\in \mbox{REC}\times N^{*}VA \times \mbox{ADM}: m_{1}(\delta_{1}, \delta)=0, m_{2}(\delta, \delta_{2})=0\},$$
where $m_{1}: M^{*}V^{*}N \times N^{*}VA \longrightarrow M^{*}\otimes A$ and $m_{2}: N^{*}VA \times A^{*}V^{*}B \longrightarrow N^{*}\otimes B$ are natural pairings. Let $p_{1}, p, p_{2}$ be projections to $M^{*}V^{*}N, N^{*}VA, A^{*}V^{*}B$ respectively. Since $\{(\delta_{1}, 0, \delta_{2}): \delta_{1}\in \mbox{REC}, \delta_{2}\in \mbox{ADM}\}$ is contained in $\Gamma$, the projection $p_{1}\times \mbox{id} \times p_{2}$ is dominant. By Lemma \ref{component}, there is a distinguished component of $\Gamma$, this component is clearly $\overline{\mbox{REC}}$. We have shown

\begin{proposition} \label{barREC}
Under the notations as above, $\overline{\mbox{REC}}$ is the distinguished component of $\Gamma$ over $\mbox{REC}\times \mbox{ADM}$ under the projection $p_{1}\times p_{2}$. 
\end{proposition}

Finally, we use the same trick as before: shrink an open set in $\overline{\mbox{REC}}$ if necessary, we may choose any decomposition $N=B'\oplus B''$ and let $q: N^{*} \longrightarrow B''^{*}$ be the projection. Then $REC'$ is the closure of the image of $\overline{\mbox{REC}}$ under the projection $q\otimes \mbox{id}$. We have 

\begin{proposition}\label{REC'}
  Let $\pi: \mathcal{E} \longrightarrow \mbox{Ext}^{1}(T\otimes B'', S\otimes A=B''^{*}VA)$ as above. Choose any projection $q: N^{*} \longrightarrow B''^{*}$. Then the image of $\pi$ is
  $$\mbox{REC}'= (q\times \mbox{id})(\overline{\mbox{REC}})\subset B''^{*}VA,$$
  where $\overline{\mbox{REC}}$ is computed in Proposition \ref{barREC}. 
\end{proposition}
Hence we reduced solving Question T for $P$ to solving Question S for $P'$, which is defined in (\ref{P'}). For a $\sigma_{0}$-semistable object $E$, let $l(E)$ be the number of $\sigma_{0}$-Jordan-H\"{o}lder factors of $E$. Note that $l(P')<l(P)$, hence in finite steps we have $l(P)=0$. Question S and Question T are solved. We summarize this as follows:

\begin{algorithm}[Concatenation]\label{concatenation}
  Let $R_{m}\in \mathcal{A}_{\mathcal{H}}$ be a rigid object with a known $T$-exhaustive filtration labeled as in Definition \ref{exhaustive}, let $N=\mbox{Hom}(T,R_{m})$ and $M=\mbox{Hom}(S,Q_{m})$. Let $\Lambda \in \mbox{REC}\subset \mbox{Ext}^{1}(S\otimes A, T\otimes B)$ be a general receptive extension in the known receptive locus $\mbox{REC}$. Let $P\in \mbox{Ext}^{1}(R_{m},\Lambda)$ be a general extension. The objective is to compute the $T$-exhaustive filtration of $P$.
  \begin{enumerate}
  \item Compute $\overline{\mbox{REC}}\subset N^{*}VA$ by Proposition \ref{barREC}.
  \item Compute $\mbox{REC}'\subset B''^{*}VA$ by Proposition \ref{REC'}. 
  \item Let $\Lambda'\in \mbox{REC}'\subset B''^{*}VA$ be a general receptive extension. Under Notation \ref{notationexhaustive}, $Q_{m}$ is a rigid object with a known $S$-exhaustive filtration. Then $Q=P/(T\otimes \mbox{Hom}(T,P))$ is a general extension in $\mbox{Ext}^{1}(Q_{m},\Lambda')$. An $S$-exhaustive filtration of $Q$ lifts to a $T$-exhaustive filtration of $P$. 
  \item If $Q=0$, the algorithm terminates. Otherwise, do (1)-(4) for $Q$. 
  \end{enumerate}
Since by doing (1)-(4) each time, $l(Q)<l(P)$, the algorithm terminates in finitely many steps.
\end{algorithm}

\subsection{Local Reduction}

In this subsection we state the local reduction that is used in the global reduction (Algorithm \ref{global reduction}). 

First we need to know the $S$-exhaustive filtration of $E$, provided its $\sigma_{+}$-Harder-Narasimhan filtration.

\begin{algorithm}[Initial filtration]\label{initial filtration}
Let
  $$0=E_{0}\subset E_{1}\subset \cdots E_{n}=E$$
  be the $\sigma_{+}$-Harder-Narasimhan filtration of $E$. Let $Q_{i}=E/E_{i}$, $G_{i}=E_{i}/E_{i-1}$. The objective is to find the $S$-exhaustive filtration of $E=Q_{0}$.

  Clearly $Q_{n}=0$ is an $S$-exhaustive filtration.
  \begin{enumerate}
  \item Suppose we know the $S$-exhaustive filtration of $Q_{i}$. Let $\Lambda_{i}=G_{i}$, $A_{i}=\mbox{Hom}(S,\Lambda_{i})$, $B_{i}=\mbox{Hom}(\Lambda_{i}, T)^{*}$.
  \item Then $\Lambda_{i}\in \mbox{REC}=\mbox{Ext}^{1}(T\otimes B_{i}, S\otimes A_{i})$ is a general extension, $Q_{i-1}\in \mbox{Ext}^{1}(Q_{i},G_{i})$ is a general extension, we do the concatenation (Algorithm \ref{concatenation}) on $Q_{i-1}$ to get the $S$-exhaustive filtration of $Q_{i-1}$.
  \item If $i-1=0$, the algorithm terminates. Otherwise, do (1)-(3) again.
  \end{enumerate}
\end{algorithm}

Now we can state local reduction formally.

\begin{algorithm}[Local reduction]\label{local reduction}
  Let $E\in \mathcal{A}_{\mathcal{H}}$ be a rigid object whose $\sigma_{+}$-Harder-Narasimhan filtration is known. The objective is to find its $\sigma_{-}$-Harder-Narasimhan filtration.
  \begin{enumerate}
  \item Initially we find the $S$-exhaustive filtration of $E$ by using Algorithm \ref{initial filtration}.
  \item The first $\sigma_{-}$-Harder-Narasimhan factor $E_{1}^{-}$ is the first factor of the $S$-exhaustive filtration. Let $Q=Q_{1}^{-}=E/E_{1}^{-}$. It has a known $T$-exhaustive filtration
    $$0 = F_{n}\subset E_{n-1}\subset F_{n-1}\subset \cdots \subset E_{1}\subset F_{1} \subset E_{0}=Q.$$
  \item We find the first $\sigma_{-}$-Harder-Narasimhan factor $O$ of $Q$ by Proposition \ref{first factor}, and note that $O\otimes \mbox{Hom}(O,Q)\subset \Gamma = F_{n-1}/F_{n}$. Let $\Lambda = \Gamma / O\otimes \mbox{Hom}(O,Q)$, then
    $$\Lambda \in \mbox{REC}\subset \mbox{Ext}^{1}(S\otimes A'', T\otimes B'')$$
    is a general receptive extension, where $A'', B''$ are vector spaces with known dimension by Proposition \ref{first factor}. The receptive locus $\mbox{REC}$ is given by Proposition \ref{REC1}.
  \item Let $R_{n-1}=Q/F_{n-1}$ and $Q' =Q/ (O\otimes \mbox{Hom}(O,Q))$. Then $Q'\in \mbox{Ext}^{1}(R_{n-1},\Lambda)$ is a general extension, where $\Lambda\in \mbox{REC}\subset A''^{*}V^{*}B''$ is a general receptive extension whose receptive locus is computed in step (3), and $R_{n-1}$ has the induced $T$-exhaustive filtration that is computed in step (2). We do the concatenation (Algorithm \ref{concatenation}) to get the $T$-exhaustive filtration of $Q'$.
  \item If $Q'=0$, the algorithm terminates. Otherwise, let $Q=Q'$ and do (3)-(5) again. 
  \end{enumerate}
\end{algorithm}

We briefly summarize local reduction (Algorithm \ref{local reduction}) in words. We start with a rigid object $E=Q_{0}^{-}$ with a known $\sigma_{+}$-Harder-Narasimhan filtration. First we solve Question S repeatedly to get an $S$-exhaustive filtration of $E$. Quotient out the first factor, we get $Q_{1}^{-}$, which has a known $T$-exhaustive filtration, which is the original $S$-exhaustive filtration with the first factor deleted. For any $i\geq 1$, once we computed the first $\sigma_{-}$-Harder-Narasimhan factor of $Q_{i}^{-}$ and take quotient $Q_{i+1}^{-}$, we fall into Question T. This can be solved, and we get a $T$-exhaustive filtration for $Q_{i+1}^{-}$, and the process is repeated. Since the length of the rigid object reduces at every step, finally this terminates and we get the $\sigma_{-}$-Harder-Narasimhan filtration of $E$.

Next we show an example of carrying out local reduction.

\begin{example}\label{hard}
  Let $(X,H)$ be a K3 surface with $\mbox{Pic}(X)=\mathbb{Z}H$ and $H^{2}=2$. Let $S\in M_{H}((5,2H,1)), T=\mathcal{O}_{X}[1]$ and $W=W(T,S)$ be the numerical wall defined by $T$ and $S$, and $\sigma_{+}$ (resp. $\sigma_{-}$) a stability condition on $\mathbf{b}$ (see Section \ref{section3}) that is right above (resp. below) $W$. Let $V=\mbox{Ext}^{1}(T,S)\cong \mathbb{C}^{6}$, $B_{0}=\mathbb{C}^{35}$, $A_{1}=\mathbb{C}^{6}$, $B_{1}=\mathbb{C}^{155}$. Let $Q_{1}\in M_{\sigma_{+}}((-5,12H,-29))$ and $E_{1}=T\otimes B_{1}$.

  Let $E$ be a general extension of $Q_{1}$ by $E_{1}$. One may check that $E$ is rigid, by noticing that $E$ is the quotient of $E'\in M_{H}((305, 477H, 746))$ by its first $\sigma_{+}$-Harder-Narasimhan factor. The shape (Definition \ref{shape}) of $\sigma_{+}$-Harder-Narasimhan filtration of $E$ is
  $$(155T, Q_{1}).$$
  
  We run local reduction (Algorithm \ref{local reduction}) to compute the $\sigma_{-}$-Harder-Narasimhan filtration of $E$. The $\sigma_{-}$-Harder-Narasimhan filtration of $Q_{1}$ has shape $(6S, 35T)$, hence there is a filtration of $E$ (not necessarily exhaustive at this moment) whose shape is
  $$E^{\bullet}=(155T, 6S, 35T).$$
  We first show that $E^{\bullet}$ is exhaustive. By the construction of exhaustive filtration (Lemma \ref{exist}), it suffices to show $\mbox{Hom}(T, E)=155$. Then the claim follows from Corollary \ref{JH} that the $\sigma_{-}$-Harder-Narasimhan filtration of $Q_{1}$ is exhaustive. We have
  $$0 \longrightarrow \mbox{Hom}(T,T\otimes B_{1}) \longrightarrow \mbox{Hom}(T,E) \longrightarrow \mbox{Hom}(T,Q_{1})=0. $$
  Hence $E^{\bullet}$ is exhaustive. This finishes step (1) in the local reduction (Algorithm \ref{local reduction}): the initial filtration (Algorithm \ref{initial filtration}) is $E^{\bullet}$.

  The $Q_{1}^{-}$ in step (2) of local reduction (Algorithm \ref{local reduction}) is $T\otimes B_{0}$.

 We now run step (3). Let $\delta_{0,1}\in \mbox{ADM}_{0,1}\subset B_{0}^{*}VA_{1}$ be a general element. Then the $\mbox{ker}(\delta_{0,1})$ is a general subspace of $A_{1}^{*}V^{*}$ of dimension 1, where $\delta_{0,1}$ is viewed as the adjoint map $\delta_{0,1}: A_{1}^{*}V^{*} \overset{\delta_{0,1}}{\longrightarrow} B_{0}^{*}$. By the definition of admissible extensions, $\mbox{ADM}_{1,1} \subset A_{1}^{*}V^{*}B_{1}$ is the closure of the union of $\mbox{ker}(\delta_{0,1})\otimes B_{1}$ for general $\delta_{0,1}$. Since $\mbox{ker}(\delta_{0,1})$ is a general 1 dimensional subspace, we have
 $$\mbox{ADM}_{1,1}=\{\delta_{1,1}\in A_{1}^{*}V^{*}B_{1}: \mbox{rk}_{(AV, B)}(\delta_{1,1})\leq 1\},$$
 where $\mbox{rk}_{(AV,B)}(\delta_{1,1})$ is the rank of $\delta_{1,1}$ viewed as a 2-tensor in $(A_{1}^{*}V^{*})\otimes B_{1}$.

  Next we compute the $k(i)$ in Proposition \ref{first factor}. Note that $\mbox{rk}_{(AV,B)}(\delta_{1,1})=1$ for a general $\delta_{1,1}\in \mbox{ADM}_{1,1}$, the map $\delta_{1,1}: A_{1} \longrightarrow V^{*}B_{1}$ is the composition of a general map
  $\delta_{1,1}': A_{1} \longrightarrow V^{*}B_{1}'$ and the inclusioin $i\otimes \mbox{id}: V^{*}B_{1}' \longrightarrow V^{*}B_{1}$, where $i: B_{1}' \longrightarrow B_{1}$ is a one dimensional subspace. Using the notation of Proposition \ref{first factor}, we have
  $$J_{i}=(\delta(N_{i})\otimes B_{1})\cap (M_{i}V^{*}B_{1}')\cap (M_{i}\otimes \delta_{1,1}'(A_{1}))=(\delta(N_{i})\otimes B_{1}')\cap (M_{i}\otimes \delta_{1,1}'(A_{1})).$$
  Since $\mbox{dim}B_{1}'=1$, $\delta_{1,1}'(A_{1})$ can be viewed as a general subspace of $V$. Since $\mbox{dim}A_{1}=\mbox{dim}V=6$, $\delta_{1,1}'(A_{1})=V\otimes B_{1}'$. We omit writing $B_{1}'$ since its dimension is 1. Then $J_{i}=\delta(N_{i})\cap M_{i}V=\delta(N_{i})$. Since $E^{\bullet}$ is exhaustive, $\mbox{Hom}(O_{0}, E)=0$. Hence the $i$ with maximal $\phi_{-}(i)$ is $i=-1$, and $O_{-1}\in M_{\sigma_{-}}((29, 12H, 5))$.  By Theorem \ref{numbers}, we have
  $$J_{-1}=\delta(N_{-1})=\mathbb{C},$$
  hence $k(-1)=1$. Using the notation in step (3), $\Lambda=T^{\oplus 154}$ and there is no receptive locus. We completed step (3) in local reduction (Algorithm \ref{local reduction}).

  Then in step (4), $Q'$ is a general extension of $R=T\otimes B_{0}$ by $\Lambda= T^{\oplus 154}$. Hence $Q'=T^{\oplus 189}$. We run the local reduction again for $Q'$, but there is nothing to do. Hence the algorithm terminates. The $\sigma_{-}$-Harder-Narasimhan filtration of $E$ has shape
  $$(S_{-1}^{-}, 189 T),$$
  where $S_{-1}^{-}\in M_{\sigma_{-}}((29, 12H, 5))$.
\end{example}

\begin{remark}
  The local reduction (Algorithm \ref{local reduction}) gives a theoretical way to compute the $\sigma_{-}$-Harder-Narasimhan filtration of a rigid object from its $\sigma_{+}$-Harder-Narasimhan filtration in principle. However as one may have noticed in Example \ref{hard}, in practice it could be hard to carry out by hand. The main computational difficulties are writing down the ideal of the admissible locus (Proposition \ref{admknown}) and computing the functions $k(i)$ (Proposition \ref{first factor}).

  In many cases, we may overcome these computational difficulties by using a simplified version (Theorem \ref{local simplification}) of local reduction, which is discussed in the next section. We will revist Example \ref{hard} in Example \ref{easy} by using the simplified local reduction.
\end{remark}

\section{Simplifications}\label{section6}

In many good cases, the algorithm can be significantly simplified. This includes two aspects. Locally, if at each wall the $\sigma_{+}$-Harder-Narasimhan filtration has only two factors, then the local reduction (Algorithm \ref{local reduction}) is significantly simplified (Theorem \ref{local simplification}) and can be computed by hand. The local complexity of a spherical object is dominated by the global complexity, which is measured by a notion called \emph{height} (Definition \ref{ht}). For a spherical object of height $\leq 2$, we show a straightforward inductive formula (Theorem \ref{global simplification}).

In practical terms, among all examples that the author did, Theorem \ref{local simplification} can be applied to most of them. However, there exist Mukai vectors that need to fully invoke the global reduction (Algorithm \ref{global reduction}), see Example \ref{3step}.

\subsection{Local Simplification}

The goal of this subsection is to find a simple way to compute the $\sigma_{-}$-Harder-Narasimhan filtration of those rigid objects $E$ that fits into the following exact sequence
$$0 \longrightarrow F^{q} \longrightarrow E \longrightarrow G^{p} \longrightarrow 0, $$
  where $F, G\in \mathcal{H}$ are $\sigma_{+}$-stable spherical objects with $\phi_{+}(F)>\phi_{+}(G)$. This is Theorem \ref{local simplification}.

Let $W$ be a wall whose corresponding rank 2 lattice is $\mathcal{H}$. Let $S_{0}, T_{0}$ be the two $\sigma_{0}$-stable objects in $\mathcal{H}$ where $\sigma_{0}\in \mathcal{H}$. Let $\sigma_{+}, \sigma_{-}$ be two
stability conditions on two sides of $W$ and near $W$, with $\phi_{+}(T_{1})>\phi_{+}(S_{0})$. Recall that the \emph{fundamental sequence} of $\mathcal{H}$ is defined in Definition \ref{fundamental sequence} as
$$a_{0}=1,~ a_{1}=g=\mbox{ext}^{1}(T_{0},S_{0}),~ a_{n}=g\cdot a_{n-1}-a_{n-2} \mbox{ for } n\geq 2.$$
For the rest of the paper, we sometimes write $S=S_{0}$ and $T=T_{0}$.

 We first consider rigid objects $E$ that can fit into an exact sequence
$$0 \longrightarrow S^{p} \longrightarrow E \longrightarrow T^{q} \longrightarrow 0. $$
Since $E$ is rigid, we may assume
$$[0 \longrightarrow S^{p} \longrightarrow E \longrightarrow T^{q} \longrightarrow 0]=\delta \in \mbox{Ext}^{1}(T^{q},S^{p})$$
is a general extension. We call such rigid objects of \emph{type I}.

\begin{lemma}\label{unbalance}
  Let $E$ be a rigid object that fits into an exact sequence
  $$0 \longrightarrow S^{p} \longrightarrow E \longrightarrow T^{q} \longrightarrow 0. $$
Then 
  $$pq\cdot g < q^{2}+p^{2}.$$
\end{lemma}

\begin{proof}
  Let
  $$\eta_{i}=[0 \longrightarrow S^{p} \overset{f_{i}}{\longrightarrow} E_{i} \overset{g_{i}}{\longrightarrow} T^{q} \longrightarrow 0]\in \mbox{Ext}^{1}(T^{q},S^{p}), i=1,2$$
  be two general extensions. Since $E$ is rigid and $\delta$ is general, $E_{1}\cong E_{2}$. Consider the following diagram
  \[\begin{tikzcd}
      0\arrow[r] & S^{p} \arrow[r,"f_{1}"] \arrow[d,dashed, "\phi"] &E_{1} \arrow[r, "g_{1}"] \arrow[d, equal] & T^{q} \arrow[r] \arrow[d, dashed, "\psi"] & 0\\
      0\arrow[r] & S^{p} \arrow[r, "f_{2}"] & E_{2} \arrow[r, "g_{2}"] & T^{q} \arrow[r] & 0.
    \end{tikzcd}\]
  Since $\mbox{Hom}(S, T)=0$, the composition map
  $$S^{p} \overset{f_{1}}{\longrightarrow} E_{1} \overset{=}{\longrightarrow} E_{2} \overset{g_{2}}{\longrightarrow} T^{q}$$
  is zero. Since $S^{p} \overset{f_{2}}{\longrightarrow} E_{2}$ is the kernel of $E_{2}\overset{g_{2}}{\longrightarrow} T^{q}$, there exists a unique map
  $\phi: S^{p} \longrightarrow S^{p}$ such that the diagram commutes. Swap the two rows in the diagram, we see that $\phi$ is an isomorphism. Since $S$ is simple, $\phi$ is an element in $\mbox{GL}_{p}$. Similarly, there exists a unique $\psi$ fitting into the diagram above, which is an element of $\mbox{GL}_{q}$.

  Modulo scalars, we see that an orbit of $\mathbb{P}\mbox{GL}_{p}\times \mathbb{P}\mbox{GL}_{q}$ action contains an open subset of $\mathbb{P}\mbox{Ext}^{1}(T^{q},S^{p})$. By comparing dimensions, we have
  $$\mbox{ext}^{1}(T^{q},S^{p})-1=\mbox{dim}\mathbb{P}\mbox{Ext}^{1}(T^{q},S^{p})\leq \mbox{dim}\mathbb{P}\mbox{GL}_{p}\times \mathbb{P}GL_{q}=p^{2}-1+q^{2}-1,$$
  namely
  $$\mbox{ext}^{1}(T^{q},S^{p})<q^{2}+p^{2}.$$
  
\end{proof}

We note a simple property of the fundamental sequence in the following lemma.

\begin{lemma}\label{decreasing}
  Assume $g(\mathcal{H})\geq 2$. Then
  $$0<\frac{a_{n+1}}{a_{n}}<\frac{a_{n}}{a_{n-1}}. $$
\end{lemma}

\begin{proof}
  For $n\in \mathbb{Z}_{\geq 1}$, let $b_{n}=\frac{a_{n}}{a_{n-1}}$. Clearly $b_{n}>0$. From the definition
  $$a_{n+1}=g\cdot a_{n}-a_{n-1},$$
  we have
  $$b_{n+1}=g-\frac{1}{b_{n}}.$$
  Consider
  $$b_{n+1}-b_{n}=-\frac{1}{b_{n}}+\frac{1}{b_{n-1}},$$
  if $b_{n}<b_{n-1}$, then $b_{n+1}<b_{n}$. Now
  $b_{2}=g-\frac{1}{g}< g=b_{1}$, the lemma is proved by induction.
\end{proof}

\begin{remark}\label{lim}
An immediate consequence of the lemma is that $\{\frac{a_{n+1}}{a_{n}} \}$ monotonically approaches its limit from above. The limit is the smaller root of the equation $x=g-\frac{1}{x}$, which we denote by $\rho(\mathcal{H})=\rho$.
\end{remark}

\begin{proposition}\label{type1}
  Assume $g(\mathcal{H})\geq 2$. Let $E$ be a rigid object that fits into an exact sequence
  $$0 \longrightarrow T_{1}^{q} \longrightarrow E \longrightarrow S_{0}^{p} \longrightarrow 0. $$
  
  If $q\leq p$, then the $\sigma_{-}$-Harder-Narasimhan filtration of $E$ is given by
  $$E= (S_{-i}^{-})^{m} \oplus (S_{-i-1}^{-})^{n},$$
  where $i$ is the unique integer such that $\frac{a_{i+1}}{a_{i}}\leq \frac{p}{q} < \frac{a_{i}}{a_{i-1}}$, and $m,n$ are the unique integers such that
  $$m\cdot v(S_{-i}^{-})+n\cdot v(S_{-i-1}^{-})=v(E).$$
  
  Symmetrically if $q\geq p$, then the $\sigma_{-}$-Harder-Narasimhan filtration of $E$ is given by
  $$E= (T_{i+1}^{-})^{m} \oplus (T_{i+2}^{-})^{n} ,$$
  where $i$ is the unique integer such that $\frac{a_{i+1}}{a_{i}}\leq \frac{q}{p} < \frac{a_{i}}{a_{i-1}}$, and $m,n$ are the unique integers such that
  $$m\cdot v(T_{i+1}^{-}) + n\cdot v(T_{i+2}^{-})=v(E).$$
\end{proposition}

\begin{proof}
  We only prove the case $q\geq p$, the other case is similar. By Lemma \ref{unbalance}, we have $pq\cdot g< p^{2}+q^{2}$, which can be rewritten as
  $$g<\frac{p}{q}+\frac{q}{p}. $$
  By Remark \ref{lim}, $\rho$ is the smaller root of the equation $x+\frac{1}{x}=g$, and by assumption $q\geq p$, hence we have $\frac{q}{p}>\rho$. Now since
  $\frac{a_{n+1}}{a_{n}} $ approaches $\rho$ monotonically from above, we may divide the interval $(\rho, +\infty)$ into disjoint subintervals
  $$(\rho,+\infty)=\coprod_{i\geq 0}[\frac{a_{i+1}}{a_{i}}, \frac{a_{i}}{a_{i-1}}), $$
  here we set $\frac{a_{0}}{a_{-1}}=+\infty$. Since $\frac{q}{p}>\rho $, there exists a unique $i$ such that $\frac{a_{i+1}}{a_{i}}\leq \frac{q}{p} < \frac{a_{i}}{a_{i-1}} $. In other words, on
  $\mathbb{Z}^{2}$ the point $(p,q)$ is in between the positive rays spanned by $(a_{i-1}, a_{i})$ and $(a_{i}, a_{i+1})$.

  Next we observe that $(a_{i-1}, a_{i}), (a_{i}, a_{i+1})$ form a basis of $\mathbb{Z}^{2}$. To see this, we compute
  \begin{eqnarray*}
   & & \mbox{det} \begin{pmatrix}
                 a_{i-1} & a_{i} \\
                 a_{i} & a_{i+1}
               \end{pmatrix}\\
                         & = & a_{i-1}a_{i+1}-a_{i}^{2}\\
                         &= &(ga_{i}-a_{i-1})a_{i-1}-(ga_{i-1}-a_{i-2})a_{i}\\
    &=&a_{i-2}a_{i}-a_{i-1}^{2}\\
    &=&\mbox{det}
               \begin{pmatrix}
                 a_{i-2} & a_{i-1} \\
                 a_{i-1} & a_{i}
               \end{pmatrix}\\
    & \vdots & \\
    &=& \mbox{det}
        \begin{pmatrix}
          a_{0} & a_{1} \\
          a_{1} & a_{2}
        \end{pmatrix}\\
    &=& 1\cdot (g^{2}-1)-g^{2}=-1.
  \end{eqnarray*}
  Hence there exists a unique pair of non-negative integers $m, n$ such that $(p,q)=m(a_{i-1}, a_{i})+n(a_{i}, a_{i+1})$.

  Next we consider the object $(T_{i+1}^{-})^{m}\oplus (T_{i+2}^{-})^{n}$. By Corollary \ref{JH}, for any $j$, $T_{j}^{-}$ fits into a general extension
  $$0 \longrightarrow T_{1}^{a_{j-1}} \longrightarrow T_{j}^{-} \longrightarrow S_{0}^{a_{j-2}} \longrightarrow 0. $$
  Hence we have the following exact sequence
  $$0 \longrightarrow T_{1}^{ma_{i}+na_{i+1}} \longrightarrow (T_{i+1}^{-})^{m}\oplus (T_{i+2}^{-})^{n} \longrightarrow S_{0}^{ma_{i-1}+na_{i}} \longrightarrow 0 .$$
  Since $(p,q)=m(a_{i-1}, a_{i})+n(a_{i}, a_{i+1})$, the sequence above also reads
  $$0 \longrightarrow T_{1}^{q} \longrightarrow (T_{i+1}^{-})^{m}\oplus (T_{i+2}^{-})^{n} \longrightarrow S_{0}^{p} \longrightarrow 0 .$$
Note that $\mbox{Ext}^{1}(T_{i+1}^{-}, T_{i+2}^{-})=0$, hence
  $$\mbox{Ext}^{1}((T_{i+1}^{-})^{m}\oplus (T_{i+2}^{-})^{n},(T_{i+1}^{-})^{m}\oplus (T_{i+2}^{-})^{n})=0,$$
  namely $(T_{i+1}^{-})^{m}\oplus (T_{i+2}^{-})^{n}$ is rigid, hence it fits into a general extension
  $$[0 \longrightarrow T_{1}^{q} \longrightarrow (T_{i+1}^{-})^{m}\oplus (T_{i+2}^{-})^{n} \longrightarrow S_{0}^{p} \longrightarrow 0] \in \mbox{Ext}^{1}(S_{0}^{p},T_{1}^{q})$$
  On the other hand, since $E$ is rigid, it also fits into a general extension
  $$[0 \longrightarrow T_{1}^{q} \longrightarrow E \longrightarrow S_{0}^{p} \longrightarrow 0]\in \mbox{Ext}^{1}(S_{0}^{p}, T_{1}^{q}). $$
  Hence $E$ and $(T_{i+1}^{-})^{m}\oplus (T_{i+2}^{-})^{n}$ are isomorphic.

\end{proof}

When one Harder-Narasimhan factor of $E$ is not $S$ or $T$, we have the following proposition. 

\begin{proposition}\label{type2}
  Let $E$ be a rigid object.
  
  If $E$ fits into the following exact sequence
  $$0 \longrightarrow T_{1}^{q} \longrightarrow E \longrightarrow G^{p} \longrightarrow 0, $$
  where $G$ is some $\sigma_{+}$-stable spherical object in $\mathcal{H}$. Then the $\sigma_{-}$-Harder-Narasimhan filtration of $E$ is of the form
  $$0=E_{0}\subset E_{1} \subset E_{2} \subset E_{3}=E,$$
  where $E_{3}/E_{2}=T_{1}^{\mbox{hom}(G^{p},T_{1})+\varepsilon}$ and $\varepsilon = \mbox{max}\{q-\mbox{ext}^{1}(G^{p},T_{1}), 0\}$, and $E_{2}$ is a rigid object of type I: $E_{2}$ fits into
  $$0 \longrightarrow T_{1}^{q-\varepsilon} \longrightarrow E_{2} \longrightarrow S_{0}\otimes \mbox{Hom}(S_{0}, G^{p}) \longrightarrow 0 .$$

  Similarly, if $E$ fits into the following exact sequence
  $$0 \longrightarrow F^{q} \longrightarrow E \longrightarrow S_{0}^{p} \longrightarrow 0,$$
  where $F$ is some $\sigma_{+}$-stable spherical object in $\mathcal{H}$. Then the $\sigma_{-}$-Harder-Narasimhan filtration of $E$ is of the form
  $$0=E_{0}\subset E_{1}\subset E_{2}\subset E_{3}=E,$$
  where $E_{1}=S_{0}^{\mbox{hom}(S_{0}, F^{q})+\varepsilon}$ and $\varepsilon = \mbox{max}\{p-\mbox{ext}^{1}(S_{0},F^{q}), 0 \}$, and $E_{3}/E_{1}$ is a rigid object of type I: $E_{3}/E_{1}$ fits into
  $$0 \longrightarrow T_{1}\otimes \mbox{Hom}(F^{q}, T_{1})^{*} \longrightarrow E_{3}/E_{1} \longrightarrow S_{0}^{p-\varepsilon} \longrightarrow 0 $$
\end{proposition}

We shall refer to the rigid objects $E$ in Proposition \ref{type2} as \emph{type II}. 

\begin{proof}
  We only prove the first case, the other case is similar. We may assume $G\neq S_{0}$, otherwise $E$ is of type I and the Harder-Narasimhan filtration is already known. In particular, $\mbox{Hom}(E, T_{1})\neq 0$.

  By Mukai's Lemma (Lemma \ref{Mukai's Lemma}), $E_{3}/E_{2}$ is a direct sum of a $\sigma_{-}$-stable spherical object in $\mathcal{H}$, hence $\phi_{-}(E_{3}/E_{2})\geq \phi_{-}(T_{1}) $. On the other hand, $\phi_{-}(E_{i}/E_{i-1})>\phi_{-}(E_{3}/E_{2})\geq \phi_{-}(T_{1})$ for $i<3$, hence $\mbox{Hom}(E_{2},T_{1})=0$. Since $\mbox{Hom}(E,T_{1})\neq 0$, we know $\mbox{Hom}(E_{3}/E_{2}, T_{1})\neq 0$, $\phi_{-}(E_{3}/E_{2})\leq \phi_{-}(T_{1})$, hence $E_{3}/ E_{2}$ is naturally isomorphic to $T_{1}\otimes \mbox{Hom}(E,T_{1})^{*}$. Applying $\mbox{Hom}(-,T_{1})$ to
  $$0 \longrightarrow T_{1}^{q} \longrightarrow E \longrightarrow G^{p} \longrightarrow 0 $$
  and note that we may assume this extension $\delta \in \mbox{Ext}^{1}( G^{p}, T_{1}^{q})$ is general since $E$ is rigid, we get
  $$0 \longrightarrow \mbox{Hom}(G^{p},T_{1}) \longrightarrow \mbox{Hom}(E, T_{1}) \longrightarrow \mbox{Hom}(T_{1}^{q}, T_{1}) \overset{\delta^{*}}{\longrightarrow}\mbox{Ext}^{1}(G^{p}, T_{1}).$$
  Note that we have the natural identification
  $$\mbox{Ext}^{1}(G^{p}, T_{1}^{q})\cong \mbox{Hom}(\mbox{Hom}(T_{1}^{q}, T_{1}),\mbox{Ext}^{1}(G^{p},T_{1}))$$
  sending $\delta$ to $\delta^{*}$. Since $\delta^{*}$ is general, the kernel has dimension $\varepsilon=\mbox{max}\{q-\mbox{ext}^{1}(G^{p},T_{1}), 0\}$, and $\mbox{Hom}(E,T_{1})\cong \mathbb{C}^{p\cdot\mbox{hom}(G, T_{1})+\varepsilon}$.
  Now we have the following commutative diagram
  \[\begin{tikzcd}
      & T_{1}^{q-\varepsilon} \arrow[d] & E_{2} \arrow[d] & S_{0}\otimes \mbox{Hom}(S_{0}, G^{p}) \arrow[d] & \\
      0 \arrow[r] & T_{1}^{q} \arrow[r]\arrow[d] & E \arrow[r]\arrow[d] & G^{p} \arrow[r]\arrow[d] & 0\\
      0 \arrow[r] & T_{1}^{\varepsilon} \arrow[r]\arrow[d] & T_{1}\otimes \mbox{Hom}(E, T_{1})^{*} \arrow[r]\arrow[d] & T_{1}\otimes \mbox{Hom}(G^{p}, T_{1})^{*} \arrow[r]\arrow[d] & 0\\
      & 0 & 0 & 0 &  ,
    \end{tikzcd}\]
  where the rows and columns are all exact. Hence $E_{2}$ fits into
  $$0 \longrightarrow T_{1}^{q-\varepsilon} \longrightarrow E_{2} \longrightarrow S_{0}\otimes \mbox{Hom}(S_{0},G^{p}) \longrightarrow 0.$$
  By Mukai's Lemma (Lemma \ref{Mukai's Lemma}), $E_{2}$ is rigid, hence a rigid object of type I. 
\end{proof}

When neither Harder-Narasimhan factors are $S$ or $T$, the key technical proposition is the following

\begin{proposition}\label{injectivity}
  Let $E$ be a rigid object that fits into the following exact sequence
  $$0 \longrightarrow F^{q} \longrightarrow E \longrightarrow G^{p} \longrightarrow 0 ,$$
  where $F, G\in \mathcal{H}$ are $\sigma_{+}$-stable spherical objects with $\phi_{+}(F)>\phi_{+}(G)$ and $\mbox{Ext}^{1}(G,F)\neq 0$. Then either
  $$\mbox{Hom}(S_{0},G^{p}) \longrightarrow \mbox{Ext}^{1}(S_{0},F^{q})$$
  is injective, or
  $$\mbox{Hom}(F^{q}, T_{1}) \longrightarrow \mbox{Ext}^{1}(G^{p}, T_{1})$$
  is injective, where the two maps are the connecting homomorphisms in the long exact sequences.
\end{proposition}

We refer to the rigid objects in Proposition \ref{injectivity} as \emph{two step} rigid objects. Before proving Proposition \ref{injectivity}, we first state a lemma confirming that the dimensions of the spaces allow the map to be injective. It is also needed for the proof of Proposition \ref{injectivity}.

\begin{lemma}\label{numerical}
  If $\mbox{Ext}^{1}(G,F) \neq 0$, then either
  $$\mbox{hom}(S_{0}, G^{p})\leq \mbox{ext}^{1}(S_{0}, F^{q}),$$
  or
  $$\mbox{hom}(F^{q}, T_{1})\leq \mbox{ext}^{1}(G^{p},T_{1}).$$
\end{lemma}

\begin{proof}
  Based on the condition that $\phi_{+}(F)> \phi_{+}(G)$, we separate into three cases. Case TS: $F=T_{i}$ and $G=S_{j}$, where $i\geq 1, j\leq 0$. Case SS: $F=S_{i}$ and $G=S_{j}$, where $i+2 \leq j\leq 0$. Case TT: $F=T_{i}$ and $G=T_{j}$, where $1\leq i \leq j-2$. (In case SS and case TT, $j\neq i+1$ because by assumption $\mbox{Ext}^{1}(G,F)\neq 0$.)

  We first analyze the the case TS. Recall that we have
  $$\mbox{hom}(S_{0}, S_{j})=a_{-j},~ \mbox{ext}^{1}(S_{0}, T_{i})=a_{i},~ \mbox{hom}(T_{i}, T_{1})=a_{i-1},~ \mbox{ext}^{1}(S_{j}, T_{1})=a_{-j+1}, $$
  where $\{a_{n}\}$ is the fundamental sequence of $\mathcal{H}$. Now suppose the lemma is not true, then we must have
  \begin{eqnarray*}p\cdot a_{-j}&> &q\cdot a_{i}, \\
    q\cdot a_{i-1} &>& p\cdot a_{-j+1}.\end{eqnarray*}
We may rewrite this as
$$\frac{a_{i}}{a_{-j}}<\frac{p}{q}<\frac{a_{i-1}}{a_{-j+1}}. $$
This is clearly impossible, since
$\frac{a_{i}}{a_{i-1}}>1>\frac{a_{-j}}{a_{-j+1}}$.
The lemma is proven for the case TS.

The case SS and the case TT are similar, we only prove for the case TT. We have
$$\mbox{hom}(S_{0}, T_{j})=a_{j-2}, ~\mbox{ext}^{1}(S_{0},T_{i})=a_{i},~ \mbox{hom}(T_{i}, T_{1})=a_{i-1},~ \mbox{ext}^{1}(T_{j},T_{1})=a_{j-3}.$$
If the lemma is not true, then we must have
\begin{eqnarray*}
  p\cdot a_{j-2} & > & q \cdot a_{i} \\
  q\cdot a_{i-1} & > & p \cdot a_{j-3},
\end{eqnarray*}
which can be rewritten as
$$\frac{a_{i}}{a_{j-2}}<\frac{p}{q}<\frac{a_{i-1}}{a_{j-3}}. $$
Hence $\frac{a_{i}}{a_{i-1}}<\frac{a_{j-2}}{a_{j-3}} $. By Lemma \ref{decreasing}, this is equivalent to $j-2<i$, which contradicts the assumption that $i\leq j-2$.
\end{proof}

In the following we will always assume $\mbox{hom}(S_{0}, G^{p})\leq \mbox{ext}^{1}(S_{0},F^{q})$, the other case can be treated in the same manner.

For any finite dimensional linear spaces $K\subset A\otimes B$, we let
$$\overline{K}=\bigcap_{K\subset A'\otimes B'\subset A\otimes B} A'\otimes B'.$$
Then $\overline{K}$ is also of the form $A_{0}\otimes B_{0}$ for some subspaces $A_{0}, B_{0}$.

In general, if $V, W$ are two linear spaces with $\mbox{dim}V \leq \mbox{dim}W$, and $f\in W\otimes V^{*}= \mbox{Hom}(V,W)$ is general, then $f: V \rightarrow W$ is injective. The following lemma shows that we may slightly weaken the condition.

\begin{lemma}\label{general}
Let $V,W$ be two linear spaces with $\mbox{dim}V \leq \mbox{dim}W$. Let $K\subset W\otimes V^{*}$ be a subspace such that $\overline{K}= W\otimes V^{*}$. Then for a general $f\in K$, $f: V \rightarrow W$ is injective.
\end{lemma}

\begin{proof}
For any $f\in K$, we may choose a base for $V$ and $W$ such that $f$ is a diagonal matrix. If $f$ is not full rank, then below the, say $n$-th row, all entries of the matrix are zero. Since $\overline{K}=W\otimes V^{*}$, there exists $g\in K$, such that in matrix form some entries below the $n$-th row is not zero. Then for sufficiently small $\varepsilon>0$, $f+\varepsilon\cdot g \in K$ has rank strictly bigger than $f$. We may then iterate this process to get an element in $K$ of full rank. Since $\mbox{dim}V\leq \mbox{dim}W$, a full rank map is injective.
\end{proof}

Applying $\mbox{Hom}(G,-)$ to
  $$0 \longrightarrow S_{0}\otimes \mbox{Hom}(S_{0},F) \longrightarrow F \longrightarrow T_{1}\otimes \mbox{Hom}(F,T_{1})^{*} \longrightarrow 0, $$
  we get long exact sequence
  $$ \mbox{Ext}^{1}(G,F) \rightarrow \mbox{Ext}^{1}(G,T_{1})\otimes \mbox{Hom}(F,T_{1})^{*} \longrightarrow \mbox{Hom}(S_{0},G)^{*}\otimes \mbox{Hom}(S_{0},F),$$
  let $K$ be the image of $\mbox{Ext}^{1}(G,F) \longrightarrow \mbox{Hom}(G,T_{1})\otimes \mbox{Hom}(F,T_{1})^{*}$.

Let $V=\mbox{Ext}^{1}(T_{1},S_{0})$. We write $V^{(n)}$ to be $V\otimes V^{*} \otimes \cdots \otimes V$ or $V\otimes V^{*} \otimes \cdots \otimes V^{*}$, where the total number of $V$ and $V^{*}$ is $n$. Let $K_{n}$ be the intersection of all kernels of parings on consecutive factors.

\begin{lemma}\label{full}
  Under the notations above, we have
  $$\overline{K}=\mbox{Ext}^{1}(S_{0},F)\otimes \mbox{Hom}(S_{0},G)^{*}. $$
\end{lemma}

\begin{proof}
  we separate into three cases. Case TS: $F=T_{i}$ and $G=S_{j}$, where $i\geq 1, j\leq 0$. Case SS: $F=S_{i}$ and $G=S_{j}$, where $i+2 \leq j\leq 0$. Case TT: $F=T_{i}$ and $G=T_{j}$, where $1\leq i \leq j-2$.

  In the case TS, $K$ is the kernel of the map
  $$\mbox{Ext}^{1}(S_{j},T_{1})\otimes \mbox{Hom}(T_{i},T_{1})^{*} \longrightarrow \mbox{Hom}(S_{0}, S_{j})^{*}\otimes \mbox{Hom}(S_{0},T_{i}).$$
  The map is the composition of
  $$ev^{*}\otimes \mbox{id}: \mbox{Ext}^{1}(S_{j},T_{1})\otimes \mbox{Hom}(T_{i},T_{1}) \longrightarrow \mbox{Hom}(S_{0},S_{j})^{*}\otimes \mbox{Ext}^{1}(S_{0}, T_{1})\otimes \mbox{Hom}(T_{i},T_{1})^{*}$$
  and
  $$\mbox{id}\otimes \delta: \mbox{Hom}(S_{0},S_{j})^{*}\otimes \mbox{Ext}^{1}(S_{0}, T_{1})\otimes \mbox{Hom}(T_{i},T_{1})^{*} \longrightarrow \mbox{Hom}(S_{0},S_{j})^{*}\mbox{Hom}(S_{0},T_{i}).$$
  By similar arguments as in Proposition \ref{description} and \ref{description2}, $ev^{*}: \mbox{Ext}^{1}(S_{j},T_{1}) \longrightarrow \mbox{Hom}(S_{0},S_{j})^{*}\otimes \mbox{Ext}^{1}(S_{0},T_{1})$ is injective and $\mbox{Ext}^{1}(S_{j},T_{1})=K_{-j+1}$ under the inclusion into $V^{(-j+1)}$.

  The map $\delta: \mbox{Ext}^{1}(S_{0},T_{1}) \otimes \mbox{Hom}(T_{i},T_{1})^{*} \longrightarrow \mbox{Hom}(S_{0}, T_{i})$ is represented by the class
  $$[0 \longrightarrow S_{0}\otimes \mbox{Hom}(S_{0},T_{i}) \longrightarrow T_{i} \longrightarrow T_{1}\otimes \mbox{Hom}(T_{i},T_{1})^{*} \longrightarrow 0 ]=\delta$$
  inside $ \mbox{Hom}(T_{i},T_{1})\otimes \mbox{Ext}^{1}(T_{1},S_{0})\otimes \mbox{Hom}(S_{0},T_{i})$.
  Moreover, we may choose a perfect pairing on $\mbox{Ext}^{1}(S_{0},T_{1})\otimes \mbox{Hom}(T_{2},T_{1})^{*}$, so that $\mbox{ker}(\delta)=K_{i}\subset V^{(i)}$. Hence inside $\mbox{Hom}(S_{0},S_{j})^{*}\otimes \mbox{Ext}^{1}(S_{0},T_{1})\otimes \mbox{Hom}(T_{i},T_{1})^{*}$, $K$ is the intersection
  $$K=(\mbox{Ext}^{1}(S_{j},T_{1})\otimes \mbox{Hom}(T_{i},T_{1})^{*}) \cap (\mbox{Hom}(S_{0}, S_{j})^{*}\otimes\mbox{ker}(\delta)).$$
  Under the inclusion
  $$\mbox{Hom}(S_{0},S_{j})^{*}\otimes \mbox{Ext}^{1}(S_{0},T_{1})\otimes \mbox{Hom}(T_{i}, T_{1})^{*}\subset V^{(-j+i)}, $$
  we see $K=K_{-j+i}$, $\mbox{Ext}^{1}(S_{j},T_{1})\otimes \mbox{Hom}(T_{i},T_{1})=K_{-j+1}\otimes K_{i-1}$. By linear algebra, one may check explicitly that $K_{-j+i}\subset K_{-j+1}\otimes K_{-i-1}$  is full on both factors.

 The case TT and the case SS can be treated in similar manners.
\end{proof}

\begin{remark}\label{full1}
  An immediate consequence of Lemma \ref{full} is that, for any positive integers $p, q$, we have
  $$\overline{K^{\oplus pq}}=\mbox{Ext}^{1}(S_{0},F^{q})\otimes \mbox{Hom}(S_{0},G^{p})^{*}.$$
\end{remark}

Now we are ready to prove Proposition \ref{injectivity}.

\begin{proof}
  By Lemma \ref{numerical}, we may assume $\mbox{hom}(S_{0}, G^{p})\leq \mbox{ext}^{1}(S_{0},F^{q})$, the other case is similar. Since $E$ is rigid, we may assume the extension
  $$[0 \longrightarrow F^{q} \longrightarrow E \longrightarrow G^{p} \longrightarrow 0 ]=\delta\in \mbox{Ext}^{1}(F^{q},G^{p})$$
  is general. Under the map
  $$\mbox{Ext}^{1}(G^{p},F^{q}) \longrightarrow \mbox{Ext}^{1}(S_{0}, F^{q})\otimes \mbox{Hom}(S_{0},G^{p})^{*} \cong \mbox{Hom}(\mbox{Hom}(S_{0},G^{p}),\mbox{Ext}^{1}(S_{0},F^{q})),$$
  $\delta\in \mbox{Ext}^{1}(G^{p},F^{q})$ is sent to the connecting homomorphism $\delta_{*}: \mbox{Hom}(S_{0},G^{p}) \longrightarrow \mbox{Ext}^{1}(S_{0},F^{q})$. Since $\delta_{*}$ is general and by Remark \ref{full1}
  $$\overline{K^{\oplus pq}}=\mbox{Ext}^{1}(S_{0},F^{q})\otimes \mbox{Hom}(S_{0},G^{p}),$$
  we may apply Lemma \ref{general}, the proposition is proved.
\end{proof}

Now we are ready to state and prove the main result of this subsection.

\begin{theorem}[Local simplification]\label{local simplification}
  Let $E$ be a rigid object that fits into the following exact sequence
  $$0 \longrightarrow F^{q} \longrightarrow E \longrightarrow G^{p} \longrightarrow 0 ,$$
  where $F, G\in \mathcal{H}$ are $\sigma_{+}$-stable spherical objects with $\phi_{+}(F)>\phi_{+}(G)$ and $\mbox{Ext}^{1}(G,F)\neq 0$. Then the $\sigma_{-}$-Harder-Narasimhan filtration of $E$ is of the form
  $$0=E_{0} \subset E_{1} \subset E_{2} \subset E_{3} \subset E_{4}=E,$$
  where either: (a) $E_{1}=S_{0}\otimes \mbox{Hom}(S_{0},F^{q})$ and $E_{4}/E_{1}$ is rigid of type II, or (b) $E_{4}/E_{3}=T_{1}\otimes \mbox{Hom}(G^{p},T_{1})^{*}$ and $E_{3}$ is rigid of type II.
\end{theorem}

\begin{proof}
  By Proposition \ref{injectivity}, we assume $\mbox{Hom}(S_{0},G^{p}) \longrightarrow \mbox{Ext}^{1}(S_{0}, F^{q})$ is injective, the other case is similar. Hence we have an isomorphism $\mbox{Hom}(S_{0},F^{q})\cong \mbox{Hom}(S_{0},E)$. At $\sigma_{-}$, the first Harder-Narasimhan factor is $E_{1}=S_{0}\otimes \mbox{Hom}(S_{0},F^{q})$. By Corollary \ref{JH}, $F^{q}$ has $\sigma_{-}$-Harder-Narasimhan filtration
  $$0 \longrightarrow S_{0}\otimes \mbox{Hom}(S_{0}, F^{q}) \longrightarrow F^{q} \longrightarrow T_{1}\otimes \mbox{Hom}(F^{q},T_{1})^{*} \longrightarrow 0.$$
  Hence the quotient $E/E_{1}$, which is rigid, fits into
  $$0 \longrightarrow T_{1}\otimes \mbox{Hom}(F^{q},T_{1})^{*} \longrightarrow Q_{1} \longrightarrow G^{p} \longrightarrow 0.$$
  By Proposition \ref{type2}, $E/E_{1}$ is of type II, the proof is completed.
\end{proof}

\begin{example}[Example \ref{hard}, revisit]\label{easy}
  We revisit Example \ref{hard} by using the simplified local reduction (Theorem \ref{local simplification}). Keeping the notations in Example \ref{hard}, the rigid object $E$ is two step, in fact of type II. Then by Proposition \ref{type2}, the connecting homomorphism $\mbox{Hom}(T\otimes B_{1}, T) \longrightarrow \mbox{Ext}^{1}(Q_{1}, T)$ has maximal rank. By Theorem \ref{numbers}, $\mbox{Ext}^{1}(Q_{1}, T)=\mathbb{C}$, and $\mbox{Hom}(T\otimes B, T)=\mathbb{C}^{155}$. Hence $\mbox{hom}(E,T)=35+154=189$, and the $\sigma_{-}$-Harder-Narasimhan filtration has shape
  $$(E_{2}, 189 T),$$
  where $E_{2}$ is a rigid object of type I that fits into the exact sequence
  $$0 \longrightarrow T \longrightarrow E_{2} \longrightarrow S^{\oplus 6} \longrightarrow 0. $$
  By Proposition \ref{type1}, the $\sigma_{-}$-Harder-Narasimhan filtration of $E_{2}$ consists of a single factor $S_{-1}^{-}\in M_{\sigma_{-}}((29, 12H, 5))$. Hence the $\sigma_{-}$-Harder-Narasimhan filtration of $E$ has shape
  $$(S_{-1}^{-}, 189 T).$$
\end{example}

\subsection{Global Simplification}

We introduce a notion called the \emph{height} (Definition \ref{ht}) to measure the complexity of a stable spherical object with respect to a generic stability conditon. In this subsection we give a simple method to compute the cohomology of a height 2 object.

Recall that a stable spherical object is represented by a pair $(v, \sigma)$, where $v\in  \mbox{H}^{*}_{alg}(X)$ is spherical and $\sigma\in \mathbf{b}$, defined in Section \ref{section3}. We define the height as follows.

\begin{definition}[Height]\label{ht}
  Let $(v, \sigma)$ be a pair such that $\sigma$ is generic, and $v=(r,D,a)$ satisfies $r>0$, $D\cdot H>0$. Let $W$ be the wall of $v$ that is right below $\sigma$. Let $s,t$ be Mukai vectors of the two stable spherical objects on $W$ and let $\sigma_{-}$ be a stability condition right below $W$. Then we say $(v,\sigma)$ \emph{splits} into $(s, \sigma_{-})$ and $(t, \sigma_{-})$ and $(s, \sigma_{-})$, $(t, \sigma_{-})$ are \emph{factors} of $(v, \sigma)$.

  We say $(v, \sigma)$ is of \emph{height zero}, if $\sigma$ is right below the numerical wall $W(v, \mathcal{O}_{X}[1])$ or $v=v(\mathcal{O}_{X}[1])=(-1,0,-1)$. We define the \emph{height} of $(v, \sigma)$ to be the length of the longest chain of splitting so that every factor has height zero. 
\end{definition}

By definition, a $\sigma_{++}$-stable spherical object $E\in M_{\sigma_{++}}(v)$ has height 1, if the next wall below $\sigma_{++}$ is the Brill-Noether wall. In this case, the whole algorithm has only one step, the result is given by Corollary \ref{JH}.

Let $E\in M_{\sigma}(v)$ be a height 2 object. Let $W'$ be the wall below $\sigma$ and $\mathcal{H}'$ be the rank 2 lattice. Then the two factors $S', T'$ have height 1 or 0. By Corollary \ref{JH}, we have
$$0 \longrightarrow S'\otimes \mbox{Hom}(S',E) \longrightarrow E \longrightarrow T'\otimes \mbox{Hom}(E,T')^{*} \longrightarrow 0.$$

Let $\sigma_{+}$ be a stability condition below but near $W'$. First note that if $(S', \sigma_{+})$ has height 0, $S'$ cannot be $\mathcal{O}_{X}[1]$, because $v(E)=(r,D,a)$ satisfies $DH>0$, hence $\mbox{Hom}(\mathcal{O}_{X}[1], E)=0$. Hence $S'$ has $ \mbox{H}^{1}(S')=0$, we have
$$0 \longrightarrow  \mbox{H}^{0}(S')\otimes \mbox{Hom}(S',E) \longrightarrow  \mbox{H}^{0}(E) \longrightarrow  \mbox{H}^{0}(T')\otimes \mbox{Hom}(E,T')^{*} \longrightarrow 0. $$
Now $(T', \sigma_{+})$ has height at most 1, hence its $\mbox{h}^{0}$ is computed by Corollary \ref{JH}.

Hence we may assume $(S', \sigma_{+})$ has height 1. The next claim is that the connecting homomorphism $\delta:  \mbox{H}^{0}(T')\otimes \mbox{Hom}(E,T')^{*} \longrightarrow  \mbox{H}^{1}(S')\otimes \mbox{Hom}(S',E)$ always has maximal possible rank. From this we may compute $ \mbox{H}^{0}(E)$ inductively from $ \mbox{H}^{0}(S')$ and $ \mbox{H}^{0}(T')$.

Let $W_{S'}$, $W_{T'}$ be the Brill-Noether walls for $S', T'$, respectively. First we note that $W_{S'}>W_{T'}$ under the order on $\mathcal{W}(\mathbb{B})$ defined in Section \ref{section3}. Suppose not, then $W_{T'}>W_{S'}$. Let $\sigma$ be a stability condition that is below $W_{T'}$ but above $W_{S'}$. Then we have
$$\phi_{\sigma}(\mathcal{O}_{X}[1])< \phi_{\sigma}(T')< \phi_{\sigma}(S')< \phi_{\sigma}(\mathcal{O}_{X}[1]),$$
a contradiction. Hence $W_{S'}>W_{T'}$. Let $\sigma_{-}$ be a stability conditon below $W_{S'}$ but above $W_{T'}$. The $\sigma_{-}$-Harder-Narasimhan filtration of $S'$ is
$$0 \longrightarrow S_{-}\otimes \mbox{Hom}(S_{-},S') \longrightarrow S' \longrightarrow \mathcal{O}_{X}[1]\otimes  \mbox{H}^{1}(S') \longrightarrow 0. $$
At $\sigma_{-}$, we have
$$\phi_{-}(S_{-})>\phi_{-}(\mathcal{O}_{X}[1])>\phi_{-}(T'),$$
hence the first $\sigma_{-}$-Harder-Narasimhan factor of $E$ is $S_{-}\otimes \mbox{Hom}(S_{-},S')$, and the quotient $Q_{1}=E/ (S_{-}\otimes \mbox{Hom}(S_{-},S'))$ is a rigid object of type II (Proposition \ref{type2}), it fits into an exact sequence
$$0 \longrightarrow \mathcal{O}_{X}[1]\otimes  \mbox{H}^{1}(S') \otimes \mbox{Hom}(S',E) \longrightarrow Q_{1} \longrightarrow T'\otimes \mbox{Hom}(E,T')^{*} \longrightarrow 0 ,$$
where $\mathcal{O}_{X}[1], T'\in \mathcal{H}_{T'}$, the rank 2 lattice of $W_{T'}$. Since $ \mbox{H}^{1}(-)=\mbox{Hom}(-,\mathcal{O}_{X}[1])^{*}$, by Proposition \ref{type2} the connecting homomorphism
$$ \mbox{H}^{1}(T')\otimes \mbox{Hom}(E,T')^{*} \longrightarrow  \mbox{H}^{1}(\mathcal{O}_{X}[1])\otimes  \mbox{H}^{1}(S') \otimes \mbox{Hom}(S',E)$$
is a general linear map, hence has maximal possible rank. Note that $ \mbox{H}^{1}(\mathcal{O}_{X}[1])\otimes  \mbox{H}^{1}(S')$ is naturally identified with $ \mbox{H}^{1}(S')$. We have proved:

\begin{theorem}[Global simplification]\label{global simplification}
  Let $E\in M_{\sigma}(v)$ be a height 2 spherical object and $S, T$ be its factors. Then in the long exact sequence induced by $ \mbox{H}^{0}(-)$ on
  $$0 \longrightarrow S\otimes \mbox{Hom}(S,E) \longrightarrow E \longrightarrow T\otimes \mbox{Hom}(E,T)^{*} \longrightarrow 0, $$
  the connecting homomorphism $ \mbox{H}^{0}(T)\otimes \mbox{Hom}(E,T)^{*} \longrightarrow  \mbox{H}^{1}(S)\otimes \mbox{Hom}(S,E)$ has maximal rank.
\end{theorem}

We show the use of Theorem \ref{global simplification} explicitly in Example \ref{height two}. 

\begin{remark}
  The height zero and height one cases are clear. If $E$ is a height zero stable spherical vector bundle with $c_{1}(E)\cdot H>0$, we have $\mbox{h}^{1}(E)=0$ by Corollary \ref{output}. If $E$ has height one, the cohomology of $E$ is computed by Theorem \ref{numbers} (see Example \ref{height one}).
\end{remark}

\section{Weak Brill-Noether in Picard Rank One}

In this section we show a numerical condition that is equivalent to weak Brill-Noether for a spherical vector bundle on a K3 surface with Picard rank one.

For the rest of this section, we assume the polarized K3 surface $(X,H)$ has $\mbox{Pic}(X)=\mathbb{Z}H$ and $H^{2}=2n$. Let $\mathbb{H}=\{(sH,tH): s\in \mathbb{R}, t\in \mathbb{R}_{>0}\}$, and let $\mathcal{A}_{0}=\mathcal{A}_{0,tH}, \forall t>0$. Recall that we defined the line $\mathbf{b}\subset \mathbb{H}$ in Section \ref{section3}. When the Picard rank is one, $\mathbf{b}=\{(\epsilon H, tH)\}$ for a sufficiently small $\epsilon>0$. Let $\mathcal{A}_{\epsilon}$ be the heart of any stability condition in $\mathbf{b}$. The main theorem of this section is the following.

\begin{theorem}\label{weak BN}
  (Weak Brill-Noether) Let $(X,H)$ be a polarized K3 surface of Picard rank one. Let $v=(r,dH,a)$ be a Mukai vector with $v^{2}=-2$ and $r, d>0$, $E\in M_{H}(v)$.
Let $y$ be the largest possible value of $\frac{a_{1}d-ad_{1}}{r_{1}d-rd_{1}} $ where $v_{1}=(r_{1}, d_{1}H, a_{1})\neq v$ satisfies 
$$v_{1}^{2}=-2, vv_{1}<0, \frac{a_{1}d-ad_{1}}{r_{1}d-rd_{1}}>0, 0<d_{1}\leq d.$$
Then $\mbox{H}^{1}(E)=0$ if and only if $y<1$.
\end{theorem}

We show the use of Theorem \ref{weak BN} in Example \ref{wBNfail} and Example \ref{wBNhold}.

The crucial observation is the following lemma.

\begin{lemma}\label{negative rank}
Let $\mathcal{H}$ be the rank 2 lattice of a wall, whose two stable spherical objects have Mukai vectors $v_{1}, v_{2}$. Assume $v_{0}=(-1,0,-1)\notin \mathcal{H}$. Then $r(v_{1})\cdot r(v_{2})<0$. 
\end{lemma}

\begin{proof}
  Write $v_{i}=(r_{i}, d_{i}H, a_{i}), i=1,2$. First note that $r_{i}\neq 0$. If $r_{i}=0$, then
  $$-2=v_{i}^{2}=d_{i}^{2}H^{2}-2r_{i}a_{i}=2nd_{i}^{2}\geq 0,$$
  a contradiction. Hence it suffices to show $r_{1}, r_{2}$ cannot both be positive or negative. We prove by contradiction that they cannot both be positive, the other case is similar. Let $\chi_{i}=v_{0}v_{i}=r_{i}+a_{i}$. Note that $r_{i}>0$ is equivalent to $a_{i}>0$, which is also equivalent to $\chi_{i}>0$, because $2r_{i}a_{i}=2nd_{i}^{2}+2>0$.

  Let $\chi=v_{1}v_{2}$. Since $v_{1}, v_{2}$ are the Mukai vectors of the two stable spherical objects in $\mathcal{H}$, we have $\chi>0$. Note that $\mathcal{H}$ cannot be negative definite, since the lattice $ \mbox{H}^{*}_{alg}(X)\cong U\oplus \mathbb{Z}H$ where $U$ is a hyperbolic plane. The signature of $U\oplus \mathbb{Z}H$ is $(2,1)$, it cannot contain a negative definite plane. Hence by Proposition \ref{lattice}, $\chi\geq 2$.

  Next we consider the sublattice $N=\langle v_{0}, v_{1}, v_{2} \rangle \subset  \mbox{H}^{*}_{alg}(X) \cong U\oplus \mathbb{Z}H$. By assuption, $v_{0}\notin \mathcal{H}$. And since $\mathcal{H}$ is the rank 2 lattice of a wall, it is primitive, hence $v_{0}\notin \mathcal{H}_{\mathbb{Q}}$, $N$ is a sublattice of full rank. In particular, its discriminant has the same sign as $\mbox{disc}(U\oplus \mathbb{Z}H)=-2n$, which is negative. Now we compute the discriminant of $N$:
  \begin{eqnarray*}
  \mbox{disc}(N) &= & \mbox{det} \begin{pmatrix}
                                   v_{0}^{2} & v_{0}v_{1} & v_{0}v_{2} \\
                                   v_{1}v_{0} & v_{1}^{2} & v_{1}v_{2} \\
                                   v_{2}v_{0} & v_{2}v_{1} & v_{2}^{2}
               \end{pmatrix}\\                       
    &=&\mbox{det}
               \begin{pmatrix}
                 -2 & \chi_{1} & \chi_{2} \\
                 \chi_{1} & -2 & \chi \\
                 \chi_{2} & \chi & -2
               \end{pmatrix}\\
    &=& 2(\chi^{2}-4)+\chi_{1}(2\chi_{1}+\chi\chi_{2})+\chi_{2}(\chi_{1}\chi+2\chi_{2}).
  \end{eqnarray*}
Since $\chi\geq 2$, $\chi^{2}-4\geq 0$. If $\chi_{1}, \chi_{2}>0$, then $\mbox{disc}(N)>0$, a contradiction.  
\end{proof}

\begin{proof}[Proof of Theorem \ref{weak BN}]
  Let $\sigma=\sigma_{(\epsilon H, tH)}\in \mathbf{b}$ with $t\gg 0$. By Theorem \ref{large volume limit}, $E\in M_{\sigma}(v)$. Using the notation in Proposition \ref{preoutput}, let $sh(\sigma)=(E^{\bullet}, \partial W)$. Since $E\in M_{\sigma}(v)$, the filtration $E^{\bullet}=E$ is trivial. By Proposition \ref{criterion}, $y=\frac{f_{0}(W)}{\sqrt{2/H^{2}}} $. We separate into three cases.
\\
  \emph{Case I}: $y<1$. Then $\sigma$ satisfies the condition in Proposition \ref{preoutput}. Since $E^{\bullet}=E$ is trivial, under the notations in Proposition \ref{preoutput}, $h=0$. Hence by Corollary \ref{output}, $\mbox{h}^{1}(E)=0$.
\\
  \emph{Case II}: $y=1$. Let $W'$ be the numerical wall defined by $v$ and $v(\mathcal{O}_{X}[1])=(-1,0,-1)$. Then $\frac{f_{0}(W')}{\sqrt{2/H^{2}}}=1 $. By \cite{Mac14}, numerical walls of $v$ are nested, hence $W=W'$. Let $S, T$ be the two stable spherical objects in $\mathcal{H}$, the rank 2 lattice of $W$. Since $\mathcal{O}_{X}[1]$ is $\sigma$-stable for all $\sigma\in \mathbf{b}$, we have $T=\mathcal{O}_{X}[1]$. Let $\sigma_{-}\in \mathbf{b}$ be a stability condition right below $W$. Since $\mbox{Hom}(\mathcal{O}_{X}[1],E)=0$, by Corollary \ref{JH}, the $\sigma_{-}$-Harder-Narasimhan filtration of $E$ is
  $$0 \longrightarrow S^{p} \longrightarrow E \longrightarrow \mathcal{O}_{X}[1]^{q} \longrightarrow 0 $$
  for some $p,q>0$. By Corollary \ref{output}, $\mbox{h}^{1}(E)=q>0$.
\\
  \emph{Case III}: $y>1$. Same notations as in Case II. By Corollary \ref{JH}, the $\sigma_{-}$-Harder-Narasimhan filtration of $E$ is
\begin{equation}\label{(1)}
  0 \longrightarrow S^{p} \longrightarrow E \longrightarrow T^{q} \longrightarrow 0
\end{equation}
for some $p,q>0$. Note that $S\in \mathcal{A}_{\epsilon}$ is a sheaf (complex concentrated in degree 0): by taking the long exact sequence of cohomology associated to (\ref{(1)}), we have
$$0 \longrightarrow \mathcal{H}^{-1}(S^{p}) \longrightarrow \mathcal{H}^{-1}(E) \longrightarrow \cdots.$$
Since $E$ is a sheaf, $\mathcal{H}^{-1}(E)=0$, hence $\mathcal{H}^{-1}(S)=0$.
Therefore $r(S)>0$. Since $y\neq 1$, $v(\mathcal{O}_{X}[1])=(-1,0,-1)\notin \mathcal{H}$. By Lemma \ref{negative rank}, $r(T)<0$, equivalently $\chi(T)<0$. Hence $\mbox{H}^{1}(T)\neq 0$. Consider the long exact sequence of cohomology on (\ref{(1)}):
$$\cdots \longrightarrow  \mbox{H}^{1}(E) \longrightarrow  \mbox{H}^{1}(T^{q}) \longrightarrow  \mbox{H}^{2}(S^{p})\longrightarrow \cdots .$$
Since $y>1$, $\sigma_{-}$ is above the Brill-Noether wall of $S$, hence $ \mbox{H}^{2}(S)=0$. Since $ \mbox{H}^{1}(T)\neq 0$, we have $ \mbox{H}^{1}(E)\neq 0$. 
\end{proof}

\section{Asymptotic Results}

In this section we prove an asymptotic estimation (Theorem \ref{asymptotic}) of the cohomology of a spherical vector bundle, in the sense that the error become relatively small when the Mukai vector is obtained by many spherical reflections from the two stable sperical Mukai vectors. 

Let $\mathcal{H}$ be the rank 2 lattice of a wall. Recall from Proposition \ref{lattice} that every effective spherical Mukai vector in $\mathcal{H}$ can be written as $t_{i}$ or $s_{j}$. In this section we give an asymptotic result for the cohomology of $E\in M_{H}(t_{i})$ or $M_{H}(s_{j})$ for $i\gg 0$ or $j\ll 0$.

Using the notation in Proposition \ref{construction}, let
\begin{align*}
  ev_{i-1}:  \mbox{H}^{0}(T_{i-1})\otimes \mbox{Hom}(T_{i-1},T_{i-2}) \longrightarrow  \mbox{H}^{0}(T_{i-2}),\\
  coev_{j}:  \mbox{H}^{0}(S_{j+1}) \longrightarrow  \mbox{H}^{0}(S_{j})\otimes \mbox{Hom}(S_{j+1},S_{j})^{*}
\end{align*}
be the induced maps on the evaluation map or coevaluation map. If we know the ranks of $ev_{i}, coev_{j}$ for all $i\geq 1$ and $j\leq 0$, then the cohomology of $T_{i}, S_{j}$ can be computed inductively from the cohomology of $T_{1}, S_{0}$. The main theorem of this section shows that the ranks of most of the maps are known.

\begin{theorem}[Asymptotic Result]\label{asymptotic}
  Let $X$ be any K3 surface. Then $ev_{i}$ is surjective for $i\geq 4$ and $coev_{j}$ is injective for $j\leq -3$. If $\mbox{Pic}(X)=\mathbb{Z}H$ and $H^{2}\neq 2$, then $ev_{i}$ is surjective for $i\geq 2$ and $coev_{j}$ is injective for $j\leq -1$.
\end{theorem}

\subsection{A Bound for $\mbox{h}^{0}$}
In this subsection we find a bound for $\mbox{h}^{0}$ of a spherical vector bundle on a K3 surface with Picard rank one and degree at least 4. The bound itself is interesting, it is also needed for the proof of Theorem \ref{asymptotic}.

\begin{proposition}[Bound of $\mbox{h}^{0}$]\label{h0}
  Let $X$ be a K3 surface with $\mbox{Pic}(X)=\mathbb{Z}H$ and $H^{2}\neq 2$.
  Then for any $v=(r,dH,a)\in  \mbox{H}^{*}_{alg}(X)$ with $v^{2}=-2$ and $\sigma\in \mbox{Stab}(X)$ above the Brill-Noether wall, if $r>0$ we have
  $$\mbox{h}^{0}(S)<2\chi(S)$$
  for $S\in M_{\sigma}(v)$. Symmetrically, if $r<0$, then
  $$\mbox{h}^{1}(T)<-2\chi(T)$$
  for $T\in M_{\sigma}(v)$.
\end{proposition}

Given the proposition, we have an immediate corollary

\begin{corollary}\label{coarse}
  Let $S, T$ be two spherical objects as in the proposition. Then
  $$-\chi(S,T)>\mbox{h}^{1}(S)+2,$$
  $$ -\chi(S,T)> \mbox{h}^{0}(T)+2.$$
\end{corollary}

\begin{proof}
  Write $v(S)=(r_{0}, d_{0}H, a_{0})$ and $v(T)=(r_{1}, d_{1}H, a_{1})$. By assumption $d_{0}, d_{1}$ are positive integers. By Proposition \ref{h0}, $\mbox{h}^{1}(S)=\mbox{h}^{0}(S)-\chi(S)<2\chi(S)-\chi(S)=\chi(S)$. On the other hand
  $$-\chi(S,T)=d_{0}d_{1}H^{2}-r_{0}a_{1}-r_{1}a_{0}\geq H^{2}+r_{0}+a_{0}\geq 2+r_{0}+a_{0}=2+\chi(S) > 2+\mbox{h}^{1}(S).$$
  The proof for $T$ is similar.
\end{proof}

We only prove the proposition for $r>0$, the other case is similar. 

\begin{lemma}\label{ratio}
Let $X$ be a K3 surface with $\mbox{Pic}(X)=\mathbb{Z}H$. Let $v=(r,dH,a)\in  \mbox{H}^{*}_{alg}(X) $ with $v^{2}=-2$, $r>0$, $W$ an actual wall for $v$ that is above the Brill-Noether wall, and $S_{0}, T_{1}$ the two stable spherical objects at the wall. Let $\sigma_{+}, \sigma_{-} \in  \mbox{Stab}(X)$ be above/below but near $W$ respectively. Then
  $$\frac{\mbox{h}^{0}(S_{-m})}{\mbox{h}^{1}(S_{-m})+2}\geq \frac{\mbox{h}^{0}(S_{0})}{\mbox{h}^{1}(S_{0})+2},$$
  where $S_{-m}\in M_{+}(v)$ is constructed as in Proposition \ref{construction}.
\end{lemma}

For any positive integer $n$, denote by $P_{n}$ Proposition \ref{h0} for $d\leq n$, denote by $L_{n}$ Lemma \ref{ratio} for $d\leq n$. We will prove Proposition \ref{h0} and Lemma \ref{ratio} simultaneously by showing that $L_{d}$ implies $P_{d}$ and $P_{d-1}$ implies $L_{d}$. 

\begin{proof}[Proof of $L_{d}$ implies $P_{d}$]
 Let $S$ be as in Proposition \ref{h0}. Since we know $\mbox{h}^{0}(S)>\mbox{h}^{1}(S)$, $P_{d}$ is equivalent to $\frac{\mbox{h}^{0}(S)}{\mbox{h}^{1}(S)}>2$. Assuming $L_{d}$ is true. we prove a stronger claim that $\frac{\mbox{h}^{0}(S)}{\mbox{h}^{1}(S)+2}\geq 2$ by induction on the height (Definition \ref{ht}) of $S$.

 If $\mbox{height}(S)=0$, then $S$ satisfies weak Brill-Noether. Write $v(S)=(r,dH, a)$, we have
 $$\frac{\mbox{h}^{0}(S)}{\mbox{h}^{1}(S)+2}=\frac{\chi(S)}{2}=\frac{r+a}{2}.$$
 It suffices to prove $r+a\geq 4$. Since $H^{2}\geq 4$, we have
 $$r+a \geq 2\sqrt{ra}=2\sqrt{d^{2}H^{2}/2+1}\geq 2\sqrt{3}>3.$$
 hence $r+a\geq 4$.

 If $\mbox{height}(S)\geq 1$, let $W$ be the first wall of $S$ that is below $\sigma$. Let $S_{0}, T_{1}$ be the two stable spherical objects at $W$, and label $S=S_{-m}$ as in Proposition \ref{construction}. By definition, $\mbox{height}(S_{0})<\mbox{height}(S)$.
 Since $\mbox{height}(S)\geq 1$, we have $c_{1}(S_{0})<c_{1}(S_{-m})$. Then by $L_{d}$, we have
 $$\frac{\mbox{h}^{0}(S)}{\mbox{h}^{1}(S)+2}\geq \frac{\mbox{h}^{0}(S_{0})}{\mbox{h}^{1}(S_{0})+2}\geq 2.$$
\end{proof}

\begin{proof}[Proof of $P_{d-1}$ implies $L_{d}$]

  First note that
  $$\mbox{h}^{0}(S_{-m})-\mbox{h}^{1}(S_{-m})=\chi(S_{-m})=r+a \geq 2\sqrt{ra}=2\sqrt{d^{2}H^{2}/2+1}\geq 2,$$
  by the same argument in the proof of $P_{d}$ from $L_{d}$, equality cannot hold, hence $\mbox{h}^{0}(S_{-m})\geq \mbox{h}^{1}(S_{-m})+2$. The function
  $$\frac{\mbox{h}^{0}(S_{-m})+x}{\mbox{h}^{1}(S_{-m})+2+x}$$
  is non-increasing for positive $x$.
Now $S_{-m}$ fits into
$$0 \longrightarrow S_{0}\otimes \mbox{Hom}(S_{0},S_{-m}) \longrightarrow S_{-m} \longrightarrow T_{1}\otimes \mbox{Hom}(S_{-m},T_{1})^{*} \longrightarrow 0 .$$
By Theorem \ref{numbers}, we know $\mbox{hom}(S_{0}, S_{-m})=a_{m}$, $\mbox{hom}(S_{-m}, T_{1})=a_{m-2}$, where $a_{j}$ is the fundamental sequence of $W$ (Definition \ref{fundamental sequence}) given by
$$a_{-1}=0, a_{0}=1, a_{1}=-\chi(S_{0}, T_{1})=:g, a_{j}=ga_{j-1}-a_{j-2}.$$
Taking the long exact sequence, let $t$ be the rank of the connecting homomorphism
$$ \mbox{H}^{0}(T_{1})\otimes \mbox{Hom}(S_{-m},T_{1})^{*} \longrightarrow  \mbox{H}^{1}(S_{0})\otimes \mbox{Hom}(S_{0},S_{-m}). $$ Then we have
  $$\frac{\mbox{h}^{0}(S_{-m})}{\mbox{h}^{1}(S_{-m})+2}=
  \frac{a_{m}\mbox{h}^{0}(S_{0})+a_{m-2}\mbox{h}^{0}(T_{1})-t}{a_{m}\mbox{h}^{1}(S_{0})+a_{m-2}\mbox{h}^{1}(T_{1})+2-t}\geq
  \frac{a_{m}\mbox{h}^{0}(S_{0})+a_{m-2}\mbox{h}^{0}(T_{1})}{a_{m}\mbox{h}^{1}(S_{0})+a_{m-2}\mbox{h}^{1}(T_{1})+2}.$$
  Since $T_{1}$ has smaller degree than $S_{-m}$, $P_{d-1}$ is true for $T_{1}$. Write $v(S_{0})=(r_{0}, d_{0}H, a_{0})$ and $v(T_{1})=(r_{1}, d_{1}H, a_{1})$, we have $\mbox{h}^{1}(T_{1})<-2(r_{1}+a_{1})$ by $P_{d-1}$. On the other hand
  $$g=-\chi(S_{0}, T_{1})=d_{0}d_{1}H^{2}-r_{0}a_{1}-r_{1}a_{0}> -(r_{1}+a_{1})>\frac{\mbox{h}^{1}(T_{1})}{2}$$
  hence
  $$\frac{\mbox{h}^{0}(S_{-m})}{\mbox{h}^{1}(S_{-m})+2}\geq \frac{a_{m}\mbox{h}^{0}(S_{0})+a_{m-2}\mbox{h}^{0}(T_{1})}{a_{m}\mbox{h}^{1}(S_{0})+a_{m-2}\mbox{h}^{1}(T_{1})+2} >\frac{a_{m}\mbox{h}^{0}(S_{0})}{a_{m}\mbox{h}^{1}(S_{0})+2ga_{m-2}+2},$$
  it suffices to show $2ga_{m-2}+2\leq 2a_{m}$. Since $a_{i}$ are integers, it suffices to show
  $$a_{m}-ga_{m-2}=ga_{m-1}-a_{m-2}-ga_{m-2}= ga_{m-1}-(g+1)a_{m-2}>0,$$
namely $\frac{a_{m-1}}{a_{m-2}}>\frac{g+1}{g}$.
If $m=1$, $2ga_{-1}+2=2< 2g=2a_{1}$. If $m\geq 2$, note that $g\geq 3$, by Lemma \ref{decreasing}, we have
$$\frac{a_{m-1}}{a_{m-2}}> \rho(W)>2>\frac{g+1}{g},$$
where $\rho(W)$ is the bigger solution of the equation $x+\frac{1}{x}=g$ (Remark \ref{lim}).
\end{proof}

\subsection{Asymptotic Result}

In this subsection we prove Theorem \ref{asymptotic}. First, we need the following simple observation.

\begin{lemma} \label{nondeg}
Let $X=\mathbb{P}^{m}\times \mathbb{P}^{n}\subset \mathbb{P}^{N}$ be the Segre embedding, $H \subset \mathbb{P}^{N}$ a hyperplane. Then $H\cap X$ is non-degenerate in $H$ set theoretically.
\end{lemma}

\begin{proof}
  We first prove that for any non-degenerate subvariety $X \subset \mathbb{P}^{N}$, $\bar{X}=H\cap X\subset H $ is non-degenerate in $H$ scheme theoretically.
  The non-degeneracy of $X$ in $\mathbb{P}^{N}$ is equivalent to injectivity of the map $ \mbox{H}^{0}(\mathcal{O}_{\mathbb{P}^{N}}(H)) \longrightarrow  \mbox{H}^{0}(\mathcal{O}_{X}(H))$. Denote $\bar{X}=H \cap X$. Then we have
  $$0 \longrightarrow  \mbox{H}^{0}(\mathcal{O}_{X}) \longrightarrow  \mbox{H}^{0}(\mathcal{O}_{X}(H)) \longrightarrow  \mbox{H}^{0}(\mathcal{O}_{\bar{X}}(H)) .$$
  Since $X$ is closed projective, $ \mbox{H}^{0}(\mathcal{O}_{X})=\mathbb{C}$. Together with the injectivity of $ \mbox{H}^{0}(\mathcal{O}_{\mathbb{P}^{N}}(H)) \longrightarrow  \mbox{H}^{0}(\mathcal{O}_{X}(H))$ shows that the composition map
  $$ \mbox{H}^{0}(\mathcal{O}_{\mathbb{P}^{N}}(H))\longrightarrow  \mbox{H}^{0}(\mathcal{O}_{\bar{X}}(H))$$
  has one dimensional kernel, namely $H$. Hence $\bar{X}$ is non-degenerate in $H$.

  Now let $X=\mathbb{P}^{m}\times \mathbb{P}^{n}$ be the Segre embedding. First observe that $\bar{X}$ is generically reduced: suppose not, note that $\bar{X}\subset \mathbb{P}^{m}\times \mathbb{P}^{n}$ has class $(1,1)$ in $\mbox{Pic}(\mathbb{P}^{m}\times \mathbb{P}^{n})$, if $\bar{X}$ is not generically reduced, then $\bar{X}_{red}$ has class $(0,1)$ or $(1,0)$, say $(0,1)$. Then $\bar{X}$ has class at least $(0,2)$, which is impossible. Now if $H$ is not tangent to $X$, then $\bar{X}$ is reduced. If $H$ is tangent to $X$, by generic reducedness of $\bar{X}_{red}$, $\bar{X}$ has to be $\mathbb{P}^{m}\times \mathbb{P}^{n-1}\cup \mathbb{P}^{m-1}\times \mathbb{P}^{n}$, which is also reduced. Hence $\bar{X}_{red}=\bar{X}$, the lemma is proved.  
\end{proof}

As corollaries, we get the following linear algebra lemmas

\begin{lemma}\label{pure tensor}
Let $V_{1}, V_{2}$ be linear spaces, and $m: V_{1}\otimes V_{2} \longrightarrow \mathbb{C}$ be a bilnear form. If $\mbox{dim}V_{1}+\mbox{dim}V_{2}\geq 3$, then every element in $\mbox{ker}(m)$ can be written as a sum of pure tensors in $\mbox{ker}(m)$. 
\end{lemma}

\begin{proof}
The $H=\mathbb{P}\mbox{ker}(m)\subset \mathbb{P}(V_{1}\otimes V_{2})$ is a hyperplane in $\mathbb{P}(V_{1}\otimes V_{2})$. The space of pure tensors is naturally identified with the Segre embedding $X=\mathbb{P}V_{1} \times \mathbb{P}V_{2}\subset \mathbb{P}(V_{1}\otimes V_{2})$, whose dimension is $\mbox{dim}V_{1}+\mbox{dim}V_{2}-2$. If $\mbox{dim}V_{1}+\mbox{dim}V_{2}-2 \geq 1$, then $H\cap X$ is non-empty. By Lemma \ref{nondeg}, $H\cap X$ is non-degnerate in $H$ set theoretically, namely every element in $\mbox{ker}(m)$ is a sum of pure tensors in $\mbox{ker}(m)$.
\end{proof}

\begin{lemma}\label{surjectivity2}
  Let $V_{1}, V_{2}, \cdots , V_{n}, \cdots $ be a sequence of linear spaces over a field of characteristic 0, such that for $n\leq m$ there are arbitrary bilinear forms $m_{n-1,n}: V_{n-1}\otimes V_{n} \longrightarrow W_{n-1,n}$ for some non-zero linear spaces $W_{n-1,n}$, and for $n\geq m+1$ there are perfect pairings $m_{n-1,n}: V_{n-1}\otimes V_{n} \longrightarrow \mathbb{C}$.
   Let
  $$K_{n}:= \bigcap_{j=2}^{n}\mbox{ker}(m_{j-1, j})\subset \bigotimes_{j=1}^{n}V_{j}$$
  where by abuse of notation $m_{j-1,j}$ means $\mbox{id}\otimes \cdots \otimes \mbox{id} \otimes m_{j-1, j} \otimes \mbox{id} \otimes \cdots \otimes \mbox{id}$. If $\mbox{dim}V_{n}\geq 2$ for $n\geq m$, then for $n\geq m+1$ there is a commutative diagram
  \[
\begin{tikzcd}
  K_{n+1}\otimes V_{n+2} \arrow[r, hookrightarrow] \arrow[d, twoheadrightarrow, "m_{n+1, n+2}"] & \bigotimes_{j=1}^{n+2}V_{j}\arrow[d, twoheadrightarrow, "m_{n+1, n+2}"]\\
  K_{n} \arrow[r, hookrightarrow] & (\bigotimes_{j=1}^{n}V_{j})
\end{tikzcd}
\]
where the vertical maps are surjective.
\end{lemma}

\begin{proof}
  Commutativity is clear. For $n\geq m$, the image of the composition map
  $$K_{n+1}\otimes V_{n+2} \hookrightarrow \bigotimes_{j=1}^{n+2}V_{j} \twoheadrightarrow \bigotimes_{j=1}^{n}V_{j}$$
  is contained in $K_{n}$. Hence it suffices to check that for $n\geq m$, the left vertical map is surjective. For any $1\leq j\leq n$, let $\{v^{j}_{i_{j}}\}_{1\leq i_{j}\leq \mbox{dim}V_{j}}$ be a basis of $V_{j}$. Let
  $$\eta = \sum_{i_{1}, \cdots , i_{n}} a_{i_{1}, \cdots , i_{n}}v^{1}_{i_{1}}\otimes \cdots \otimes v^{n}_{i_{n}} \in K_{n}\subset \bigotimes_{j=1}^{n}V_{j}$$
  be any element in $K_{n}$. We need to find an element in $\tilde{\eta}\in K_{n+1}\otimes V_{n+2}$ such that $m_{n+1, n+2}(\tilde{\eta})=\eta$.
  \\
  The condition that $\eta\in K_{n}$ spells out
  $$m_{j,j+1}(\eta)=\sum_{\substack{i_{k} \\k\neq j, j+1 \\ 1\leq k \leq n}}\sum_{i_{j}, i_{j+1}}a_{i_{1}, \cdots i_{n}}m_{j, j+1}(v^{j}_{i_{j}}\otimes v^{j+1}_{i_{j+1}}) v^{1}_{i_{1}}\otimes \cdots v^{j-1}_{i_{j-1}}\otimes v^{j+2}_{i_{j+2}}\otimes \cdots v^{n}_{i_{n}}=0.$$
  Since $v^{1}_{i_{1}}\otimes \cdots v^{j-1}_{i_{j-1}}\otimes v^{j+2}_{i_{j+2}}\otimes \cdots v^{n}_{i_{n}}$ form a basis for $V_{1}\otimes \cdots \otimes V_{j-1}\otimes V_{j+2}\otimes \cdots \otimes V_{n}$, we know that
  $$\sum_{i_{j}, i_{j+1}}a_{i_{1}, \cdots i_{n}}m_{j, j+1}(v^{j}_{i_{j}}\otimes v^{j+1}_{i_{j+1}})=0$$
  for any $i_{1}, i_{2}, \cdots , i_{j-1}, i_{j+2}, \cdots , i_{n}$. We take $j=n-1$. Since $n\geq m+1$, by assumption $m_{n-1,n}:V_{n-1}\otimes V_{n} \longrightarrow \mathbb{C}$ is a perfect pairing. By Lemma \ref{pure tensor}, there is a basis of $\mbox{ker}(m_{n-1,n})$ that consists of pure tensors, say $\{u^{n-1}_{i_{n-1}}\otimes u^{n}_{i_{n}}\}$. Since
$$\sum_{i_{j}, i_{j+1}}a_{i_{1}, \cdots i_{n}}v^{j}_{i_{j}}\otimes v^{j+1}_{i_{j+1}}\in \mbox{ker}(m_{n-1,n}),$$
it can be written as a linear combination of this new basis. By extending $\{u^{n-1}_{i_{n-1}}\otimes u^{n}_{i_{n}}\}$ to a basis $\{u^{n-1}_{i_{n-1}}\otimes u^{n}_{i_{n}}\}_{1\leq i_{j}\leq \mbox{dim}V_{j}}$, we may rewrite
  $$\eta = \sum_{i_{1}, \cdots , i_{n}} b_{i_{1}, \cdots , i_{n}}v^{1}_{i_{1}}\otimes \cdots \otimes v^{n-2}_{i_{n-2}}\otimes u^{n-1}_{i_{n-1}}\otimes u^{n}_{i_{n}},$$
  where $u^{n-1}_{i_{n-1}}\otimes u^{n}_{i_{n}}\in \mbox{ker}(m_{n-1,n})$. By the same argument as above, we have
  \begin{equation}\label{2}
    \sum_{i_{j}, i_{j+1}}b_{i_{1}, \cdots i_{n}}m_{j, j+1}(v^{j}_{i_{j}}\otimes v^{j+1}_{i_{j+1}})=0
  \end{equation}
  for all $1\leq j\leq n-2$.
  \\
  Since $\mbox{dim}V_{n}\geq 2$, let $\tau$ be any permutation on $\{1, 2, \cdots , \mbox{dim}V_{n} \}$ without fixed point.
  Let
  $$\widetilde{\eta}=\frac{1}{\mbox{dim}V_{n}} \sum_{i_{1}, \cdots , i_{n}} b_{i_{1}, \cdots , i_{n}} v^{1}_{i_{1}}\otimes \cdots \otimes v^{n-2}_{i_{n-2}}\otimes u^{n-1}_{i_{n-1}}\otimes u^{n}_{i_{n}}\otimes (u^{n}_{\tau(i_{n})})^{*},$$
  where $(u^{n}_{\tau(i_{n})})^{*}\in V_{n+1}$ are the dual elements of $u^{n}_{\tau(i_{n})}$ the the identification $V_{n+1}=V_{n}^{*}$ via the perfect pairing $m_{n,n+1}$.
  By (\ref{2}), $m_{j, j+1}(\widetilde{\eta})=0$ for $1\leq j \leq n-2$. since $u^{n-1}_{i_{n-1}}\otimes u^{n}_{i_{n}}\in \mbox{ker}(m_{n-1,n})$, we also have $m_{n-1,n}(\widetilde{\eta})=0$. Since $\tau$ has no fixed point, we have
  $$m_{n,n+1}(u^{n}_{i_{n}}\otimes (u^{n}_{\tau(i_{n})})^{*}=0$$
    for any $i_{n}$, we have $\widetilde{\eta}\in \mbox{ker}(m_{n,n+1})$. Hence we showed $\widetilde{\eta}\in K_{n+1}$. Finally, we have
    $$m_{n+1, n+2}(\widetilde{\eta} \otimes (\sum_{i_{n}}(u^{n}_{i_{n}})^{**}))= \eta, $$
   where $ (u^{n}_{i_{n}})^{**} \in V_{n+2}$ are dual elements of $(u^{n}_{i_{n}})^{*}\in V_{n+1}$ under the identification $V_{n+2}\cong V_{n+1}^{*}$ via the perfect pairing $m_{n+1, n+2}$. 

\end{proof}

\begin{proof}[Proof of Theorem \ref{asymptotic}]
  We prove the claims for $ev_{i}$, $coev_{j}$ are similar.
  
  Let $V_{1}= \mbox{H}^{0}(T_{1})$, $V_{n}=\mbox{Hom}(T_{n},T_{n-1})^{*}$ for $n\geq 2$. We define $m_{1,2}$ as the connecting homomorphism
  $$m_{1,2}:  \mbox{H}^{0}(T_{1})\otimes \mbox{Hom}(T_{2}, T_{1})^{*}\cong  \mbox{H}^{0}(T_{1})\otimes \mbox{Ext}^{1}(T_{1}, S_{0}) \longrightarrow  \mbox{H}^{1}(S_{0}).$$
  For $n\geq 3$, let
  $$m_{n-1,n}: V_{n-1}\otimes V_{n} \longrightarrow \mathbb{C}$$
  be the perfect pairing induced by the dual of $\delta_{n,n-1}: \mathbb{C} \rightarrow \mbox{Hom}(T_{n},T_{n-1})\otimes \mbox{Hom}(T_{n-1}, T_{n-2})$ which was introduced in Lemma \ref{identification}.
  Let $W_{1,2}= \mbox{H}^{1}(S_{0})$. By Proposition \ref{description2}, we have the natural identification
  $$\mbox{Hom}(T_{n},T_{1})^{*}=\bigcap_{j=3}^{n}\mbox{ker}(m_{n-1,n}).$$
  Since we have the exact sequence
  $$ \mbox{H}^{0}(T_{n})\longrightarrow  \mbox{H}^{0}(T_{1})\otimes \mbox{Hom}(T_{n},T_{1})^{*} \longrightarrow  \mbox{H}^{1}(S_{0})\otimes \mbox{Hom}(S_{0}, T_{n}),$$
the image of $ \mbox{H}^{0}(T_{n}) \longrightarrow  \mbox{H}^{0}(T_{1})\otimes \mbox{Hom}(T_{n},T_{1})^{*}$ is hence identified with $K_{n}$ using the notation in Lemma \ref{surjectivity2}. By setting $m=2$ in Lemma \ref{surjectivity2}, we have
  $$K_{n+1}\otimes V_{n+2} \longrightarrow K_{n}$$
  is surjective for $n\geq 3$. Hence in the following commutative diagram
  \\
  \begin{adjustbox}{scale=0.8, center}
  \begin{tikzcd}
    0 \arrow[r] &  \mbox{H}^{0}(S_{0})\otimes \mbox{Hom}(S_{0},T_{n+1})\otimes \mbox{Hom}(T_{n+1}, T_{n}) \arrow[r]\arrow[d] &  \mbox{H}^{0}(T_{n+1})\otimes \mbox{Hom}(T_{n+1}, T_{n}) \arrow[r]\arrow[d] &  K_{n+1}\otimes V_{n+2} \arrow[r]\arrow[d] & 0\\
    0 \arrow[r] &  \mbox{H}^{0}(S_{0})\otimes \mbox{Hom}(S_{0}, T_{n}) \arrow[r]&  \mbox{H}^{0}(T_{n}) \arrow[r] & K_{n} \arrow[r] & 0
  \end{tikzcd}
  \end{adjustbox}
the third vertical map is surjective for $n\geq 3$. The first vertical map is also surjective, hence by the Five Lemma the middle vertical map is also surjective. We proved that
$$ev_{i}:  \mbox{H}^{0}(T_{i})\otimes \mbox{Hom}(T_{i},T_{i-1}) \longrightarrow  \mbox{H}^{0}(T_{i-1})$$
is surjective for $i\geq 4$, which is the first part of Theorem \ref{asymptotic}.

  Now we prove the second part. Assume $\mbox{Pic}(X)=\mathbb{Z}H$. First we consider $ev_{2}$.
  Consider the exact sequence
  $$ \mbox{H}^{0}(T_{2}) \longrightarrow  \mbox{H}^{0}(T_{1})\otimes \mbox{Ext}^{1}(T_{1}, S_{0}) \longrightarrow  \mbox{H}^{1}(S_{0}). $$
  For any $f\in  \mbox{H}^{0}(T_{1})$, consider the map
  $$ev_{1}(f\otimes -): \mbox{Ext}^{1}(T_{1},S_{0}) \longrightarrow  \mbox{H}^{1}(S_{0}).$$
By Corollary \ref{coarse}, $\mbox{ext}^{1}(T_{1}, S_{0})=q>\mbox{h}^{1}(S_{0})+2>\mbox{h}^{1}(S_{0})$, hence the map has a non-trivial kernel that contains some non-zero $g\in \mbox{Ext}^{1}(T_{1},S_{0})$. Hence there exists an
  $h\in  \mbox{H}^{0}(T_{2})$ that maps to $f\otimes g$ under the coevaluation map. Take $h\otimes g^{*}\in  \mbox{H}^{0}(T_{2})\otimes \mbox{Hom}(T_{2},T_{1})$, we have
  $$ev_{2}(h\otimes g^{*})=f \cdot \langle g, g^{*}\rangle=f.$$
  We proved the surjectivity of $ev_{2}$.

  Finally we consider $ev_{3}$. It suffices to show that 
  $$K_{3}\otimes V_{4} \overset{m_{3,4}}{\longrightarrow} K_{2}$$
   is surjective. Take any element
   $$\eta=\sum_{i_{1}, i_{2}} a_{i_{1}, i_{2}}v^{1}_{i_{1}}\otimes v^{2}_{i_{2}}\in K_{2},$$
   where $\{v^{j}_{i_{j}}\}_{1\leq i_{j}\leq \mbox{dim}V_{j}}$ is a basis of $V_{j}$ for $j=1,2$. 
  Consider
  $$\widetilde{\eta}=\eta\otimes \gamma\in K_{2}\otimes V_{3} $$
  for some non-zero $\gamma \in \mbox{Hom}(T_{3},T_{2})$. 
  Then $\widetilde{\eta}\in K_{3}$ if the following equations are satified for all $1\leq i_{1} \leq \mbox{h}^{1}(T_{1})$:
  $$\sum_{i_{2}}a_{i_{1},i_{2}}m_{2,3}(v^{2}_{i_{2}}\otimes \gamma)=0.$$
  The number of equations is $\mbox{h}^{0}(T_{1})$, but $\mbox{hom}(T_{3}, T_{2})=q> \mbox{h}^{0}(T_{1})$ by Corollary \ref{coarse}, hence there exists a non-zero solution $\gamma \in \mbox{Hom}(T_{3},T_{2})$. Consider
  $$\widetilde{\eta}\otimes \gamma^{*}\in K_{3}\otimes V_{4},$$
  where $\gamma^{*}\in \mbox{Hom}(T_{4}, T_{3})$ is the dual element of $\gamma$ under the identification $\mbox{Hom}(T_{4}, T_{3})=\mbox{Hom}(T_{3}, T_{2})^{*}$ via the perfect pairing $m_{3,4}$. Then we have $m_{3,4}(\tilde{\eta}\otimes \gamma^{*})=\eta$. The theorem is proved.

\end{proof}

\begin{remark}\label{badev1}
In Theorem \ref{asymptotic}, we said nothing about $ev_{1}$ (resp. $coev_{0}$). They are in general not surjective (resp. injective). We will exhibit one such in Example \ref{badev1ex}. 
\end{remark}

\section{Examples}\label{section9} 

In this section we show many explicit examples. For simplicity and without affecting the gist of the idea, in the following we let $(X,H)$ be a polarized K3 surface with $\mbox{Pic}(X)=\mathbb{Z}H$ and $H^{2}=2$, unless otherwise stated.

First, we consider the following example where there are no walls between the Gieseker chamber and the Brill-Noether wall.

\begin{example}[Fibonacci bundles, \cite{CNY21} Example 6.3]\label{height one}
  Let $s=(1,H,2)$, $t=(-1,0,-1)$. Let $S_{0}\in M_{H}(s)$, $T_{1}=\mathcal{O}_{X}[1]$. Let $T_{i}, S_{j}$ be constructed as in Proposition \ref{construction}. By Proposition \ref{criterion}, $S_{j}$ are $H$-Gieseker stable vector bundles. By Theorem \ref{numbers}, we have
  \begin{gather*}
    \mbox{h}^{0}(S_{j})= f_{-2j+3},~ \mbox{h}^{1}(S_{j})=f_{-2j-1}.
  \end{gather*}
  Here $f_{n}$ are the Fibonacci numbers:
  $$f_{0}=f_{1}=1,~ f_{n+1}=f_{n}+f_{n-1} \mbox{ for }n\geq 1.$$
  Explicitly, $v(S_{j})=(f_{-2j}, f_{-2j+1}H, f_{-2j+2}), j\leq 0$. The vector bundles $S_{j}$ are called the \emph{Fibonacci bundles}. 
\end{example}

Next, we consider a slightly complicated example: a height 2 (Definition \ref{ht}) spherical vector bundle.

\begin{example}[A height 2 bundle]\label{height two}
  Let $v=(305, 477H, 746)$ and $E\in M_{H}(v)$. Then the next wall (see Section \ref{section3}) $W$ for $E$ has stable spherical objects
  $$S_{0}\in M_{H}((2, 3H, 5)), ~T_{1}\in M_{H}((-5,12H,-29)). $$
  Using the notation of Proposition \ref{construction}, $E=S_{-1}$. By Definition \ref{ht}, $S_{0}$ and $T_{1}$ are of height one, hence $E$ has height two.

  We have $g(W)=-\chi(S_{0}, T_{1})=155$ (Definition \ref{fundamental sequence}). By Corollary \ref{JH}, the Harder-Narasimhan filtration of $E$ under $W$ is
  $$0 \longrightarrow S_{0}^{155} \longrightarrow E \longrightarrow T_{1} \longrightarrow 0. $$
  By Theorem \ref{global simplification}, the connecting homomorphism $ \mbox{H}^{0}(T_{1}) \longrightarrow  \mbox{H}^{1}(S_{0}^{155})$ has maximal possible rank. By Example \ref{height one}, $\mbox{h}^{0}(S_{0})=8, \mbox{h}^{1}(S_{0})=1$. By a similar argument as Example \ref{height one}, $ \mbox{h}^{0}(T_{1})=1$. Hence the connecting homomorphism is injective, we have
  $$\mbox{h}^{0}(E)=155\cdot \mbox{h}^{0}(S_{0})=1240, ~\mbox{h}^{1}(E)=\mbox{h}^{0}(E)-\chi(E)= 189. $$
\end{example}

The spherical vector bundle in the next example has height greater than two, so we cannot use Theorem \ref{global simplification}. We recall the notion of \emph{shape} (Definition \ref{shape}): If
  $$0=E_{0}\subset E_{1}\subset \cdots \subset E_{m}=E$$
  is some filtration of $E$ such that $G_{i}=E_{i}/E_{i-1}\cong F_{i}^{\oplus n_{i}}$, then the shape of the filtration is
  $$(n_{1}F_{1}, n_{2}F_{2}, \cdots, n_{m}F_{m}). $$

\begin{example}[A height 3 bundle]\label{height three}
  Let $v=(1340641, 1733695H, 2241986)$ and $E\in M_{H}(v)$. One may check that $E$ has height 3, hence in the following we invoke the global reduction (Algorithm \ref{global reduction}).

  The next wall $W$ for $E$ has stable spherical objects
  $$S_{0}\in M_{H}((58, 75H, 97)), ~T_{1}\in M_{H}((-29, 70H, -169)), $$
  and $g(W)=-\chi(S_{0}, T_{1})=23115$. Using the notation of Proposition \ref{construction}, $E=S_{-1}$. The Harder-Narasimhan filtration of $E$ under $W$ is
  $$E^{\bullet}_{W}=(g(W)\cdot S_{0}, T_{1}). $$

  The next wall for $E^{\bullet}_{W}$ is an actual wall of $S_{0}$, we denote it by $W^{S}$. It has two stable spherical objects
  $$S_{0}^{S}\in M_{H}((1,H,2)), ~ T_{1}^{S}\in M_{H}((-5, 12H, -29)),$$
  and $g(W^{S})=-\chi(S_{0}^{S}, T_{1}^{S})=63$. Using the notation of Proposition \ref{construction}, $S_{0}=S_{-1}^{S}$. The Harder-Narasimhan filtration of $E$ under $W^{S}$ has shape
  $$E^{\bullet}_{W^{S}}=(g(W)g(W^{S})\cdot S_{0}^{S}, g(W)\cdot T_{1}^{S}, T_{1}). $$
 
  The next wall for $E^{\bullet}_{W^{S}}$ is an actual wall for $T_{1}$, but it is also the numerical wall defined by $T_{1}$ and $T_{1}^{S}$. We denote this wall by $W^{T}$. Then the two stable spherical objects of $W^{T}$ are
  $$S_{0}^{T}\in M_{H}((5, 2H, 1)), ~ T_{1}^{T}=\mathcal{O}_{X}[1]\in M_{H}((-1, 0, -1)), $$
  and $g(W^{T})=-\chi(S_{0}^{T}, T_{1}^{T})=6$. Under the notations in Proposition \ref{construction}, $T_{1}^{S}=T_{3}^{T}$ and $T_{1}=T_{4}^{T}$. We need to find the Harder-Narasimhan filtration $E^{\bullet}_{W^{T}}$ under $W^{T}$. Let $\sigma_{+}, \sigma_{-}$ be stability conditions above or below $W^{T}$ respectively, and let
  $$F^{\bullet}_{W^{T}}=(g(W)T_{1}^{S}, T_{1})= (g(W)T_{3}^{T}, T_{4}^{T}).$$
  The question is to find the $\sigma_{-}$-Harder-Narasimhan filtration of $F$. By Theorem \ref{local simplification}, we need to know which inequality in Lemma \ref{numerical} holds. By Theorem \ref{numbers}, we have
\begin{gather*}
  \mbox{hom}(S_{0}^{T}, T_{4}^{T})=a_{2}(W^{T})=35, \mbox{ext}^{1}(S_{0}^{T}, (T_{3}^{T})^{\oplus g(W)})=g(W)\cdot a_{3}(W^{T})=4715460>35.
\end{gather*}
Hence by Proposition \ref{injectivity}, $\mbox{Hom}(S_{0}^{T},T_{4}) \longrightarrow \mbox{Ext}^{1}(S_{0}^{T}, (T_{3}^{T})^{\oplus g(W)})$ is injective, the first $\sigma_{-}$-Harder-Narasimhan factor for $F$ has shape
$$F_{1}=g(W)\cdot \mbox{hom}(S_{0}^{T}, T_{3}^{T})\cdot S_{0}^{T}=g(W)a_{1}(W^{T})\cdot S_{0}^{T}. $$
By Theorem \ref{local simplification}, $Q_{1}=F/F_{1}$ is of type II (Proposition \ref{type2}). The $\sigma_{+}$-Harder-Narasimhan filtration of $Q_{1}$ is
$$(Q_{1})_{+}^{\bullet}=(g(W)\mbox{hom}(T_{3}^{T}, T_{1}^{T})\cdot T_{1}^{T}, T_{4}^{T})=(g(W)a_{2}(W^{T})\cdot T_{1}^{T}, T_{4}^{T}).$$
Applying $\mbox{Hom}(-,T_{1}^{T})$, we have
$$0 \longrightarrow \mbox{Hom}(T_{4}^{T}, T_{1}^{T}) \longrightarrow \mbox{Hom}(Q_{1},T_{1}^{T}) \longrightarrow \mbox{Hom}((T_{1}^{T})^{\oplus g(W)a_{2}(W^{T})},T_{1}^{T}) \overset{\delta}{\longrightarrow} \mbox{Ext}^{1}(T_{4}^{T},T_{1}^{T}). $$
By Proposition \ref{type2}, the connecting homomorphism $\delta$ has maximal possible rank. By Theorem \ref{numbers}, we have
\begin{gather*}
  \mbox{hom}(T_{4}^{T}, T_{1}^{T})=a_{3}(W^{T})=204, \mbox{ext}^{1}(T_{4}^{T},T_{1}^{T})=a_{1}(W^{T})=6, \\
  \mbox{Hom}((T_{1}^{T})^{\oplus g(W)a_{2}(W)},T_{1}^{T})=g(W)a_{2}(W^{T})=809025.
\end{gather*}
Hence the last $\sigma_{-}$-Harder-Narasimhan factor of $Q_{1}$ is
$$F/F_{3}=(a_{3}(W^{T})+g(W)a_{2}(W^{T})-a_{1}(W^{T}))\cdot T_{1}^{T}= 809223 \cdot T_{1}^{T}. $$
Therefore the $\sigma_{+}$-Harder-Narasimhan filtration of $F_{3}/F_{1}$ has shape
$$(F_{3}/F_{1})_{+}^{\bullet}=(6T_{1}^{T}, \mbox{hom}(S_{0}^{T},T_{4}^{T})\cdot S_{0}^{T})=(6T_{1}^{T}, 35S_{0}^{T}). $$
By Proposition \ref{type1}, the $\sigma_{-}$-Harder-Narasimhan filtration of $F_{3}/F_{1}$ consists of only one factor $(S_{-2}^{T})_{-}\in M_{\sigma_{-}}(s_{-2}^{T})$, where $s_{-2}^{T}=(169, 70H, 29)$. Hence the $\sigma_{-}$-Harder-Narasimhan filtration of $F$ has shape
$$F_{-}^{\bullet}=(g(W)a_{1}(W^{T})\cdot S_{0}, (S_{-2}^{T})_{-}, 809223\cdot T_{1}^{T}).$$

The $\sigma_{-}$-Harder-Narasimhan filtration of $E$ has shape
$$E_{-}^{\bullet}=(g(W)g(W^{S})\cdot S_{0}^{S}, g(W)a_{1}(W^{T})\cdot S_{0}^{T}, (S_{-2}^{T})_{-}, 809223\cdot \mathcal{O}_{X}[1]).$$
The next wall for $E_{-}^{\bullet}$ is below the Brill-Noether wall of any factor, hence the global reduction (Algorithm \ref{global reduction}) terminates. By Corollary \ref{output}, we have
$$\mbox{h}^{1}(E)=809223, \mbox{h}^{0}(E)=\chi(E)+\mbox{h}^{1}(E)=4391850.$$

\end{example}

As mentioned in Section \ref{section6}, there are indeed spherical vector bundles whose cohomology cannot be computed by Theorem \ref{local simplification}. In the next example we exhibit one such.

\begin{example}[Theorem \ref{local simplification} fails]\label{3step}
  Let $v=(42687466,66760513H,104409245)$ and $E\in M_{H}(v)$. We run the global reduction (Algorithm \ref{global reduction}).

  The next wall $W$ for $E$ has stable spherical objects
  $$S_{0}\in M_{H}((305, 477H, 746)), ~T_{1}\in M_{H}((-29, 70H, -169)), $$
  and $g(W)=-\chi(S_{0}, T_{1})=139959$. Using the notation of Proposition \ref{construction}, $E=S_{-1}$. The Harder-Narasimhan filtration of $E$ under $W$ has shape
  $$E^{\bullet}_{W}=(g(W)\cdot S_{0}, T_{1}). $$

  The next wall for $E^{\bullet}_{W}$ is an actual wall of $S_{0}$, we denote it by $W^{S}$. It has two stable spherical objects
  $$S_{0}^{S}\in M_{H}((2,3H,5)), ~ T_{1}^{S}\in M_{H}((-5, 12H, -29)),$$
  and $g(W^{S})=-\chi(S_{0}^{S}, T_{1}^{S})=155$. Using the notation of Proposition \ref{construction}, $S_{0}=S_{-1}^{S}$. The Harder-Narasimhan filtration of $E$ under $W^{S}$ has shape
  $$E^{\bullet}_{W^{S}}=(g(W)g(W^{S})\cdot S_{0}^{S}, g(W)\cdot T_{1}^{S}, T_{1}). $$

 The next wall for $E^{\bullet}_{W_{S}}$ is the Brill-Noether wall of $S_{0}^{S}$, we denote it by $W^{SS}$. It has two stable spherical objects
 $$S_{0}^{SS}\in M_{H}((1,H,2)), ~ T_{1}^{SS}=\mathcal{O}_{X}[1]\in M_{H}((-1,0,-1)), $$
 and $g(W^{SS})=-\chi(S_{0}^{SS}, T_{1}^{SS})=3$. Using the notation of Proposition \ref{construction}, $S_{0}^{S}=S_{-1}^{SS}$. By Example \ref{height one}, the Harder-Narasimhan filtration of $E$ under $W^{SS}$ has shape
$$E^{\bullet}_{W^{SS}}=(g(W)g(W^{S})g(W^{SS})\cdot S_{0}^{SS}, g(W)g(W^{S})\cdot T_{1}^{SS}, g(W)\cdot T_{1}^{S}, T_{1}).$$

  The next wall for $E^{\bullet}_{W^{SS}}$ is an actual wall for $T_{1}$, but it is also the numerical wall defined by $T_{1}$ and $T_{1}^{S}$. Furthermore, $T_{1}^{SS}=\mathcal{O}_{X}[1]$ is also on the wall. We denote this wall by $W^{T}$. Then the two stable spherical objects of $W^{T}$ are
  $$S_{0}^{T}\in M_{H}((5, 2H, 1)), ~ T_{1}^{T}=\mathcal{O}_{X}[1]\in M_{H}((-1, 0, -1)), $$
  and $g(W^{T})=-\chi(S_{0}^{T}, T_{1}^{T})=6$. Using the notation of Proposition \ref{construction}, $T_{1}^{SS}=T_{1}^{T}$, $T_{1}^{S}=T_{3}^{T}$, and $T_{1}=T_{4}^{T}$. The filtration
  $$Q_{1}^{\bullet}=(g(W)g(W^{S})\cdot T_{1}^{T}, g(W)\cdot T_{3}^{T}, T_{4}^{T})$$
  is not two-step (Proposition \ref{injectivity}), hence Theorem \ref{local simplification} cannot be applied.

\end{example}

The next two examples use Theorem \ref{weak BN} to test weak Brill-Noether.

\begin{example}[weak Brill-Noether fails]\label{wBNfail}
Let $v=(10, 13H, 17)$. Under the notations in Theorem \ref{weak BN}, one checks that $v_{1}=(1,H,2)$ satisfies the conditions. In this case, $\frac{a_{1}d-ad_{1}}{r_{1}d-rd_{1}}=3 $, hence $y\geq 3>1$. By Theorem \ref{weak BN}, weak Brill-Noether fails for $E\in M_{H}(v)$. In fact, we have $\mbox{h}^{0}(E)=33$ and $\mbox{h}^{1}(E)=6$.
\end{example}

\begin{example}[weak Brill-Noether holds]\label{wBNhold}
Let $v=(195562, 59615H , 18173)$. Under the notations in Theorem \ref{weak BN}, one may check that those $v_{1}$ satisfying the conditions are $(10, 3H, 1)$, $(269, 82H, 1)$, and $(7253, 2211H, 674)$. The corresponding $y$ values are all $\frac{7}{13} $. Since $\frac{7}{13}<1 $, by Theorem \ref{weak BN}, weak Brill-Noether holds for $E\in M_{H}(v)$: we have $\mbox{h}^{0}(E)=\chi(E)=213735$ and $\mbox{h}^{1}(E)=0$.
\end{example}

Note that the relation of weak Brill-Noether and the complexity of the Mukai vector is delicate. Weak Brill-Noether can fail for a small Mukai vector such as Example \ref{wBNfail}, or even some smaller ones such as $(2, 3H, 5)$. Weak Brill-Noether can also hold for a large Mukai vector, such as Example \ref{wBNhold}.

As mentioned in Remark \ref{badev1}, next we give an example where $ev_{1}$ fails to be surjective.

\begin{example}[Non-surjective $ev_{1}$]\label{badev1ex} 
   Let $v=(-37666, 22095H, -12961)$ and $E\in M_{H}(v)$. Then the next wall $W$ for $E$ has stable spherical objects
  $$S_{0}\in M_{H}((12,5H,29)), ~T_{1}\in M_{H}((-29, 17H, -10)),$$
  and $g(W)=-\chi(S_{0}, T_{1})=1299$. Under the notations in Proposition \ref{construction}, $E=T_{2}$. The Harder-Narasimhan filtration of $E$ under $W$ is
  $$0 \longrightarrow S_{0} \longrightarrow E \longrightarrow T_{1}\otimes \mbox{Ext}^{1}(T_{1},S_{0}) \longrightarrow 0 .$$
  Under the notations in Theorem \ref{asymptotic}, $ev_{1}$ is the connecting homomorphsim
  $$ev_{1}:  \mbox{H}^{0}(T_{1})\otimes \mbox{Ext}^{1}(T_{1}, S_{0}) \longrightarrow  \mbox{H}^{1}(S_{0}).$$
  The height of $S_{0}$ and $T_{1}$ are 1 and 2, respectively. By similar arguments as Example \ref{height one} and Example \ref{height two}, we have
  $$\mbox{h}^{0}(T_{1})=6, ~\mbox{h}^{1}(S_{0})=1.$$
  Hence the maximal possible rank of $ev_{1}$ is 1. However, we will show that $ev_{1}$ is in fact zero, which is equivalent to the adjoint map $f: \mbox{Ext}^{1}(T_{1}, S_{0}) \longrightarrow  \mbox{H}^{0}(T_{1})^{*} \otimes \mbox{H}^{1}(S_{0})$ being zero. 

  The next wall $W^{T}$ for $T_{1}$ has stable spherical objects
  $$S_{0}^{T}\in M_{H}((1, 2H, 5)), ~T_{1}^{T}\in M_{H}((-2, H, -1)),$$
  and the next wall for $S_{0}$ has stable spherical objects
  $$S_{0}^{S}\in M_{H}((1, 2H, 5)), ~T_{i}^{S}=\mathcal{O}_{X}[1].$$
  Hence $S_{0}^{T}=S_{0}^{S}$. Under the notations in Proposition \ref{construction}, $S_{0}=S_{-1}^{S}$. Hence by Theorem \ref{numbers}, we have
  $$\mbox{Ext}^{1}(S_{0}^{T}, S_{0})=\mbox{Ext}^{1}(S_{0}^{S}, S_{-1}^{S})=0. $$
  Also note that $ \mbox{H}^{0}(T_{1}^{T})=0$ by a direct computation.
  
  By Corollary \ref{JH}, the map $f$ fits into a commutative diagram
  \\
\begin{adjustbox}{scale=0.9, center}
  \begin{tikzcd}
      \mbox{Hom}(T_{1}, T_{1}^{T})\otimes \mbox{Ext}^{1}(T_{1}^{T}, S_{0}) \arrow[r]\arrow[d] & \mbox{Ext}^{1}(T_{1}, S_{0}) \arrow[r]\arrow[d, "f"] & \mbox{Hom}(S_{0}^{T}, T_{1})^{*}\otimes \mbox{Ext}^{1}(S_{0}^{T}, S_{0})=0 \arrow[d] \\
      0=\mbox{Hom}(T_{1}, T_{1}^{T})\otimes  \mbox{H}^{0}(T_{1}^{T})^{*}\otimes  \mbox{H}^{1}(S_{0}) \arrow[r] &  \mbox{H}^{0}(T_{1})^{*}\otimes  \mbox{H}^{1}(S_{0}) \arrow[r] & \mbox{Hom}(S_{0}^{T}, T_{1})^{*}\otimes  \mbox{H}^{0}(S_{0}^{T})\otimes  \mbox{H}^{1}( S_{0}).
    \end{tikzcd}
  \end{adjustbox}
  Hence $f=0$, $ev_{1}$ is not surjective.
\end{example}

The following example shows that in Proposition \ref{lattice}, item (2) can happen.

\begin{example}[Negative definite rank 2 lattice]\label{ndlattice}
  Let $(X,H)$ be a K3 surface that contains a line $L$. Then $L$ has no deformation, $\mbox{dim}|L|=0$. Consider the following short exact sequence in the heart $\mathcal{A}_{\mathbf{b}}$ (Section \ref{section3}):
  $$0 \longrightarrow \mathcal{O}_{X}(L) \longrightarrow \mathcal{O}_{L}(-2)  \longrightarrow \mathcal{O}_{X}[1]  \longrightarrow 0. $$
  By Proposition \ref{criterion}, the sequence defines an actual wall for $\mathcal{O}_{L}(-2)$, whose stable spherical objects are $\mathcal{O}_{X}(L)$ and $\mathcal{O}_{X}[1]$. We have $\chi(\mathcal{O}_{X}(L), \mathcal{O}_{X}[1])=-1$, hence the associated lattice is negative definite.
\end{example}

\nocite{*}
\section{References}

\bibliographystyle{alpha}
\renewcommand{\section}[2]{} 
\bibliography{spherical}
\end{document}